\documentclass[12pt,reqno]{amsart}
\usepackage[margin=1in]{geometry}
\usepackage{amsmath,amssymb,amsthm,graphicx,amsxtra, setspace}
\usepackage[utf8]{inputenc}
\usepackage{mathrsfs}
\usepackage{hyperref}
\usepackage{upgreek}
\usepackage{mathtools}
\usepackage{xcolor}
\allowdisplaybreaks

\usepackage[pagewise]{lineno}

\usepackage{graphicx,eurosym}
\usepackage{hyperref}
\usepackage{mathtools}

\usepackage[cyr]{aeguill}

\colorlet{darkblue}{blue!50!black}

\hypersetup{
	colorlinks,%
	citecolor=blue,%
	filecolor=red,%
	linkcolor=darkblue,%
	urlcolor=blue,%
	pdfnewwindow=true,%
	pdfstartview={FitH}
}

\colorlet{darkblue}{red!100!black}

\newtheorem{theorem}{Theorem}[section]
\newtheorem{lemma}[theorem]{Lemma}
\newtheorem{proposition}[theorem]{Proposition}

\newtheorem{definition}[theorem]{Definition}

\newtheorem{remark}[theorem]{Remark}
\newtheorem{hypothesis}[theorem]{Hypothesis}

\let\emptyset\varnothing

\let\originalleft\left
\let\originalright\right
\renewcommand{\left}{\mathopen{}\mathclose\bgroup\originalleft}
\renewcommand{\right}{\aftergroup\egroup\originalright}

\theoremstyle{definition}

\newtheorem{condition}{Condition}[section]

\def\1{\mathcal{O}}

\def\wi{\widehat}
\def\vi{\widetilde}

\def\h{\boldsymbol{h}}
\def\d{\mathrm{d}}
\def\bfX{\mathbf{X}}

\def\I{\mathrm{I}}
\def\B{\mathrm{B}}
\def\D{\mathrm{D}}
\def\A{\mathrm{A}}

\def\W{\mathrm{W}}
\def\R{\mathbb{R}}
\def\E{\mathbb{E}}

\def\H{\mathbb{H}}
\def\V{\mathbb{V}}
\def\e{\epsilon}
\def\2{\mathcal{E}}
\def\L{\mathrm{L}}
\def\q{\boldsymbol{q}}
\def\u{\boldsymbol{u}}
\def\v{\boldsymbol{v}}

\def\n{\boldsymbol{n}}

\def\p{\boldsymbol{p}}
\def\x{\boldsymbol{x}}
\def\y{\boldsymbol{y}}
\def\z{\boldsymbol{z}}
\def\bfe{\boldsymbol{e}}
\def\bf{\boldsymbol}
\def\mbf{\mathbf}

\def\C{{\mathrm{C}}}

\def\Y{\mathbf{Y}} 
\def\F{\mathrm{F}}
\def\P{\mathbb{P}}
\def\N{\mathbb{N}}
\def\G{\mathrm{G}}
\def\bfZ{\mathbf{Z}}
\def\Z{\mathrm{Z}}

\def\mb{\mathbb}

\def\fk{\mathfrak}
\def\cc{\mathfrak{C}}
\def\bb{\mathfrak{B}}

\newcommand{\Addresses}{{
		\footnote{
			\noindent \textsuperscript{1,2}Department of Mathematics, Indian Institute of Technology Roorkee-IIT Roorkee,
			Haridwar Highway, Roorkee, Uttarakhand 247667, INDIA.\par\nopagebreak
			\noindent  \textit{e-mail:} \texttt{Manil T. Mohan: maniltmohan@ma.iitr.ac.in, maniltmohan@gmail.com.}
			
			\textit{e-mail:} \texttt{Ankit Kumar: akumar14@mt.iitr.ac.in.}
			
			\noindent \textsuperscript{*}Corresponding author.
			
			\textit{Key words:} Stochastic partial differential equations, locally monotne, pseudo-monotone,  L\'evy noise, Wentzell-Freidlin large deviation, weak convergence.
			
			Mathematics Subject Classification (2020): Primary 60H15, 60F10; Secondary 35R60, 35Q35, 37L55.

}}}
\begin{document}	
	
	\title[LDP for a class of SPDEs with L\'evy noise]{Large deviation principle for a class of stochastic partial differential equations with fully local monotone coefficients perturbed by L\'evy noise
		\Addresses}

	\author[A. Kumar and M. T. Mohan]
	{Ankit Kumar\textsuperscript{1} and Manil T. Mohan\textsuperscript{2*}}

	\maketitle
	
\begin{abstract}
The asymptotic analysis of a class of stochastic partial differential equations (SPDEs)  with fully locally monotone coefficients covering  a large variety of physical systems, a wide class of quasilinear SPDEs and  a good number of fluid dynamic models is carried out in this work. The aim of this work is to develop the large deviation theory for small  Gaussian as well as Poisson noise perturbations of the above class of SPDEs.  We establish a Wentzell-Freidlin type large deviation principle for the strong solutions to such SPDEs  perturbed by L\'evy noise in a suitable Polish space using  a variational representation (based on a weak convergence approach) for nonnegative functionals of general Poisson random measures  and Brownian motions. The well-posedness of an associated deterministic control problem is established by exploiting pseudo-monotonicity arguments and the stochastic counterpart is obtained by an application of Girsanov's theorem. 
\end{abstract}
\section{Introduction}\label{Sec1}\setcounter{equation}{0}
Let us denote a separable Hilbert space and a reflexive Banach space by $\H$ and $\V$, respectively, and let  the embedding $\V\subset \H$ be continuous and dense. Denote the dual space of $\V$ by $\V^*$ and let $\H^*\cong \H$. The norms of $\H,\V$ and $\V^*$ are denoted by $\|\cdot\|_\H,\|\cdot\|_\V$ and $\|\cdot\|_{\V^*}$, respectively, and we have a Gelfand triplet $\V\subset \H\subset \V^*$. We  represent  $(\cdot,\cdot),$ the inner product in the Hilbert space $\H$ and $\langle\cdot,\cdot\rangle,$ the duality pairing between $\V$ and $\V^*$. Also, $\langle \u,\v\rangle =(\u,\v)$, whenever $\u\in\H$ and $\v\in\V$. The space of all Hilbert-Schmidt operators from $\H$ to itself is denoted by $\L_2(\H,\H)$ (for convenience, we use $\L_2(\H)$) with the norm $\|\cdot\|_{\L_2}$and the inner product $(\cdot,\cdot)_{\L_2}$. 

Let $(\Omega,\mathscr{F},\P)$ be a complete probability space equipped with an increasing family of sub-sigma fields $\{\mathscr{F}_t\}_{t\geq0}$ of $\mathscr{F}$ satisfying:
\begin{itemize}
	\item[(i)] $\mathscr{F}_0$ contains all elements $A \in\mathscr{F}$ with $\P(A)=0$.
	\item[(ii)] $\mathscr{F}_t=\mathscr{F}_{t+}=\cap_{s>t}\mathscr{F}_s$, for $t\geq 0$.
\end{itemize}  	In this work, we establish a Wentzell-Freidlin type  large deviation principle (LDP) for the following class of stochastic partial differential equations (SPDEs) with fully locally monotone coefficients in the Gelfand triplet $\V\subset\H\subset\V^*$ driven by L\'evy noise:
\begin{equation}\label{1.1}
	\left\{
	\begin{aligned}
		\mathrm{d} \mathbf{Y}^\epsilon(t)&=\mathrm{A}(t,\mathbf{Y}^\epsilon(t))\mathrm{d} t+\sqrt{\epsilon}\mathrm{B}(t,\mathbf{Y}^\epsilon(t))\mathrm{d} \mathrm{W}(t)+\epsilon\int_\mathrm{Z}\gamma(t,\mathbf{Y}^\epsilon(t-),z)\widetilde{N}^{\epsilon^{-1}}(\mathrm{d} t,\mathrm{d} z),\\
		\mathbf{Y}^\epsilon(0)&=\boldsymbol{x}\in\mathbb{H},
	\end{aligned}
	\right.\end{equation} for a.e. $t\in[0,T]$, where the mappings 
\begin{align*}
	\mathrm{A}:[0,T]\times \mathbb{V}\to\mathbb{V}^*,\; \mathrm{B}:[0,T]\times \mathbb{V}\to\mathrm{L}_2(\mathbb{H},\mathbb{H}), \text{ and } \gamma:[0,T]\times\mathbb{V}\times\mathrm{Z}\to\mathbb{H},
\end{align*}are measurable,  $\mathrm{Z}$ is a locally compact Polish space, $\mathrm{W}(\cdot)$ is an $\mathbb{H}$-valued cylindrical Wiener process on the filtered probability space $(\Omega,\mathscr{F},\{\mathscr{F}_t\}_{t\geq 0},\P)$  and  $\widetilde{N}^{\epsilon^{-1}}$ is a  Poisson random measure on $[0,T]\times\mathrm{Z}$ with a $\sigma$-finite intensity measure $\epsilon^{-1}\lambda_T\otimes\nu$, $\lambda_T$ is the Lebesgue measure on $[0,T]$ and $\nu$ is a $\sigma$-finite measure on $\mathrm{Z}$. Moreover,  $\widetilde{N}^{\epsilon^{-1}}([0,t]\times O)=N^{\epsilon^{-1}}([0,t]\times O)-\epsilon^{-1}t\nu(O)$, for all $O\in\mathcal{B}(\mathrm{Z})$ with $\nu(O)<\infty$, is the compensated Poisson random measure. 
The well-posedness of the problem \eqref{1.1} with $\e=0$ is discussed in \cite{WL4} and  the global solvability of the system \eqref{1.1} perturbed by Gaussian and L\'evy noises are considered in \cite{MRSSTZ,AKMTM4}, respectively.

\subsection{Hypothesis and solvability results} Let us first discuss the assumptions satisfied by the mappings  and recall the solvability results for the system \eqref{1.1}  from \cite{AKMTM4}.
The mappings $\A$, $\B$ and $\gamma$ satisfy the following Hypothesis: 
\begin{hypothesis}\label{hyp1}
	Let $\fk{a}\in\L^1(0,T;\R^+)$ and $\beta\in(1,\infty)$.
	\begin{itemize}
		\item[(H.1)] \textsf{(Hemicontinuity)}. The map $\R\ni\lambda \mapsto \langle \A(t,\x+\lambda \y),\z\rangle \in\R$ is continuous for any $\x,\y,\z \in \V$ and for a.e. $t\in[0,T]$.
		\item[(H.2)] \textsf{(Local monotonicity)}.  There exist  non negative constants $\zeta$ and $C$ such that for any $\x,\y\in\V$ and a.e. $t\in[0,T]$, 
		\begin{align}\label{3.5}\nonumber
			2\langle \A(t,\x)-\A(t,\y),\x-\y\rangle +&\|\B(t,\x)-\B(t,\y)\|_{\L_2}^2+\int_{\Z}\|\gamma(t,\x,z)-\gamma(t,\y,z)\|_{\H}^2\lambda(\d z) 		\\& \leq \big[\fk{a}(t)+\rho(\x)+\eta(\y)\big]\|\x-\y\|_{\H}^2, \\ \nonumber
			|\rho(\x)|+|\eta(\x)| &\leq C(1+\|\x\|_{\V}^\beta)(1+\|\x\|_{\H}^\zeta), 
		\end{align}where $\rho$ and $\eta$ are two measurable functions from $\V$ to $\R$.
		\item[(H.3)]	\textsf{(Coercivity)}. There exists a positive constant $C$ such that for any $\x\in\V$ and a.e.  $t\in[0,T]$, 
		\begin{align}\label{3.7}
			\langle \A(t,\x),\x\rangle 
			\leq \fk{a}(t)(1+\|\x\|_{\H}^2)-L_\A\|\x\|_{\V}^\beta.
		\end{align}
		\item[(H.4)] \textsf{(Growth)}. There exist non-negative constants  $\alpha$ and $C$ such that for any $\x\in\V$ and a.e. $t\in[0,T]$,
		\begin{align}\label{3.8}
			\|\A(t,\x)\|_{\V^*}^{\frac{\beta}{\beta-1}}\leq (\fk{a}(t)+C\|\x\|_{\V}^\beta)(1+\|\x\|_{\H}^\alpha).
		\end{align} 
		\item[(H.5)] For any sequence $\{\x_n\}_{n\in\mathbb{N}}$ and $\x$ in $\V$ with $\|\x_n-\x\|_\H\to0$, as $n\to\infty$, we have 
		\begin{align}\label{3.9}
			\|\B(t,\x_n)-\B(t,\x)\|_{\L_2}\to 0, \text{ for a.e. }t\in[0,T]. 
		\end{align}
		Moreover, there exists $\fk{b}\in \L^1(0,T;\R^+)$  such that for any $\x\in \V$ and a.e. $t\in[0,T]$, 
		\begin{align}\label{3.10}
			\|\B(t,\x)\|_{\L_2}^2\leq \fk{b}(t)(1+\|\x\|_{\H}^2). 
		\end{align}
		\item[(H.6)]  The jump noise coefficient $\gamma(\cdot,\cdot,\cdot)$ satisfies: \begin{enumerate}
			\item $\E\bigg[\int_0^T\int_\Z\|\zeta(t,\cdot,z)\|_\H^2\nu(\d z)\d t\bigg]<\infty.$
			\item For any sequence $\{\x_n\}_{n\in\mathbb{N}}$ and $\x$ in $\V$ with $\|\x_n-\x\|_\H\to0$, as $n\to\infty$, we have 
			\begin{align}\label{3.010}
				\int_\Z \|\gamma(t,\x_n,z)-\gamma(t,\x,z)\|_{\H}^2\lambda(\d z) 	\to 0, \text{ for a.e. }t\in[0,T]. 
			\end{align}
			\item There exist functions $\fk{h}_p\in \L^1(0,T;\R^+)$  such that for any $\x\in \V$, a.e. $t\in[0,T]$, 
			\begin{align}\label{3.11}
				\int_\Z \|\gamma(t,\x,z)\|_{\H}^p\lambda(\d z) \leq \fk{h}_p(t)(1+\|\x\|_{\H}^p), \text{ for } p\in[2,\infty).
			\end{align}
		\end{enumerate}
	\end{itemize}
\end{hypothesis}
Let us now introduce the concept  of solution and its uniqueness to the system \eqref{1.1}. We   denote $ \D([0,T];\H)$ for  the space of all c\`adl\`ag functions (right continuous with left limits) from $[0,T]$ to $\H$. 
\begin{definition}
	Let $(\Omega,\mathscr{F},\{\mathscr{F}_t\}_{t\geq 0},\P)$ be a given stochastic basis and the initial data $\x\in\H$. Then \eqref{1.1} has a \textsf{probabilistically strong solution} if and only if there exists a progressively measurable process $\Y^\e:[0,T]\times \Omega\to\H$, $\P$-a.s., with paths
		\begin{align*}\Y^\e(\cdot,\cdot)\in \D([0,T];\H)\cap \L^\beta(0,T;\V),\;\P\text{-a.s., for } \beta\in(1,\infty),
		\end{align*}
	and  the following equality holds $\P$-a.s., in $\V^*$,
		\begin{align*}
			\Y^\e(t)&=\x+\int_0^t\A(s,\Y^\e(s))\d s+\sqrt{\e}\int_0^t \B(s,\Y^\e(s))\d\W(s)\\&\quad+\e\int_0^t\int_\Z\gamma(s,\Y^\e(s-),z)\vi{N}^{\e^{-1}}(\d s,\d z), \; t\in[0,T].
		\end{align*}

\end{definition}
\begin{definition}
	We say that the system \eqref{1.1} admits a \textsf{pathwise unique solution} if any two $\H$-valued solutions $\Y_1(\cdot)$ and $\Y_2(\cdot)$ defined on the same probability space with respect to the same Poisson random measure and Wiener process starting from the same initial data $\x\in\H$ coincide almost surely.
\end{definition}
With a minor modification in Theorem 2.11, \cite{AKMTM4}, one can establish the following result on the existence of a unique pathwise strong solution  to the system \eqref{1.1}.
\begin{theorem}\label{ER1}
	Assume that Hypothesis \ref{hyp1} (H.1)-(H.6) hold and the embedding $\V\subset \H$ is compact. Then, for any initial data $\x\in\H$, there exists a \textsf{probabilistically strong (analytically weak) solution} to the system \eqref{1.1}. Furthermore, for any $p\geq 2$, the following estimate holds:
	\begin{equation}\label{4.8}
		\E\bigg[\sup_{t\in[0,T]}\|\Y(t)\|_\H^p\bigg]+\E\bigg[\int_0^T\|\Y(t)\|_\V^\beta\d t\bigg]^{\frac{p}{2}} <\infty.
	\end{equation}
\end{theorem} 

  \subsection{Literature survey}
   The existence and uniqueness of solutions for the deterministic problem corresponding to the system \eqref{1.1} has been established in \cite{WL4}, and the stochastic counterpart has been discussed in the works   \cite{MRSSTZ,AKMTM4}, where the authors considered the system \eqref{1.1} perturbed by multiplicative Gaussian and L\'evy noises, respectively.  The global solvability of   time fractional SPDEs with fully locally monotone coefficients driven by additive Gaussian noise is studied in \cite{WLRMJLDS1,WLRMJLDS2}.   
  Several authors proved the well-posedness of  SPDEs with locally monotone coefficients, cf. \cite{ZBWLJZ,WHSS,WL,WLMR1,WLMR2,WLMR3,CPMR}, etc.   Well-posedness of a class of SPDEs driven by L\'evy noise with generalized coercivity condition, and with  locally monotone coefficients is obtained in \cite{ZBWLJZ}, and with monotone coefficients is established in \cite{TKMR}. Solvability results for different types of SPDEs driven  by L\'evy noise have been established in \cite{ZBDG,ZBXPJZ,ZDRZ,MTM2,MTM4,MTMSSS2,ZTHWYW,JXJZ}, etc., and references therein.
  
  The theory of large deviations,  which provides asymptotic estimates for probabilities of rare events is one of the most classical areas  and important research topics in probability theory and has rightly received attention after the framework and applications provided by Stroock and Varadhan in  \cite{ST,VA}, respectively, and extended its applications to different areas (cf.  \cite{DZ,DE,ST}, etc). The authors in  \cite{PLC,KX,SW}, etc., proved Wentzell-Freidlin type LDP for different classes of SPDEs. The weak convergence approach (cf. \cite{ABPDVM}) has been explored in the works \cite{SCMR,WL2,ICAM1}  to obtain  large deviations principle for the small noise limit of the systems of stochastic reaction-diffusion equations with globally Lipschitz but unbounded coefficients, stochastic evolution equations with general monotone drift and  for a class of stochastic 2D hydrodynamical type systems,  respectively, driven by multiplicative Gaussian noise. 
  
  From the last few decades, several authors have established the LDP for different type of SPDEs perturbed by pure jump  and L\'evy noises. In the works \cite{ADA1,ADA2}, the authors obtained the LDP for  the solutions of stochastic differential equations (SDEs) perturbed by Poisson random  measures. The authors in \cite{ABJCPD,ABPFD} used  the results due to Varadhan and Bryc \cite{DZ}, and extended the earlier results of Budhiraja and Dupis \cite{ABPD1} to establish the LDP for SDEs with Poisson noises by first proving the Laplace principle in Polish spaces using the weak convergence approach. The large deviation theory for small Poisson noise perturbations of a general class of deterministic infinite dimensional models is developed in  \cite{ABPDVM}. The authors in \cite{MRTZ} established an LDP for a class of SPDEs perturbed by an addditive jump noise. The LDP  for the solutions of an abstract stochastic evolution equations driven by small L\'evy noise is obtained in \cite{ASJZ}. Using the methodology developed in \cite{ADA1,ADA2}, the authors in \cite{TXTZ} proved the LDP for 2D Navier-Stokes equations perturbed by additive L\'evy noise. Using the results established in \cite{ABPDVM}, the LDP for two dimensional Navier-Stokes equations governed by L\'evy noise is obtained in \cite{ZBXPJZ,JZTZ},  and for a class of SPDEs with locally monotone coefficients driven by pure jump noise is established in \cite{JXJZ}. Using the same weak convergence approach (cf. \cite{ABPDVM}), several authors established the LDP for many physically relevant models perturbed by L\'evy noise, cf. \cite{HBAM,ZDRZ,UMMTM,MTM3}, etc.   
    
\subsection{Novelties, difficulties and approaches} The main goal of this work is to develop the large deviation theory for small  Gaussian as well as Poisson noise perturbations of a class of SPDEs which covers a wide class of physical models. Very recently, the authors  in  \cite{MRSSTZ} established well-posedness results for a class of SPDEs with fully local monotone coefficients perturbed by multiplicative Gaussian noise under various assumptions on the noise coefficient. Then, the authors in \cite{AKMTM4} extended the theory developed in \cite{MRSSTZ} from multiplicative Gaussian  to L\'evy noise and showed the existence of a uniue probabilistically strong solution in the space $\D([0,T];\H)\cap\L^\beta(0,T;\V)$, for $\beta\in(1,\infty)$.  Even though SPDEs with locally monotone coefficients are well-studied (cf. \cite{ZBWLJZ,WHSS,WHSL,KKMTM,SLWLYX,WL,WLMR1,WLMR2,WLMR3,WLRMJLDS2,WLMRXSYX,TMRZ,CPMR,JXJZ}, etc., and references therein),  the systems having fully local monotone coefficients  are not much explored in the literature,  since the well-posedness results were not known previously. 

We establish a Wentzell-Freidlin type large deviation principle for the strong solutions to the system \eqref{1.1}  in the state space $\D([0,T];\H) $ using  a variational representation (based on a weak convergence approach). We  exploit the weak convergence method developed by Budhiraja et.al.  in \cite{ABPD1,ABJCPD,ABPDVM},  etc., where the authors established the LDP  using  a variational representation for nonnegative functionals of general Poisson random measures and Brownian motions  (see Condition \ref{Cond1} and Theorem \ref{thrm5} below).  In order to establish the LDP for the system \eqref{1.1}, we follow the works \cite{JXJZ,JZTZ}, where the authors established the LDP, for a class of SPDEs with locally monotone coefficients driven by pure jump noise, 2D stochastic Navier-Stokes equations perturbed by L\'evy noise, respectively. The well-posedness of the  associated deterministic control problem (see \eqref{SE1} below) is established by exploiting pseudo-monotonicity arguments and the stochastic control problem (see \eqref{CSPDE1} below) is obtained by an application of Girsanov's theorem. Moreover, Skorokhod's representation theorem is used to verify Condition \ref{Cond1} (2), which is a law of large numbers result for stochastic controlled  systems with  small noise. The LDP for the system \eqref{1.1} is established under some additional assumptions on the noise coefficients (see Hypothesis \ref{hyp2} below),  while the same problem is open under Hypothesis \ref{hyp1}. 
   \subsection{Organization of the paper} 
  This article is organized as follows: In Section \ref{Sec2}, we provide some  basic results related to pseudo-monotone operators (Lemmas \ref{lemmaMT2} and \ref{Cgslem00})  and a general criteria of LDP along with a sufficient condition (see Condition \ref{Cond1}) for LDP. In  Section \ref{Sec3}, we recall some useful results from \cite{JXJZ,JZTZ} which  help us to obtain LDP for the SPDE under our consideration. We need some additional assumptions on the noise coefficients to obtain LDP for the system \eqref{1.1} and we collect them in Hypothesis \ref{hyp2}. We state  our main result as Theorem \ref{thrm5}, whose proof rely on verifying Condition \ref{Cond1}. In order to verify Condition \ref{Cond1} (1), we first establish  the existence and uniqueness result of the deterministic control system \eqref{SE1} using a Faedo-Galerkin approximation technique and pseudo-monotonicity property of the nonlinear operator $\A(\cdot,\cdot)$ (see Theorem \ref{CPDE1}).  With the aid of Theorem \ref{CPDE1}, Condition \ref{Cond1} (1) is verified in Proposition \ref{VCond1}.  Section \ref{Sec5} is devoted for the verification of Condition \ref{Cond1} (2), which is the most difficult part of the work. Firstly, we consider the stochastic control problem  \eqref{CSPDE1} and discuss the existence and uniqueness of its pathwise strong solution by using Girsanov's theorem (Lemma \ref{lem5.3}). In order to prove the tightness of the laws of solutions of \eqref{CSPDE1}, we first obtain the uniform bounds (independent of $\e$) in  Lemmas \ref{lemUF2} and  \ref{lemUF02}. The tightness of the laws of solutions in the space $\D([0,T];\V^*)$ with the Skorokhod topology is established in Lemma \ref{Cgslem1}. Convergence of the processes and proper identification of limits is  carried forward in Lemmas   \ref{Cgslem2}, \ref{Cgslem3}, \ref{Cgslem4} and \ref{Cgslem6}. The solvability results of the limiting process is discussed in Proposition \ref{propUF}. We wind up the article with the proof of Theorem \ref{VCond2} verifying Condition \ref{Cond1} (2) and the main ingredient of the proof is celebrated Skorokhod's representation theorem. Some important applications of the model described in this work is provided in Subsection \ref{app}.

\section{Preliminaries}\label{Sec2}\setcounter{equation}{0}
In this section, we provide results related to pseudo-monotone  operators (see \cite{WL4,MRSSTZ}) and a general  criteria for LDP given in \cite{ABPDVM}. For LDP, we follow the framework and notations from \cite{ABPFD,ABPDVM,JXJZ,JZTZ}, etc.

\subsection{Monotonicity}In this subsection, we recall some results related to  the operator $\A(\cdot,\cdot)$. 
\begin{definition}\label{PM}
	An operator $\A:\V \to\V^*$ is said to be  \textsf{pseudo-monotone}  if the following condition holds: 
	for any sequence $\{\v_m\}_{m\in\N}$ with the weak limit $\v$ in $\V$ and 
	\begin{align}\label{MT1}
		\liminf_{m\to0} \langle \A(\v_m),\v_m-\v\rangle \geq 0,
	\end{align}imply that 
	\begin{align}\label{MT2}
		\limsup_{m\to\infty} \langle \A(\v_m),\v\rangle \leq \langle \A(\v),\v-\u\rangle, \  \text{ for all } \ \u\in\V.
	\end{align}
\end{definition}Let us now recall some useful results from \cite{WL4,MRSSTZ}, which help us to obtain the well-posedness of SPDEs under our consideration. 
\begin{proposition}[Proposition 2.1, \cite{WL4}]\label{PropMT}
	Suppose that the operator $\A(\cdot)$ is pseudo-monotone and $\{\v_m\}_{m\in\mathbb{N}}$ converges  to $\v$ weakly in $\V,$ then 
	\begin{align}\label{MT3}
		\liminf_{m\to\infty} \langle \A(\v_m),\v_m-\v\rangle \leq 0.
	\end{align}
\end{proposition}

The next result plays a crucial role in the proof of Theorem \ref{CPDE1}. The lemma says that  $\v\mapsto\A(\cdot,\v)$ is pseudo-monotone  from $\L^\beta(0,T;\V)$ to $\L^{\frac{\beta}{\beta-1}}(0,T;\V^*)$, for $\beta\in(1,\infty)$. The detailed proof  can be obtained  from \cite{NH,WL4,NS}, etc. 
\begin{lemma}\label{lemmaMT2}
For $\beta\in(1,\infty)$, 	under Hypothesis \ref{hyp1}, assume that 
	\begin{align}\label{MT14}\nonumber
		\v_m &\xrightharpoonup[]{} \v \  \text{ in }\ \L^\beta(0,T;\V),\\\nonumber
		\A(\cdot,\v_m(\cdot)) &\xrightharpoonup[]{}  \wi{\A}(\cdot)\ \text{ in } \ \L^\frac{\beta}{\beta-1}(0,T;\V^*),\\
		\liminf_{m\to\infty} \int_0^T\langle \A(t,\v_m(t)),\v_m(t)\rangle \d t&\geq \int_0^T\langle \wi{\A}(t),\v(t)\rangle\d t,
	\end{align}then for any $\u\in\L^\beta(0,T;\V)$, we have 
	\begin{align}\label{MT15}
		\int_0^T\langle \A(t,\v(t)),\v(t)-\u(t)\rangle \d t\geq \limsup_{m\to\infty}\int_0^T\langle \A(t,\v_m(t)),\v_m(t)-\u(t)\rangle\d t.
	\end{align}
\end{lemma}
\begin{proof}
	Using Hypothesis \ref{hyp1} (H.3), (H.4) and Young's inequality, we obtain 
	\begin{align}\label{MT16}\nonumber
		&	\langle \A(\cdot,\v_m(\cdot)),\v_m(\cdot)-\v(\cdot)\rangle \\&\nonumber= \langle \A(t,\v_m(\cdot)),\v_m(\cdot)\rangle -\langle \A(t,\v_m(\cdot)),\v(\cdot)\rangle\\&\nonumber\leq 
		\fk{a}(\cdot)\big(1+\|\v_m(\cdot)\|_\H^2\big)-L_A\|\v_m(\cdot)\|_\H^\beta +\|\A(t,\v_m(\cdot))\|_{\V^*}\|\v(\cdot)\|_\V \\& \nonumber\leq -L_\A\|\v_m(\cdot)\|_\H^\beta+\fk{a}(\cdot)\big(1+\|\v_m(\cdot)\|_\H^2\big)\nonumber\\&\quad+ \big[\fk{a}(\cdot)+C\|\v_m(\cdot)\|_\H^\beta\big]^{\frac{\beta-1}{\beta}}\big[1+\|\v_m(\cdot)\|_\H^\alpha\big]^{\frac{\beta-1}{\beta}}\|\v(\cdot)\|_\V
	\nonumber	\\&\leq -\frac{L_\A}{2} \|\v_m(\cdot)\|_\H^\beta+ \fk{a}(\cdot)\big(2+\|\v_m(\cdot)\|_\H^2\big)+C\|\v(\cdot)\|_\V^{\beta}+C\|\v_m(\cdot)\|_\H^{\alpha(\beta-1)}\|\v(\cdot)\|_\V^\beta,
	\end{align}
for a.e. $t\in[0,T]$. For simplification, we write 
	\begin{align*}
		g_m(t)&:=	\langle \A(t,\v_m(t)),\v_m(t)-\v(t)\rangle ,\\
		F_m(t)&:= \fk{a}(t)\big(2+\|\v_m(t)\|_\H^2\big)+C\|\v(t)\|_\V^{\beta}+C\|\v_m(t)\|_\H^{\alpha(\beta-1)}\|\v(t)\|_\V^\beta.
	\end{align*}Thus, \eqref{MT16} reduces to 
	\begin{align*}
		g_m(t) \leq -\frac{L_\A}{2} \|\v_m(t)\|_\H^\beta+F_m(t),
	\end{align*}for a.e. $t\in[0,T]$. 	The rest of the proof follows in a similar lines as in the proof of Lemma 2.5, \cite{WL4} by establishing 
	\begin{enumerate}
		\item  For a.e. $t\in[0,T]$, we have 
		\begin{align}\label{MT20}		
			\limsup_{m\to\infty}g_m(t) \leq 0;
		\end{align}
		\item there exists a subsequence $\{\v_{m_i}\}_{i\in\N}$ of $\{\v_m\}_{m\in\N}$ such that 
		\begin{align}\label{MT23}
			\lim_{i\to\infty} g_{m_i}(t)=0,\ \text{ for a.e. }\ t\in[0,T].
		\end{align}
	\end{enumerate}Then exploiting the pseudo-monotone property of the operator $\A(\cdot,\cdot)$  (see Remark 2.2, \cite{MRSSTZ}) and Fatou's lemma, one can complete the proof.
\end{proof}
Let us now rephrase Lemma \ref{lemmaMT2} in the random case. That is, one can show that the operator $\Y(\cdot)\mapsto\A(\cdot,\Y(\cdot))$  is pseudo-monotone from $\L^\beta(\Omega\times[0,T];\V)$ to $ \L^{\frac{\beta}{\beta-1}}(\Omega\times[0,T];\V^*)$, for $\beta\in(1,\infty)$. 
\begin{lemma}[Lemma 2.16, \cite{MRSSTZ}]\label{Cgslem00}
	If 
	\begin{align}\label{Cgslem001}\nonumber
		\Y^{\e}& \xrightharpoonup{}  \Y, \text{ in }\L^\beta(\Omega\times[0,T];\V),\\\nonumber
		\A(\cdot,\Y^{\e}(\cdot))&\xrightharpoonup{}  \wi{\A}(\cdot),\text{ in } \L^{\frac{\beta}{\beta-1}}(\Omega\times[0,T];\V^*),\\
		\liminf_{\e\to0} \E\bigg[\int_0^T \langle \A(t,\Y^{\e}(t)),\Y^{\e}(t)\rangle \d t\bigg]&\geq 	\E\bigg[\int_0^T \langle \wi{\A}(t),\Y(t)\rangle \d t\bigg], 
	\end{align}then $\wi{\A}(\cdot)=\A(\cdot,\Y^{\e}(\cdot)),\;\d t\otimes\P$-a.e.
\end{lemma}

\subsection{Large deviation principle}
Let $\{\Y^\e\}_{\e>0}$ be a family of random variables defined on a probability space $(\Omega,\mathscr{F},\P)$ and taking values in a Polish space $\mathcal{E}$.  The LDP is concerned with the exponential decay of $\P(\Y^\e\in O)$ as $\e\to0$, for an event  $O\in\mathcal{B}(\mathcal{E})$. The exponential decay rate of such probabilities is generally  expressed in terms of a `rate function' $\I$ defined as follows: 
\begin{definition}[Rate function]
	A function $\I:\mathcal{E}\to[0,\infty]$ is called a \textsf{rate function} on $\mathcal{E}$, if for any $M<\infty$, the level set $\{\x\in\mathcal{E}:\I(\x)\leq M\}$ is a compact subset of $\mathcal{E}$. For $O\in\mathcal{B}(\mathcal{E})$, we define $\I(O)=\inf\limits_{\x\in O}\I(\x)$. 
\end{definition}
\begin{definition}[Large deviation principle]
	Let $\I$ be a rate function on $\mathcal{E}$. The sequence $\{\Y^\e\}_{\e>0}$ is said to satisfy the \textsf{large deviation principle (LDP)}  on $\mathcal{E}$ with the rate function $\I$ if the following conditions hold:
	\begin{enumerate}
		\item (Large deviation upper bound). For each closed subset $\F$ of $\mathcal{E}$, 
		\begin{align*}
			\limsup_{\e\to0}\e\ln\P(\Y^\e\in\F)\leq -\I(\F).
		\end{align*}
	\item (Large deviation lower bound). For each open subset $\G$ of $\mathcal{E}$, 
	\begin{align*}
		\liminf_{\e\to0}\e\ln\P(\Y^\e\in\G)\geq -\I(\G).
	\end{align*}
	\end{enumerate} 
\end{definition}
\subsection{Controlled Poisson random measure}
In rest of the paper, we use the following notations. Let $\Z$ be a locally compact Polish space. Let  $\mathscr{M}_{FC}(\Z)$  be the space of all measures $\nu$ on $(\Z,\mathcal{B}(\Z))$ such that $\nu(K)<\infty$, for every compact subset $K$ of $\Z$, and set $\C_0(\Z)$ as the space of continuous functions with compact supports. Endow the space $ \mathscr{M}_{FC}(\Z)$ with the weakest topology such that for every $f\in\C_0(\Z)$, the function $\nu\mapsto\langle \nu,f\rangle=\int_\Z f(z)\d \nu(z), \;\nu\in\mathscr{M}_{FC}(\Z)$ is continuous. Under the topology defined above, the space $\mathscr{M}_{FC}(\Z)$ is a Polish space. For any fixed $T\in(0,\infty)$ and let $\Z_T=[0,T]\times \Z$. Fix a measure $\nu\in \mathscr{M}_{FC}(\Z)$, and let $\nu_T=\lambda_T\otimes\nu$, where $\lambda_T$ is Lebesgue measure on $[0,T]$.

Let us recall that a Poisson random measure $\n$ on $\Z_T$ with the intensity measure $\nu_T$ is a $\mathscr{M}_{FC}(\Z_T)$-valued random variable such that for each $B\in \mathcal{B}(\Z_T)$ with $\nu_T(B)<\infty,\; \n(B)$ is a Poisson distribution with mean $\nu_T(B)$ and for disjoint events $B_1,\ldots,B_j\in\mathcal{B}(\Z_T)$, $\n(B_1),\ldots,\n(B_j)$ are mutually independent random variables (see \cite{NISW}). Denote by $\P,$ the measure induced by $\n$ on  $(\mathscr{M}_{FC}(\Z_T),\mathcal{B}(\mathscr{M}_{FC}(\Z_T)))$. Then letting $\mb{M}=\mathscr{M}_{FC}(\Z_T),\;\P$ is the unique probability measure on $(\mb{M},\mathcal{B}(\mb{M}))$ under the canonical map, $N:\mb{M}\to \mb{M},\; N(m)\doteq m$, is a Poisson random measure with the intensity measure $\nu_T$. With applications to large deviations, we consider  for any $\theta>0,$ probability measure $\P_\theta$ on $(\mb{M},\mathcal{B}(\mb{M}))$ under which $N$ is a Poisson random measure with the intensity $\theta\nu_T$. The corresponding expectation operators will be denoted by $\E$ and $\E_\theta$ (for $\P$ and $\P_\theta$), respectively. 

For a metric space $\mathcal{E}$, the space of all real valued bounded $\mathcal{B}(\mathcal{E})\backslash \mathcal{B}(\R)$-measurable maps and real-valued bounded continuous functions is denoted by $\mathrm{M}_b(\mathcal{E})$ and $\C_b(\mathcal{E})$, respectively. Now, set $\mathrm{Y}=\Z\times [0,\infty)$ and $\mathrm{Y}_T=[0,T]\times\mathrm{Y}$. Similarly, $\bar{\mb{M}}=\mathscr{M}_{FC}(\mathrm{Y}_T)$ and let $\bar{\P}$ be the unique probability measure on $(\bar{\mb{M}},\mathcal{B}(\bar{\mb{M}}))$ under which the canonical map $\bar{N}:\bar{\mb{M}}\to\bar{\mb{M}},\;\bar{N}(\bar{m})\doteq  \bar{m}$, is a Poisson random measure with the intensity measure $\bar{\nu}_T=\lambda_T\otimes\nu\otimes\lambda_\infty$, where $\lambda_\infty$ is the Lebesgue measure on $[0,\infty).$ The corresponding expectation operator is denoted by $\bar{\E}$. Let $\mathscr{F}_T=\sigma\{\bar{N}((0,s)\times O):0\leq s\leq t,\;O\in\mathcal{B}(\mathrm{Y})\}$, and denote its completion by $\bar{\mathscr{F}}_T$ under $\bar{\P}$. Set $\bar{\mathscr{P}}$ as the predictable $\sigma$-field on $[0,T]\times \mb{M}$ with the filtration $\{\bar{\mathscr{F}}_t:0\leq t\leq T\}$ on $(\bar{\mb{M}},\mathcal{B}(\bar{\mb{M}}))$. Let $\bar{\mathcal{A}}$ be the class of all $(\bar{\mathscr{P}}\otimes\mathcal{B}(\Z))\backslash\mathcal{B}[0,\infty)$-measurable maps $\varphi:\Z_T\times\bar{\mb{M}}\to [0,\infty)$.  For $\varphi\in\bar{\mathcal{A}}$, we suppress the argument $\bar{m}$ in $\varphi(s, x, \bar{m})$ and for simplicity, we write $\varphi(s, x)=\varphi(s, x, \bar{m})$. For $\varphi\in\bar{\mathcal{A}}$, define a counting process $N^{\varphi}$ on $\Z_T$ by 
\begin{align*}
	N^{\varphi}((0,T]\times U)=\int_{(0,T]\times U}\int_{(0,\infty)} \chi_{[0,\varphi(s,r)]}(r)\bar{N}(\d s,\d z,\d r), \; t\in[0,T], \ U\in\mathcal{B}(\Z),
\end{align*}where $\chi_A$ denotes the characteristic function of the set $A$. Note that $N^{\varphi}$ is the controlled Poisson measure, with selecting intensity for the points at location $z$ and time $s$ in a possibly random but nonanticipating way. When $\varphi(s,z,\bar{m})\equiv \theta\in(0,\infty)$, we write $N^\varphi=N^\theta$. It should also be noted that $N^\theta$ has the same distribution with respect to $\bar{\P}$ as $N$ has with respect to $\P_\theta$.

\subsection{Poisson random measure and Wiener process}\label{subsec2.4} Set $\mbf{W}=\C([0,T];\R^\infty),\; \mbf{V}=\mbf{W}\times\mb{M}$ and $\mbf{\bar{V}}=\mbf{W}\times\bar{\mb{M}}$. Then, define a mapping $N^{\mbf{V}}:\mbf{V}\to\mb{M}$ by $N^{\mbf{V}}(w,m)=m$, for $(w,m)\in\mbf{V}$, and define $\W^{\mbf{V}}=\{\beta_j^{\mbf{V}}\}_{j\in\mathbb{N}}$ by $\beta_j^{\mbf{V}}(w,m)=w_j$ for $(w,m)\in\mbf{V}$. Similarly, we can define the maps $\bar{N}^{\mbf{\bar{\mbf{V}}}}:\bar{\mbf{V}}\to\bar{\mb{M}}$ and $\W^{\bar{\mbf{V}}}=\{\beta_j^{\bar{\mbf{V}}}\}$. Define the $\sigma$-filtration $\mathscr{H}_t^{\mbf{V}}:=\sigma\{N^{\mbf{V}}((0,s]\times O), \beta_j^{\mbf{V}}(s):0\leq s\leq t,O\in\mathcal{B}(\Z),\; j\in\mathbb{N}\}$. For every $\theta>0,\;\P_\theta^{\mbf{V}}$ denotes the unique probability measure on $(\mbf{V},\mathcal{B}(\mbf{V}))$ such that:
\begin{enumerate}
	\item $\{\beta_j^{\mbf{V}}\}_{j\in\mathbb{N}}$ is an independent and identically distributed family of standard Brownian motions;
	\item $N^{\mbf{V}}$ is a Poisson random measure with the intensity measure $\theta\nu_T$.
\end{enumerate}
Analogously, we can define $(\bar{\P}_\theta^{\bar{\mbf{V}}},\bar{\mathscr{H}}_t^{\bar{\mbf{V}}})$ and denote $\bar{\P}_{\theta=1}^{\bar{\mbf{V}}}$ by $\bar{\P}^{\bar{\mbf{V}}}$.  We denote the $\bar{\P}^{\bar{\mbf{V}}}$-completion of $\{\bar{\mathscr{H}}_t^{\bar{\mbf{V}}}\}$ by $\{\bar{\mathscr{F}}_t^{\bar{\mbf{V}}}\}$ and the predictable $\sigma$-field $\bar{\mathscr{P}}^{\bar{\mbf{V}}}$ on $[0,T]\times\bar{\mbf{V}}$ with the filtration $\{\bar{\mathscr{F}}_t^{\bar{\mbf{V}}}\}$  on $(\bar{\mbf{V}},\mathcal{B}(\bar{\mbf{V}}))$. Let ${\mathcal{\bar{A}_+}}$ be the class of all $(\bar{\mathcal{P}}^{\bar{\mbf{V}}}\otimes \mathcal{B}(\Z))\backslash\mathcal{B}[0,\infty)$-measurable maps $\varphi:\Z_T\times\bar{\mbf{V}}\to[0,\infty)$. Define $\ell :[0,\infty)\to[0,\infty)$ by 
\begin{align*}
	\ell(r)=r\ln r-r+1,\; r\in[0,\infty).
\end{align*}For any $\varphi\in\bar{\mathcal{A}}_+$ the quantity 
\begin{align*}
	L_T(\varphi)=\int_{\Z_T}\ell(\varphi(t,z,\omega))\nu_T(\d t,\d z), 
\end{align*}is well defined as a $[0,\infty)$-valued random variable.
Denote $\mathcal{I}_2,$ the Hilbert space of all real sequences $\x=\{\x_j\}_{j\in\N}$ satisfying $\|\x\|^2=\sum\limits_{j=1}^{\infty}\x_j^2<\infty$, with the usual inner product. Define
 \begin{align*}
\mathcal{L}_2:=\left\{\psi:\psi \text{ is } \bar{\mathcal{P}}^{\bar{\mbf{V}}}\backslash \mathcal{B}(\R^\infty) \text{ measurable and } \int_0^T \|\psi(s)\|^2\d s<\infty,\; \bar{\P}^{\bar{\mbf{V}}}\text{-a.s.}\right\}.
\end{align*}Set $\mathcal{U}:=\mathcal{L}_2\times\bar{\mathcal{A}}_+$. Define 
\begin{align*}
	\vi{L}_T (\psi):=\frac{1}{2}\int_0^T \|\psi(s)\|^2\d s, \ \text{ for } \ \psi\in\mathcal{L}_2,
\end{align*}and 
 \begin{align}\label{2p9}
\bar{L}_T(\q):=\vi{L}_T(\psi)+L_T(\varphi)\ \text{ for }\ \q=(\psi,\varphi)\in\mathcal{U}.
\end{align}

Let $\{K_n\subset\Z, n=1,2,\dots\}$ be an increasing family of compact sets such that $\bigcup\limits_{n=1}^\infty K_n=\Z$. For each $n\in\N$, we define 
\begin{align*}
	\bar{\mathcal{A}}_{b,n}\doteq \bigg\{\varphi\in \bar{\mathcal{A}}_+&: \text{ for all } (t,\omega)\in[0,T]\times \bar{\mb{M}},\; n\geq \varphi(t,z,\omega)\geq \frac{1}{n}\text{ if } z\in K_n \\& \text{ and } \varphi(t,z,\omega)=1 \text{ if }z\in K_n^c\bigg\},
\end{align*}and let $\bar{\mathcal{A}}_b=\bigcup\limits_{n=1}^\infty \bar{\mathcal{A}}_{b,n}$. Considering $\varphi$ as a control that perturbs jump rates away from 1, when $\varphi\neq 1$, we see that controls in $\bar{\mathcal{A}}_b$ are bounded and perturb only off a compact set, where the bounds and set can depend on $\varphi$. The following representation formula is established in \cite{ABPFD} (see Theorem 2.4, \cite{ABJCPD} also).
\begin{theorem}
	Let $F\in \mathrm{M}_b(\mb{M})$. Then for any $\theta>0$, 
	\begin{align*}
	-\ln\E_\theta\big[e^{-F(N)}\big] &=-\ln\bar{\E}\big[e^{-F(N^\theta)}\big]=\inf_{\varphi\in \bar{\mathcal{A}}_+}\bar{\E}\big[\theta L_T(\varphi)+F(N^{\theta\varphi})\big]\\&
	=\inf_{\varphi\in \bar{\mathcal{A}}_b}\bar{\E}\big[\theta L_T(\varphi)+F(N^{\theta\varphi})\big].
	\end{align*}
\end{theorem}
\subsection{A general large deviation result}
In this subsection, we recall a general criterion on the LDP established in \cite{ABPDVM}. Let $\{\mathscr{G^\e}\}_{\e>0}$ be a family of measurable maps from $\bar{\mbf{V}}$ to $\mb{U}$, where $\mb{U}$ is some Polish space. Let us present a sufficient condition for the LDP to hold for the family $\big\{\mathscr{G}^\e(\sqrt{\e}\W,\e N^{\e^{-1}})\big\}$ as $\e\to0$.

Let us define 
\begin{align*}
	S^{\Upsilon}=\big\{g:\Z_T\to[0,\infty):L_T(g)\leq \Upsilon\}
\end{align*}and 
\begin{align*}
	\vi{S}^{\Upsilon}=\big\{f:\L^2([0,T];\mathcal{I}_2):\vi{L}_T(f)\leq \Upsilon\}.
\end{align*}A function $g\in S^\Upsilon$ can be identified with a measure $\nu_T^g\in\mb{M}$ defined by 
\begin{align*}
	\nu_T^g(O)=\int_Og(s,z)\nu_T(\d s,\d z) \ \text{ for } \  O\in\mathcal{B}(\Z_T).
\end{align*}From Appendix A.1, \cite{ABJCPD}, we know that this identification induces a topology on $S^\Upsilon$ under which $S^\Upsilon$ is a compact space and in the sequel, we use this topology on $S^\Upsilon$. Set $\bar{S}^\Upsilon=\vi{S}^\Upsilon\times S^\Upsilon$. We define $\mb{S}=\bigcup\limits_{\Upsilon\geq 1}\bar{S}^\Upsilon$, and let 
\begin{align*}
	\mathcal{U}^\Upsilon=\{\boldsymbol{q}=(\psi,\varphi)\in\mathcal{U}: \q(\omega)\in \bar{S}^\Upsilon, \bar{\P}^{\bar{\mbf{V}}}\text{-a.e. } \omega\}.
\end{align*}
From \cite{ABJCPD,ABPFD}, we know that the following condition is sufficient for proving the  LDP for a family $\big\{\mathscr{G}^\e(\sqrt{\e}\W,\e N^{\e^{-1}})\big\}$ as $\e\to0$.
\begin{condition}\label{Cond1}
	There exists a measurable map $\mathscr{G}^0:\bar{\mbf{V}}\to\mb{U}$ such that the following hold:
	\begin{enumerate}
		\item For any $\Upsilon\in\N$, let $\p_n=(f_n,g_n), \p=(f,g)\in \bar{S}^\Upsilon$ be such that $\p_n\to\p$ as $n\to\infty$. Then 
		\begin{align*}
			\mathscr{G}^0\left( \int_0^{\cdot}f_n(s)\d s, \nu_T^{g_n}\right) \to \mathscr{G}^0\left( \int_0^{\cdot}f(s)\d s, \nu_T^{g}\right) \  \text{ as } \  n\to\infty.
		\end{align*}
	\item For any $\Upsilon\in\N$, let $\q_\e=(\psi_\e,\varphi_\e),\ \q=(\psi,\varphi)\in  \mathcal{U}^\Upsilon$ be such that $\q_\e$ converges in distribution to $\q$ as $\e\to0$. Then, 
	\begin{align*}
		\mathscr{G}^\e\left(\sqrt{\e}\W(\cdot)+\int_0^{\cdot}\psi_\e\d s,\e N^{\e^{-1}\varphi_\e}\right) \Rightarrow \mathscr{G}^0\left(\int_0^{\cdot}\psi(s)\d s,\nu_T^\varphi\right) \  \text{ as } \ \e\to0,
	\end{align*} where $``\Rightarrow"$ denotes the convergence in distribution.
	\end{enumerate}
\end{condition}It should be noted that 
\begin{enumerate}
	\item the condition (1) requires continuity in the control for deterministic controlled systems,
	\item the condition (2) is a law of large numbers result for stochastic controlled  systems with  small noise.
\end{enumerate}In both cases, we are allowed to assume the controls take values in a compact set.

For $\phi \in\mb{U}$, define $\mb{S}_\phi=\left\{\p=(f,g)\in\mb{S}: \phi =\mathscr{G}^0\big(\int_0^{\cdot}f(s)\d s ,\nu_T^g\big)\right\}$. Let $\I:\mb{U}\to[0,\infty]$ be defined by 
\begin{align}\label{RF}
	\I(\phi)= \inf_{\p=(f,g)\in\mb{S}_\phi} \bar{L}_T(\p), \ \text{ for } \  \phi\in\mb{U},
\end{align}with the convention $\I(\phi)=\infty$ if $\mb{S}_\phi=\emptyset$, where $\bar{L}_T(\cdot)$ is defined in \eqref{2p9}. 
 
 Let us recall an important result form \cite{ABPDVM}. 
 \begin{theorem}\label{thrm1}
 	For $\e>0$, let $\mathcal{Z}^\e$ be defined by $\mathcal{Z}^\e=\mathscr{G}^\e(\sqrt{\e}\W,\e N^{\e^{-1}})$ and suppose that  Condition \ref{Cond1} holds true. Then $\I$ defined in \eqref{RF} is a rate function on $\mb{U}$ and the family $\{\mathcal{Z}^\e\}_{\e>0}$ satisfies the LDP with the rate function $\I$.
 \end{theorem}The following strengthened form of Theorem \ref{thrm1} is useful in applications (see Appendix A.2, \cite{ABJCPD}). Define $\vi{\mathcal{U}}^\Upsilon=\mathcal{U}^\Upsilon \cap \{\q=(\psi,\varphi):\varphi\in \bar{\mathcal{A}}_b\}$.
\begin{theorem}\label{thrm2}
	Suppose that Condition \ref{Cond1} holds true with $\vi{\mathcal{U}}^\Upsilon$ instead of $\mathcal{U}^\Upsilon$. Then the conclusions of Theorem \ref{thrm1} continue to hold.
\end{theorem}

\section{LDP for the system \eqref{1.1}}\label{Sec3}\setcounter{equation}{0}
Our goal in this work is to obtain a Wentzell-Freidlin type  LDP for the solution to the system \eqref{1.1} as $\e\to0$ on $\D([0,T];\H)$, the space of c\`adl\`ag functions from $[0,T]$ to $\H$.  Let $\Y^\e(\cdot)$ be the unique solution to the system \eqref{1.1} with the initial condition $\x\in\H$. In this section, we state the LDP on  $\D([0,T];\H)$  for the solution $\Y^\e(\cdot)$ under appropriate assumptions.

Choose $\mb{U}=\D([0,T];\H)$ in  Condition \ref{Cond1} with the Skorokhod topology $\mb{U}_S$. We already know that $(\mb{U},\mb{U}_S)$ is a Polish space. For any $p>0$, we define
\begin{align*}
	\mathcal{H}_p=\bigg\{&h:\Z_T\to\R^+: \text{ there exists } \delta>0, \text{ such that for all } E\in \mathcal{B}([0,T])\otimes\mathcal{B}(\Z) \\&\qquad\text{ with }\nu_T(E)<\infty, \text{ we have }\int_E\exp\big(\delta h^p(s,z)\big)\nu(\d z)\d t<\infty \bigg\}.
\end{align*} 
It is easy to verify that $\mathcal{H}_p\subset \mathcal{H}_q$, for any $q\in(0,p)$ and \begin{align}\label{LDP001}
	\bigg\{h:\Z_T\to\R^+: \sup_{(t,z)\in\Z_T}h(t,z)<\infty\bigg\}\subset \mathcal{H}_p, \ \text{ for all } \ p>0.
\end{align}
In the sequel, we need the following lemmas. The proof of the following lemma can be obtained from Lemma 3.4, \cite{ABJCPD}.
\begin{lemma}\label{lemUF}
	For any $h\in \mathcal{H}_p\cap \L^{p'}(\nu_T),\;p'\in(0,p]$, there exists a constant $C_{h,p,p',\Upsilon}$ such that 
	\begin{align*}
		C_{h,p,p',\Upsilon}:=\sup_{g\in S^\Upsilon}\int_{\Z_T}h^{p'}(s,z)(g(s,z)+1)\nu(\d z)\d s<\infty.
	\end{align*}For any $h\in \mathcal{H}_2\cap \L^2(\nu_T)$, there exists a constant $C_{h,\Upsilon}$ such that 
	\begin{align*}
		C_{h,\Upsilon}:=\sup_{g\in S^\Upsilon}\int_{\Z_T}h(s,z)|g(s,z)-1|\nu(\d z)\d s<\infty.
	\end{align*}
\end{lemma}
The proof of the following lemma can be obtained from Lemma 3.3, \cite{XYJZTZ}.
\begin{lemma}
	\begin{enumerate}
		\item If $\sup\limits_{t\in[0,T]}\|\Y(t)\|_\H<\infty$, for any $\p=(f,g)\in\mb{S}$, then 
		\begin{align*}
			\B(\cdot,\Y(\cdot))f(\cdot)\in \L^1(0,T;\H),\; \int_{\Z}\gamma(\cdot,\Y(\cdot),z)(g(\cdot,z)-1)\nu(\d z)\in \L^1(0,T;\H).
		\end{align*}
		\item If the family of mappings $\{\Y_n:[0,T]\to\H,\ n\in\N\}$ satisfies $C=\sup\limits_n\sup\limits_{s\in[0,T]}\|\Y_n(s)\|_\H<\infty$, then 
		\begin{align*}
			\vi{C}^\Upsilon&:=\sup_{\p=(f,g)\in\bar{S}^\Upsilon}\sup_n\bigg[\int_0^T\bigg\|\int_\Z\gamma(s,\Y_n(s),z)(g(s,z)-1)\nu(\d z)\bigg\|_\H\d s \\&\qquad+\int_0^T\|\B(s,\Y_n(s))f(s)\|_\H\d s\bigg]\\&<\infty.
		\end{align*}
	\end{enumerate}
\end{lemma}
The proof of which can be obtained from Lemma 3.11, \cite{ABJCPD} (see Lemma 2.8, \cite{MBPD} also).
\begin{lemma}\label{lemUF3}
	Let $h:\Z_T\to\R$ be a measurable function such that 
	\begin{align*}
		\int_{\Z_T}|h(s,z)|^2\nu(\d z)\d s<\infty,
	\end{align*}and for all $\delta\in(0,\infty)$ and $K\in \mathcal{B}( \Z_T)$ satisfying $\nu_T(K)<\infty$, 
	\begin{align*}
		\int_K \exp \big(\delta|h(s,z)|\big)\nu(\d z)\d s<\infty.
	\end{align*}
	\begin{enumerate}
		\item Fix $\Upsilon\in\N$, and let $g_n,g\in S^\Upsilon$ be such that $g_n\to g$ as $n\to\infty$. Then
		\begin{align*}
			\lim_{n\to\infty} \int_{\Z_T} h(s,z)(g_n(s,z)-1)\nu(\d z)\d s= \int_{\Z_T} h(s,z)(g(s,z)-1)\nu(\d z)\d s.
		\end{align*}
		\item Fix $\Upsilon\in\N$. Given $\e>0$, there exists a compact set $K_\e\subset \Z$, such that 
		\begin{align*}
			\sup_{g\in S^\Upsilon} \int_0^T\int_{K_\e^c}|h(s,z)||g(s,z)-1|\nu(\d z)\d s\leq \e.
		\end{align*}
		\item For every $\e>0$, there exists $\delta>0$ such that for any  $K\subset \mathcal{B}([0,T])$ satisfying $\lambda_T(K)<\delta$,  we have 
		\begin{align*}
			\sup_{g\in S^\Upsilon} \int_K\int_\Z h(s,z)\big|g(s,z)-1\big|\nu(\d z)\d s\leq\e.
		\end{align*}
	\end{enumerate}
\end{lemma}
In order to prove  LDP for the solution to  the system \eqref{1.1}, we need some additional  assumptions  on  the noise coefficients stated below: 
\begin{hypothesis}\label{hyp2} The coefficients $\B$ and $\gamma$ satisfy the following assumptions: 
\begin{enumerate}
	\item[(H.7)] There exists a function  $L_\B\in\L^1(0,T;\R^+)$ such that 
	\begin{align*}
		\|\B(t,\x)-\B(t,\y)\|_{\L_2}^2 \leq L_\B(t)\|\x-\y\|_{\H}^2, \  \text{ for all }\ (t,\x), (t,\y)\in[0,T]\times\V.
	\end{align*}
	\item[(H.8)] There exist positive constants $\eta_0$, $p\geq M$ with $M:=\frac{2\alpha(\beta-1)(\beta+\eta_0)}{\beta}\vee\frac{4(\beta-1)(\beta+\eta_0)}{\beta}\vee4\vee(\alpha+2)$ and $L_\gamma \in \L^2(\nu_T)\cap\L^4(
	\nu_T)\cap \L^{\alpha+2}(\nu_T)\cap \L^M(\nu_T)\cap \L^{\frac{M}{2}}(\nu_T)\cap \mathcal{H}_p$ such that 
	\begin{align*}
		\|\gamma(t,\x,z)\|_{\H} \leq L_\gamma(t,z)(1+\|\x\|_{\H}),\  \text{ for all } \ (t,\x,z)\in [0,T]\times \V\times\Z.
	\end{align*}
\item[(H.9)] There exists a function $R_\gamma\in \L^2(\nu_T)\cap\mathcal{H}_2 $ such that 
	\begin{align*}
	\|\gamma(t,\x,z)-\gamma(t,\y,z)\|_{\H} \leq R_\gamma(t,z)\|\x-\y\|_\H, \ \text{ for all } \  (t,z)\in[0,T]\times\Z,\; \x,\y\in\V.
		\end{align*}
\end{enumerate}
\begin{remark}
	The assumptions (H.7) and (H.9) are stronger than (H.5) and (H.6) (2) in Hypothesis \ref{hyp1}, respectively. The assumption (H.8) implies Hypothesis \ref{hyp1} (H.6) (3). 
\end{remark}
\end{hypothesis}

Let $\I:\D([0,T];\H)\to[0,\infty]$ be defined as in \eqref{RF}. Our main result of this paper can be stated as follows:
\begin{theorem}\label{thrm5}
	Assume that Hypotheses \ref{hyp1} and \ref{hyp2} hold. Then the family $\{\Y^\e\}_{\e>0}$ satisfies the LDP on the space $\D([0,T];\H)$ with the rate function $\I(\cdot)$  with respect to the topology of uniform convergence.
\end{theorem}
\begin{proof}
	The proof of this Theorem is straight forward from Theorem \ref{thrm1}. In view of Theorem \ref{thrm1}, it is sufficient to verify Condition \ref{Cond1} (1) and (2). Part (1) of Condition \ref{Cond1}  is verified in Proposition \ref{VCond1}, and the remaining part (2) is established in Theorem \ref{VCond2}.
\end{proof}

\section{Verification of Condition \ref{Cond1} (1)}\label{Sec4}\setcounter{equation}{0}
In this section, we verify Condition \ref{Cond1} (1),   so that first part of the proof of Theorem \ref{thrm5} can be completed. We first consider a deterministic control problem for a class of partial differential equations with fully local monotone coefficients and discuss its solvability results using Lemmas \ref{lemmaMT2}, \ref{lemUF} and \ref{lemUF3}. 

\subsection{Skeleton equations} We start by introducing  the measurable map $\mathscr{G}^0$ that is used in the definition of rate function $\I(\cdot)$ given in \eqref{RF} and also used for the verification of Condition \ref{Cond1}. Under appropriate assumptions, for every $\p=(f,g)\in \mb{S}$, we show that the following deterministic control problem in $\V^*$:
\begin{equation}\label{SE1}
	\left\{
	\begin{aligned}
	\d \vi{\Y}^{\p} (t)&=\bigg[\A(t,	\vi{\Y}^{\p} (t))+\B(t,	\vi{\Y}^{\p} (t))f(t)+\int_\Z \gamma(t,	\vi{\Y}^{\p} (t),z)\big(g(t,z)-1\big)\nu(\d z)\bigg]\d t,\\
	 \vi{\Y}^{\p} (0)&=\x\in\H,
\end{aligned}\right.	 
\end{equation}
for a.e. $t\in[0,T]$ has a unique continuous weak solution. 

We first prove the existence and uniqueness of weak  solution of the deterministic control problem \eqref{SE1}, using a Faedo-Galerkin approximation technique and exploiting the pseudo-monotonicity property of the nonlinear operator $\A(\cdot,\cdot)$.
\begin{theorem}\label{CPDE1}
	Assume $\x\in\H$, and $\p=(f,g)\in\mb{S}$. Under Hypotheses \ref{hyp1} and \ref{hyp2},  there exists a \textsf{unique weak  solution} $\vi{\Y}^{\p}\in\C([0,T];\H)\cap \L^\beta(0,T;\V)$, for $\beta\in(1,\infty)$ such that 
	\begin{align}\label{SE2}\nonumber
		\vi{\Y}^{\p} (t)&=\x +\int_0^t\A(s,	\vi{\Y}^{\p} (s))\d s+\int_0^t\B(s,	\vi{\Y}^{\p} (s))f(s)\d s\\&\quad+\int_0^t\int_\Z \gamma(s,	\vi{\Y}^{\p} (s),z)\big(g(s,z)-1\big)\nu(\d z)\d s,\  \text{ in } \ \V^*.
	\end{align}Moreover, for fixed $\Upsilon\in\N$, there exists a positive constant $C$ such that 
\begin{align}\label{CPDE01}
	\sup_{\p\in \bar{S}^{\Upsilon}}\bigg[\sup_{t\in[0,T]}\|\vi{\Y}^{\p}(s)\|_\H^2+\int_0^T\|\vi{\Y}^{\p}(s)\|_\V^\beta \d s\bigg]\leq C\left(1+\|\x\|_{\H}^2\right).
\end{align} 
\end{theorem}
\begin{proof}The proof of this theorem is divided into the following steps: 
	
	\vspace{2mm}
	\noindent
	\textbf{Step 1:} \textsf{An abstract operator.} First we define the operator $\cc(\cdot):\V\to\V^*$ by
	\begin{align}\label{CPDE02}
		\langle \cc(\cdot,\phi),\xi\rangle =-	\langle \A(\cdot,\phi),\xi\rangle -	\big( \B(\cdot,\phi)f(\cdot),\xi\big) -\bigg(\int_\Z \gamma(\cdot,\phi,z)\big(g(\cdot,z)-1\big)\nu(\d z),\xi\bigg) ,
	\end{align}for all $\phi,\xi\in\V$. For simplicity of notations, in this step, we suppress the dependency of the first variable  in each of the operators and functions. 

\vspace{2mm}
\noindent
\textsl{The operator $\cc(\cdot)$ is well-defined:} By the definition of the operator $\cc(\cdot)$, we have for all $\xi\in\V$
\begin{align}\label{CPDE03}
	&	\big|\langle \cc(\phi),\xi\rangle \big|\nonumber\\&\leq \big|\langle \A(\phi),\xi\rangle \big|+\big|	\big( \B(\phi)f,\xi\big)\big|+\bigg|\bigg(\int_\Z \gamma(\phi,z)\big(g(z)-1\big)\nu(\d z),\xi\bigg)\bigg|\nonumber\\&
		\leq \|\A(\phi)\|_{\V^*}\|\xi\|_\V^{\frac{1}{\beta}}\|\xi\|_\V^{\frac{\beta-1}{\beta}}+\|\B(\phi)\|_{\L_2}\|f\|_\H\|\xi\|_\H +\int_\Z \|\gamma(\phi,z)\|_\H\big|g(z)-1\big|\|\xi\|_\H\nu(\d z)\nonumber\\&\leq C\big(\|\A(\phi)\|_{\V^*}^{\frac{\beta-1}{\beta}}+1\big)\|\xi\|_\V +\sqrt{\fk{b}}\big(1+\|\phi\|_\H\big)\|f\|_\H\|\xi\|_\H\nonumber\\&\quad+\big(1+\|\phi\|_\H\big)\|\xi\|_\H\int_\Z L_\gamma(z)\big|g(z)-1\big|\nu(\d z)\nonumber\\&\leq \bigg\{C\big[\big(\fk{a}+C\|\phi\|_\V^\beta\big)\big(1+\|\phi\|_\H^\alpha\big)+1 \big]+\sqrt{\fk{b}}\|f\|_\H\big(1+\|\phi\|_\H\big)\nonumber\\&\qquad+\big(1+\|\phi\|_\H\big)\int_\Z L_\gamma(z)\big|g(z)-1\big|\nu(\d z)\bigg\}\|\xi\|_\V,
\end{align}
where we have used Hypothesis \ref{hyp1} (H.4)-(H.6), Young's and H\"older's inequalities.
Therefore from \eqref{CPDE03}, we conclude
\begin{align}\label{CPDE08}\nonumber
	\|\cc(\phi)\|_{\V^*} &\leq \bigg\{C\big(\fk{a}+C\|\phi\|_\V^\beta\big)\big(1+\|\phi\|_\H^\alpha\big)+1+\sqrt{\fk{b}}\|f\|_\H\big(1+\|\phi\|_\H\big)\\&\qquad+ \big(1+\|\phi\|_\H\big)\int_\Z L_\gamma(z)\big|g(z)-1\big|\nu(\d z)\bigg\}.
\end{align}
Hence, we rewrite the integral equation  \eqref{SE2} in the following form:
\begin{align}\label{CPDE09}
	\vi{\Y}^{\p}(t)=\x-\int_0^t\cc(	s,\vi{\Y}^{\p}(s))\d s, \ \text{ for all } \ t\in[0,T],
\end{align}in the triplet $\V\subset\H\subset\V^*$. Let us establish some properties of the operator $\cc(\cdot)$, which will help us to obtain the existence of  a unique weak solution of the problem \eqref{SE2}.

\vspace{2mm}
\noindent
\textsl{The operator $\cc(\cdot)$ is coercive:} By using Hypotheses \ref{hyp1} (H.3), (H.5), \ref{hyp2} (H.8), H\"older's and Young's inequalities, we have for all $\phi\in\H$
\begin{align}\label{CPDE10}\nonumber
	\langle \cc(\phi),\phi\rangle&\nonumber \geq L_\A\|\phi\|_\V^2- \bigg[\fk{a}+\sqrt{\fk{b}}\|f\|_\H+\int_\Z L_\gamma(z)\big|g(z)-1\big|\nu(\d z)\\&\qquad+\bigg\{\fk{a}+2\sqrt{\fk{b}}\|f\|_\H+ 2\int_\Z L_\gamma(z)\big|g(z)-1\big|\nu(\d z)\bigg\}\|\phi\|_\H^2\bigg],
\end{align}which implies that the operator $\cc(\cdot)$ is coercive.

\vspace{2mm}
\noindent
\textsl{The operator $\cc(\cdot)$ is fully  locally monotone:} The following local monotnocity condition holds for the operator $\cc(\cdot)$:
\begin{align}\label{CPDE11}	\nonumber
&	\langle \cc(\phi_1)-\cc(\phi_2),\phi_1-\phi_2\rangle \\&\nonumber\qquad+\bigg\{\frac{1}{2}\big(\fk{a}+\rho(\phi_1)+\eta(\phi_2)\big)+\|f\|_\H^2+\int_\Z R_\gamma(z)\big|g(z)-1\big|\nu(\d z)\bigg\}\|\phi_1-\phi_2\|_\H^2\\&\quad\geq 0\ \text{ for all }\  \phi_1,\phi_2\in\V.
\end{align}The above condition \eqref{CPDE11} can be verified as follows: From the definition of $\cc(\cdot)$, we have 
\begin{align}\label{CPDE12}\nonumber
	&	\langle \cc(\phi_1)-\cc(\phi_2),\phi_1-\phi_2\rangle \\&\nonumber=	-\langle \A(\phi_1)- \A(\phi_2),\phi_1-\phi_2\rangle-\big((\B(\phi_1)-\B(\phi_2))f,\phi_1-\phi_2\big)\\&\quad-\bigg(\int_\Z \big(\gamma(\phi_1,z)-\gamma(\phi_2,z)\big)\big(g(z)-1\big)\nu(\d z),\phi_1-\phi_2\bigg).
\end{align}
Using the Cauchy-Schwarz  and Young's inequalities, we estimate  the penultimate term in the right hand side of \eqref{CPDE12} as
\begin{align}\label{CPDE13}	\nonumber
	\big| \big((\B(\phi_1)-\B(\phi_2))f,\phi_1-\phi_2\big)\big| &\leq \|(\B(\phi_1)-\B(\phi_2))f\|_\H\|\phi_1-\phi_2\|_\H \\&\leq 
	\frac{1}{2}\|\B(\phi_1)-\B(\phi_2)\|_{\L_2}^2+\frac{1}{2}\|f\|_\H^2\|\phi_1-\phi_2\|_\H^2.
\end{align}Making use of Hypothesis \ref{hyp2} (H.9), the Cauchy-Schwarz and H\"older's inequalities for the final term  in the right hand side of \eqref{CPDE12}, we obtain 
\begin{align}\label{CPDE14}\nonumber
&	\bigg|\bigg(\int_\Z \big(\gamma(\phi_1,z)-\gamma(\phi_2,z)\big)\big(g(z)-1\big)\nu(\d z),\phi_1-\phi_2\bigg)\bigg|\\&\leq \|\phi_1-\phi_2\|_\H^2\int_\Z R_\gamma(z)\big|g(z)-1\big|\nu(\d z).
\end{align}Substituting \eqref{CPDE13} and \eqref{CPDE14} in \eqref{CPDE12}, then using Hypothesis \ref{hyp1} (H.2),  we deduce the required condition \eqref{CPDE11}.

\vspace{2mm}
\noindent
\textsl{The operator $\cc(\cdot)$ is hemicontinuous:}  Let the operator $\bb(\cdot)$ be  defined as follows:
\begin{align*}
	\langle \bb(\phi),\xi\rangle =-	\big( \B(\phi)f(\cdot),\xi\big) -\bigg(\int_\Z \gamma(\phi,z)\big(g(\z)-1\big)\nu(\d z),\xi\bigg)\  \text{ for all } \ \phi,\xi\in\V.
\end{align*} 
Consider a sequence $\{\phi_n\}_{m\in\N}\in\V$ such that $\|\phi_n-\phi\|_\V\to0$ as $n\to\infty$. For any $\xi\in\V$, we find
\begin{align}\label{CPDE15}\nonumber
&|\langle \bb(\phi_n)-\bb(\phi),\xi\rangle|\\&\nonumber\leq |((\B(\phi_n)-\B(\phi))f,\xi)| +\bigg|\bigg(\int_\Z\big( \gamma(\phi_n,z)-\gamma(\phi,z)\big(g(z)-1\big)\nu(\d z),\xi\bigg)\bigg|\\&\nonumber\leq \|\B(\phi_n)-\B(\phi)\|_{\L_2}\|f\|_\H\|\xi\|_\H+\int_\Z \|\gamma(\phi_n,z)-\gamma(\phi,z)\|_\H |g(z)-1|\|\xi\|_\H\nu(\d z) \\&\nonumber\leq L_\B\|\phi_n-\phi\|_\H \|f\|_\H\|\xi\|_\H+\|\phi_n-\phi\|_\H\|\xi\|_\H\int_\Z R_\gamma(z)|g(z)-1|\nu(\d z)\\& \to 0\ \text{ as } n\to\infty,
\end{align}
where we have used Hypothesis \ref{hyp2} (H.7), (H.9). This implies  the demiconitnuity of $\bb(\cdot)$ and hence hemicontinuity also. Using the hemicontinuity property (see Hypothesis \ref{hyp1} (H.1)) of the operator $\A(\cdot)$, we conclude that  the operator $\cc(\cdot)= - \A(\cdot) +	 \bb(\cdot)$ is hemicontinuous.

\vspace{2mm}
\noindent
\textbf{Step 2:} \textsf{Faedo-Galerkin approximation and finite dimensional problem.}  Let us prove the existence of a weak solution to the equation \eqref{SE2} using a Faedo-Galerkin approximation, local monotonicity  and hemicontinuity properties of the operator $\cc(\cdot)$. 

 Let $\{\bfe_1,\bfe_2,\ldots,\bfe_n,\ldots\}\subset \V$ be a complete orthonormal system in  $\H$ and let $\H_n$ be the $n$-dimensional subspace of $\H$ spanned by $\{\bfe_1,\bfe_2,\ldots,\bfe_n\}$. Let us denote the projection of the space $\V^*$ into $\H_n$ by $\mathbf{P}_n$, that is, $\mathbf{P}_n\x =\sum\limits_{j=1}^{n}\langle \x,\bfe_j\rangle\bfe_j $. As we know that every element $\x\in\H$ induces a functional $\x^*\in\H$ by the formula $\langle \x^*,\y\rangle =(\x,\y),\; \y\in\V$, then $\mathbf{P}_n\big|_{\H}$, the orthogonal projection of $\H$ onto $\H_n$ is given by $\mathbf{P}_n\x=\sum\limits_{j=1}^{n}(\x,\bfe_j)\bfe_j$. Clearly,  $\mathbf{P}_n$ is the orthogonal projection from $\H$ onto $\H_n$.
 
 Define $\cc_n(\cdot,\vi{\Y}^{\p}_n(\cdot))=\mathbf{P}_n\cc(\cdot,\vi{\Y}^{\p}_n(\cdot))$. With the above setting, we consider the following system of ODEs:
 \begin{equation}\label{CPDE16}
 	\left\{ 	\begin{aligned}
 		\langle \partial_t \vi{\Y}^{\p}_n(t), \bf{\phi}\rangle& =-\langle \cc_n(\vi{\Y}^{\p}_n(t)),\bf{\phi}\rangle, \ \text{ for }\ t\in[0,T],\\
 		\big(\vi{\Y}^{\p}_n(0),\bf{\phi}\big)&=\big(\x_0^n,\bf{\phi}\big),
 	\end{aligned}
 \right.
 \end{equation}where $\x_0^n=\mathbf{P}_n \x$ for all $\bf{\phi}\in \H_n$. By a classical result from  Theorem 1.2., \cite{NVK} (cf. Theorem 3.3.1, \cite{CPMR}), there exists a unique local solution $\vi{\Y}^{\p}_n\in \C([0,T^*];\H_n)$, for some $0<T^*\leq T$. Let now show that $T=T^*$ by establishing  some uniform energy estimates. 

\vspace{2mm}
\noindent
\textsl{Energy estimates:}
Let us prove the uniform energy estimates satisfied by the system \eqref{CPDE16}, which extend this local solution to the global one. Taking the inner product with $\vi{\Y}^{\p}_n(\cdot)$ to the first equation of the system \eqref{CPDE16}, we find  
\begin{align*}
	\|\vi{\Y}^{\p}_n(t)\|_\H^2&=\|\x_0^n\|_\H^2-2\int_0^t\langle \cc_n(\vi{\Y}^{\p}_n(s)),\vi{\Y}^{\p}_n(s)\rangle\d s, 
\end{align*}
for all $t\in[0,T]$. Using the coercivity condition \eqref{CPDE10} and Young's inequality,  we get for all $t\in[0,T]$,
\begin{align}\label{CPDE17}\nonumber
&	\|\vi{\Y}^{\p}_n(t)\|_\H^2 +2L_\A\int_0^t\|\vi{\Y}^{\p}_n(s)\|_\V^2\d s\\&\nonumber\leq \|\x_0^n\|_\H^2
+2\int_0^t \bigg(\fk{a}(s)+\sqrt{\fk{b}(s)}\|f(s)\|_\H+\int_\Z L_\gamma(s,z)\big|g(s,z)-1\big|\nu(\d z)\bigg)\d s\\&\nonumber\qquad+2\int_0^t\bigg(\fk{a}(s)+2\sqrt{\fk{b}(s)}\|f(s)\|_\H+ 2\int_\Z L_\gamma(s,z)\big|g(s,z)-1\big|\nu(\d z)\bigg)\|\vi{\Y}^{\p}_n(s)\|_\H^2\d s
\\&\nonumber\leq \|\x\|_\H^2+\int_0^t\bigg(2\fk{a}(s)+\fk{b}(s)+\|f(s)\|_\H^2+2\int_\Z L_\gamma(s,z)\big|g(s,z)-1\big|\nu(\d z)\bigg)\d s\\&\qquad+2\int_0^t\bigg(\fk{a}(s)+\fk{b}(s)+\|f(s)\|_\H^2+ 2\int_\Z L_\gamma(s,z)\big|g(s,z)-1\big|\nu(\d z)\bigg)\|\vi{\Y}^{\p}_n(s)\|_\H^2\d s,
\end{align}
An application of Gronwall's inequality in \eqref{CPDE17} yields
\begin{align}\label{CPDE19}\nonumber
&	\sup_{t\in[0,T]}\|\vi{\Y}^{\p}_n(t)\|_\H^2\\&\nonumber \leq \bigg\{ \|\x\|_\H^2+\int_0^T\bigg(2\fk{a}(s)+ \fk{b}(s)+\|f(s)\|_\H^2 +2\int_\Z L_\gamma(t,z)\big|g(s,z)-1\big|\nu(\d z)\bigg)\d s\bigg\}\\&\qquad\times\exp\bigg\{2\int_0^T\bigg( \fk{a}(s)+\fk{b}(s)+\|f(s)\|_\H^2+2\int_\Z L_\gamma(t,z)\big|g(s,z)-1\big|\nu(\d z)\bigg)\d s \bigg\},
\end{align} and the right hand side is finite for all $\p=(f,g)\in\mathbb{S}$ (see Lemma \ref{lemUF} also). Substituting the above inequality \eqref{CPDE19} in \eqref{CPDE17}, we obtain 
\begin{align}\label{CPDE20}\nonumber
	&\sup_{t\in[0,T]}\|\vi{\Y}^{\p}_n(t)\|_\H^2+2L_\A\int_0^T\|\vi{\Y}^{\p}_n(s)\|_\V^\beta\d s \\&\nonumber \leq \bigg\{ \|\x\|_\H^2+\int_0^T\bigg(2\fk{a}(s)+ \fk{b}(s)+\|f(s)\|_\H^2 +2\int_\Z L_\gamma(t,z)\big|g(s,z)-1\big|\nu(\d z)\bigg)\d s\bigg\}\\&\qquad\times\exp\bigg\{4\int_0^T\bigg( \fk{a}(s)+\fk{b}(s)+\|f(s)\|_\H^2+\int_\Z L_\gamma(t,z)\big|g(s,z)-1\big|\nu(\d z)\bigg)\d s \bigg\}.
\end{align}

\vspace{2mm}
\noindent
\textsl{Estimate for time derivative:} For $\bf{\xi}\in \L^\infty(0,T;\H)\cap\L^\beta(0,T;\V)$ for $\beta\in(1,\infty)$,
using the Cauchy-Schwarz, H\"older's and Young's inequalities,  and Hypotheses \ref{hyp1} (H.4), (H.5) and \ref{hyp2} (H.8), we obtain
\begin{align}\label{CPDE22}\nonumber
\int_0^T	&\big|\langle \partial_t \vi{\Y}^{\p}_n(t), \bf{\xi}(t)\rangle \big|\d t \\&\nonumber\leq \int_0^T \big|\langle \A_n(t,\vi{\Y}^{\p}_n(t)),\bf{\xi}(t)\rangle\big|\d t+
\int_0^T	\big|\big(\B_n(t,\vi{\Y}^{\p}_n(t))f(t),\bf{\xi}(t)\big)\big|\d t\\&\nonumber\quad +\int_0^T\bigg| \bigg(\int_\Z\gamma_n(t,\vi{\Y}^{\p}_n(t),z)\big(g(t,z)-1\big)\nu(\d z),\bf{\xi}(t)\bigg)\bigg|\d t \\&\nonumber \leq 
\underbrace{\bigg(\int_0^T\|\A_n(t,\vi{\Y}^{\p}_n(t))\|_{\V^*}^{\frac{\beta}{\beta-1}}\d t\bigg)^{\frac{\beta-1}{\beta}}}_{I_1}\bigg(\int_0^T\|\bf{\xi}(t)\|_\V^\beta\d t\bigg)^{\frac{1}{\beta}}\\&\nonumber\quad+\bigg(\sup_{t\in[0,T]}\|\bf{\xi}(t)\|_\H\bigg)\underbrace{\bigg(\int_0^T\|\B(t,\vi{\Y}^{\p}_n(t))f(t)\|_\H\d t\bigg)}_{I_2}\\&\nonumber\quad +
\bigg(\sup_{t\in[0,T]}\|\bf{\xi}(t)\|_\H\bigg)\underbrace{\bigg(\int_{\Z_T}\|\gamma(t,\vi{\Y}^{\p}_n(t),z)\|_\H\big|g(t,z)-1\big|\nu(\d z)\d t\bigg)}_{I_3}
\\&\nonumber \leq 
\bigg(\int_0^T  \big(\fk{a}(t)+C\|\vi{\Y}^{\p}_n(t))\|_\V^\beta\big)\big(1+\|\vi{\Y}^{\p}_n(t)\|_\H^\alpha\big)  \d t\bigg)^{\frac{\beta-1}{\beta}}\bigg(\int_0^T\|\bf{\xi}(t)\|_\V^\beta\d t\bigg)^{\frac{1}{\beta}}\\&\nonumber\quad+\bigg(\sup_{t\in[0,T]}\|\bf{\xi}(t)\|_\H\bigg)\bigg(\int_0^T\sqrt{\fk{b}(t)}\big(1+\|\vi{\Y}^{\p}_n(t)\|_\H\big)\|f(t)\|_\H\d t\bigg)\\&\nonumber\quad +
\bigg(\sup_{t\in[0,T]}\|\bf{\xi}(t)\|_\H\bigg)\bigg(\int_{\Z_T}L_\gamma(t,z)\big|g(t,z)-1\big|\big(1+\|\vi{\Y}^{\p}_n(t)\|_\H\big)\nu(\d z)\d t\bigg)
\\&\nonumber \leq 
\bigg(1+\sup_{t\in[0,T]}\|\vi{\Y}^{\p}_n(t)\|_\H^\alpha \bigg)^{\frac{\beta-1}{\beta}}\bigg(\int_0^T  \big(\fk{a}(t)+C\|\vi{\Y}^{\p}_n(t))\|_\V^\beta\big) \d t\bigg)^{\frac{\beta-1}{\beta}}\bigg(\int_0^T\|\bf{\xi}(t)\|_\V^\beta\d t\bigg)^{\frac{1}{\beta}}\\&\nonumber\quad+\frac{1}{2}
\bigg(1+\sup_{t\in[0,T]}\|\vi{\Y}^{\p}_n(t)\|_\H \bigg)\bigg(\int_0^T\big(\fk{b}(t)+\|f(t)\|_\H^2\big)\d t\bigg)\bigg(\sup_{t\in[0,T]}\|\bf{\xi}(t)\|_\H\bigg)\\&\quad +
\bigg(1+\sup_{t\in[0,T]}\|\vi{\Y}^{\p}_n(t)\|_\H \bigg)
\bigg(\int_{\Z_T}L_\gamma(t,z)\big|g(t,z)-1\big|\nu(\d z)\d t\bigg)\bigg(\sup_{t\in[0,T]}\|\bf{\xi}(t)\|_\H\bigg).
\end{align}
 In view of the energy estimate \eqref{CPDE20} and Lemma \ref{lemUF}, we find that the right hand of \eqref{CPDE22} is finite and independent of $n$, which implies that
\begin{align}\label{CPDE23}
	\partial_t \vi{\Y}^{\p}_n\in \L^1(0,T;\H)+\L^\frac{\beta}{\beta-1}(0,T;\V^*),
\end{align}for $\beta\in(1,\infty)$. 

\vspace{2mm}
\noindent
\textbf{Step 3:} \textsf{Existence of a weak solution.} 
Using the estimates \eqref{CPDE20} and \eqref{CPDE22}, and an application of the Banach-Alaoglu theorem guarantee  the existence of a subsequence $\{\vi{\Y}^{\p}_n\}$ (still denoting by same index) such that 
\begin{equation}\label{CPDE24}
\left\{	\begin{aligned}
		\vi{\Y}^{\p}_n&\xrightharpoonup{\ast} \vi{\Y}^{\p}, &&\text{ in } \L^\infty(0,T;\H),\\
		\vi{\Y}^{\p}_n &\xrightharpoonup{}\vi{\Y}^{\p}, &&\text{ in } \L^\beta(0,T;\V),\;\beta\in(1,\infty)\\
			\vi{\Y}^{\p}_n(T)&\xrightharpoonup{}\wi{\Y}^{\p}(T), &&\text{ in } \H,\\
			\A_n(\cdot,	\vi{\Y}^{\p}_n(\cdot))&\xrightharpoonup{} \wi{\A}(\cdot), &&\text{ in } \L^\frac{\beta}{\beta-1}(0,T;\V^*),\\
			\B_n(\cdot,	\vi{\Y}^{\p}_n(\cdot))f(\cdot)&\xrightharpoonup{}\wi{\B}(\cdot)f(\cdot), && \text{ in }\L^1(0,T;\H),\\
				\partial_t 	\vi{\Y}^{\p}_n &\xrightharpoonup{} \partial_t\vi{\Y}^{\p}, &&\text{ in } \L^1(0,T;\H)+\L^\frac{\beta}{\beta-1}(0,T;\V^*),
	\end{aligned}
		\right.
	\end{equation}and \begin{align}\label{CPDE024}
	\int_\Z\gamma_n(\cdot,	\vi{\Y}^{\p}_n(\cdot))\big(g(\cdot,z)-1\big)\nu(\d z)\xrightharpoonup{}\int_\Z\wi{\gamma}(\cdot,z)\big(g(\cdot,z)-1\big)\nu(\d z),  \text{ in }\L^1(0,T;\H). 
\end{align} 
The convergences for $\A_n,\B_n$ and $\gamma_n$ are justified by the estimates for $I_1,I_2$ and $I_3$ in \eqref{CPDE22}.  By an application of the Aubin-Lions compactness lemma, we conclude that 
\begin{align}\label{CPDE25}
	\vi{\Y}^{\p}_n \to \vi{\Y}^{\p}, \text{ in } \L^\beta(0,T;\H), \ \text{ for }\ \beta\in(1,\infty). 
\end{align}Furthermore, along a subsequence (still denoting by same index), we have
\begin{align}\label{CPDE26}
\vi{\Y}^{\p}_n\to \vi{\Y}^{\p}, \ \text{ for a.e. }\ t\in[0,T] \ \text{ in }\ \H. 
\end{align}By Proposition 23.23, \cite{EZ} (see Theorem 1.8, pp. 33, \cite{VVCMIV} and Lemma 1.2, pp. 179, \cite{RT}  also), we have the map  $t\mapsto \|\vi{\Y}^{\p}(t)\|_\H^2$ is absolutely continuous with $$\frac{1}{2}\frac{\d}{\d t}\|\vi{\Y}^{\p}(t)\|_\H^2=\langle \partial_t \vi{\Y}^{\p}(t),\vi{\Y}^{\p}(t)\rangle, \ \text{ for a.e. }\ t\in[0,T],$$ which is straightforward since   $\vi{\Y}^{\p} \in \L^\beta(0,T;\V)$ and $\partial_t \vi{\Y}^{\p} \in \L^1(0,T;\H)+\L^\frac{\beta}{\beta-1}(0,T;\V^*)$. 

Now, our aim is to identify the limiting function obtained in the weak convergences \eqref{CPDE24} and \eqref{CPDE024}.

\vspace{2mm}
\noindent
\textbf{Claim 1:} The functions $\vi{\Y}^p,\wi{\A}(\cdot), \wi{\B}(\cdot)$ and $\wi{\gamma}(\cdot,\cdot)$ satisfy $\vi{\Y}^p\in \mathcal{M}:= \{\u\in\L^\infty(0,T;\H)\cap\L^\beta(0,T;\V):\partial_t\u\in\mathrm{L}^1(0,T;\H)+\mathrm{L}^{\frac{\beta}{\beta-1}}(0,T;\V^*)\}$ and 
\begin{align}\label{424}
	\partial_t\vi{\Y}^p(t)=\wi{\A}(t)+\wi{\B}(t)f(t)+\int_\Z\wi{\gamma}(t,z)\big(g(t,z)-1\big)\nu(\d z), \ \text{ for a.e. }\ t\in(0,T), 
\end{align}also $\vi{\Y}^p(0)=\x$, and $ \vi{\Y}^p(T)=\wi{\Y}^p(T)$.

	The proof of Claim 1 for a similar model can be obtained from Lemma 30.5, \cite{EZ} (cf. Lemma 2.4, \cite{WL4} also). We include it here for completeness. We know that $\vi{\Y}^p(\cdot)$ satisfies the following  integration by parts formula (Proposition 1.7.2, \cite{PCAM}):
	\begin{align}\label{425}
		\big(\vi{\Y}^p(T),\u(T)\big)-\big(\vi{\Y}^p(0),\u(0)\big)=\int_0^T \langle \partial_t\vi{\Y}^p(t),\u(t)\rangle \d t+\int_0^T \langle \partial_t\u(t),\vi{\Y}^p(t)\rangle \d t, 
	\end{align}for $\u,\vi{\Y}^p\in  \mathcal{M}$.
	
	For any $\psi\in \C^\infty([0,T])$, $\u\in \H_n$, from  \eqref{CPDE16}, we have 
	\begin{align}\label{IP1}\nonumber
		&\big(\vi{\Y}^p_n(T),\psi(T)\u\big)-\big(\x_0^n,\psi(0)\u\big)\\&\nonumber=\int_0^T \langle \partial_t\vi{\Y}^p_n(t),\psi(t)\u\rangle \d t+\int_0^T \langle \psi'(t)\u,\vi{\Y}^p_n(t)\rangle \d t\\& \nonumber=
		\int_0^T \bigg\langle \A_n(t,\vi{\Y}^p_n(t))+\B_n(t,\vi{\Y}^p_n(t))f(t)+\int_\Z\gamma_n(t,\vi{\Y}^p_n(t),z)\big(g(t,z)-1\big)\nu(\d z),\psi(t)\u\bigg\rangle\d t\\&\quad+\int_0^T \langle \psi'(t)\u,\vi{\Y}^p_n(t)\rangle \d t. 
	\end{align}
Passing $n\to\infty$, we obtain for all $\u\in\bigcup\limits_n\H_n$,
	\begin{align}\label{IP2}\nonumber
		\big(\wi{\Y}^p(T),\psi(T)\u\big)-\big(\x,\psi(0)\u\big)&=\int_0^T\bigg\langle 
		\wi{\A}(t)+\wi{\B}(t)f(t)+\int_\Z\wi{\gamma}(t,z)\big(g(t,z)-1\big)\nu(\d z) ,\psi(t)\u\bigg\rangle \d t\\& \quad + \int_0^T \langle \psi'(t)\u,\vi{\Y}^p(t)\rangle \d t. 
	\end{align}We know that $\bigcup\limits_n\H_n$ is dense in $\V$. Therefore \eqref{IP2} remains valid for all $\u\in\V,\;\psi \in\C^\infty([0,T])$.
	
	If we fix $\psi(T)=\psi(0)=0$, then we have
	\begin{align*}
		\int_0^T\bigg\langle 
		\wi{\A}(t)+\wi{\B}(t)f(t)+\int_\Z\wi{\gamma}(t,z)\big(g(t,z)-1\big)\nu(\d z) ,\u\bigg\rangle\psi(t) \d t=-\int_0^T\langle \vi{\Y}^p(t),\u\rangle \psi'(t)\d t,
	\end{align*} and it  gives \eqref{424} and $\vi{\Y}^p\in  \mathcal{M}$. Using integration by parts formula \eqref{425}, we obtain 
	\begin{align}\label{IP3}\nonumber
		&\big(\vi{\Y}^p(T),\psi(T)\u\big)-\big(\vi{\Y}^p(0),\psi(0)\u\big)\\&\nonumber=\int_0^T \langle \partial_t\vi{\Y}^p(t),\psi(t)\u\rangle \d t+\int_0^T \langle \psi'(t)\u,\vi{\Y}^p(t)\rangle \d t\\& =
		\int_0^T \bigg\langle \wi{\A}(t)+\wi{\B}(t)f(t)+\int_\Z\wi{\gamma}(t,z)\big(g(t,z)-1\big)\nu(\d z),\psi(t)\u\bigg\rangle\d t+\int_0^T \langle \psi'(t)\u,\vi{\Y}^p(t)\rangle \d t,
	\end{align}  for all $\u\in\V,\;\psi \in\C^\infty([0,T])$.
 Using \eqref{IP2}, we find 
	\begin{align*}
		&\big(\vi{\Y}^p(T),\psi(T)\u\big)-\big(\vi{\Y}^p(0),\psi(0)\u\big)= 	\big(\wi{\Y}^p(T),\psi(T)\u\big)-\big(\x,\psi(0)\u\big).
	\end{align*}Then, by choosing $\psi(T)=1,\;\psi(0)=0$ and $\psi(T)=0,\;\psi(0)=1$, respectively, we deduce 
	\begin{align*}
		\vi{\Y}^p(T)=\wi{\Y}^p(T), \ \text{ and }\ \vi{\Y}^p(0)=\x,
	\end{align*}hence the proof of Claim 1 is over.

\vspace{2mm}
\noindent
\textbf{Claim 2:} $\wi{\A}(\cdot)=\A(\cdot,\vi{\Y}^p(\cdot))$, as an element in $\L^\frac{\beta}{\beta-1}(0,T;\V^*)$, for $\beta\in(1,\infty)$.

Using the integration by parts formula, we get 
\begin{align*}
	&	\|\vi{\Y}^p_n(T)\|_\H^2-\|\x_0^n\|_\H^2 \\&= 2\int_0^T\bigg\langle\A_n(t,\vi{\Y}^p_n(t))+\B_n(t,\vi{\Y}^p_n(t))f(t)+\int_\Z\gamma_n(t,\vi{\Y}^p_n(t),z)\big(g(t,z)-1\big)\nu(\d z),\vi{\Y}^p_n(t)\bigg\rangle \d t ,
\end{align*}and 
\begin{align*}
	\|\vi{\Y}^p(T)\|_\H^2-\|\x\|_\H^2 &= 2\int_0^T\bigg\langle\wi{\A}(t)+\wi{\B}(t)f(t)+\int_\Z\wi{\gamma}(t,z)\big(g(t,z)-1\big)\nu(\d z),\vi{\Y}^p(t)\bigg\rangle \d t .
\end{align*}	
Making use of the convergence \eqref{CPDE24} and the lower semicontinuity of $\H$-norm, we obtain 
\begin{align*}
\|\vi{\Y}^p(T)\|_\H^2= \|\wi{\Y}^p(T)\|_\H^2\leq \liminf_{n\to\infty}\|\vi{\Y}^p_n(T)\|_\H^2.
\end{align*}
Thus, we have 
\begin{align}\label{429}
	&\liminf_{n\to\infty} \int_0^T \langle\A_n(t,\vi{\Y}^p_n(t)),\vi{\Y}^p_n(t)\rangle \d t \nonumber\\&\geq \frac{1}{2}\big(	\|\vi{\Y}^p(T)\|_\H^2-\|\x\|_\H^2\big)-\int_0^T\bigg(\wi{\B}(t)f(t)+\int_\Z\wi{\gamma}(t,z)\big(g(t,z)-1\big)\nu(\d z),\vi{\Y}^p(t)\bigg)\d t\nonumber \\& = \int_0^T \langle \wi{\A}(t),\vi{\Y}^p(t)\rangle \d t.
\end{align}
From the convergences \eqref{CPDE24} and \eqref{429}, one can apply  Lemma \ref{lemmaMT2} to find  for any $\u\in\L^\beta(0,T;\V)$
\begin{align*}
	\int_0^T \langle\A(t,\vi{\Y}^p(t)),\vi{\Y}^p(t)-\u(t)\rangle \d t &\geq \limsup_{n\to\infty}\int_0^T\langle \A_n(t,\vi{\Y}^p_n(t)), \vi{\Y}^p_n(t)-\u(t)\rangle \d t \\&\geq	\liminf_{n\to\infty}\int_0^T\langle \A_n(t,\vi{\Y}^p_n(t)), \vi{\Y}^p_n(t)-\u(t)\rangle \d t 
	\\&\geq 	\int_0^T \langle \wi{\A}(t),\vi{\Y}^p(t)-\u(t)\rangle\d t. 
\end{align*}
Due to the arbitrary choice of $\u\in\L^\beta(0,T;\V)$, and uniqueness of the limit,  we get $\A(\cdot,\vi{\Y}^p(\cdot))=\wi{\A}(\cdot)$ as an element in the space $\L^\frac{\beta}{\beta-1}(0,T;\V^*)$.

\vspace{2mm}
\noindent
\textbf{Claim 3:} $\wi{\B}(\cdot)f(\cdot)=\B(\cdot,\vi{\Y}^p(\cdot))f(\cdot)$ and $\wi{\gamma}(\cdot,z)\big(g(\cdot,z)-1\big)=\gamma(\cdot,\vi{\Y}^p(\cdot),z)\big(g(\cdot,z)-1\big)$ in $\L^1(0,T;\H)$:

Our first aim is to prove the following:
\begin{align}\label{B1}
	\lim_{n\to\infty}\int_0^T\big\|\big(\B_n(t,\vi{\Y}^{\p}_n(t))-\B(t,\Y^{\p}(t)\big)f(t)\big\|_\H\d t=0.
\end{align} For any $\delta>0$, define $A_{\delta,n}=\{t\in[0,T]:\|\vi{\Y}^{\p}_{n'}(t)-\vi{\Y}^{\p}(t)\|_\H>\delta\}$. Then  by using  the measure theoretic version of Markov's inequality and the strong convergence \eqref{CPDE25}, we have 
\begin{align}\label{B2}
	\lim_{n\to\infty}\lambda_T\big(A_{\delta,n}\big)\leq \lim_{n\to\infty}\frac{\int_0^T\|\vi{\Y}^{\p}_{n}(t)-\vi{\Y}^{\p}(t)\|_\H^\beta\d t}{\delta^\beta}=0.
\end{align}Let us fix $K=\sup\limits_{n\in\N}\sup\limits_{t\in[0,T]}\|\vi{\Y}^{\p}_{n}(t)\|_\H\vee \sup\limits_{t\in[0,T]}\|\vi{\Y}^{\p}(t)\|_\H<\infty$.	We consider 
\begin{align}\label{B02}\nonumber
	&	\int_0^T\|(\mathbf{P}_{n}\B(t,\vi{\Y}_{n}^{\p}(t))-\B(t,\vi{\Y}^{\p}(t)))f(t)\|_\H\d t \\&\leq \int_0^T\big\|\mathbf{P}_{n}\big(\B(t,\vi{\Y}^{\p}_{n}(t))-\B(t,\Y^{\p}(t)\big)f(t)\big\|_\H\d t  +\int_0^T\big\|(\I-\mathbf{P}_{n})\B(t,\vi{\Y}^{\p}_{n}(t))f(t)\big\|_\H\d t .
\end{align}We consider the first term from the right hand side of the above inequality \eqref{B02} and estimate it using  Hypothesis \ref{hyp2} (H.7), \eqref{B2}, H\"older's and Young's inequalities as
\begin{align}\label{B3}\nonumber
	&	\int_0^T\big\|\mathbf{P}_{n}\big(\B(t,\vi{\Y}^{\p}_{n}(t))-\B(t,\Y^{\p}(t)\big)f(t)\big\|_\H\d t  \\&\nonumber\leq \int_0^T\|\B(t,\vi{\Y}^{\p}_{n}(t))-\B(t,\Y^{\p}(t)\|_{\L_2}\|f(t)\|_\H\d t \\&\nonumber\leq \int_0^T\sqrt{L_\B(t)}\|\vi{\Y}^{\p}_{n}(t)-\vi{\Y}^{\p}(t)\|_\H\|f(t)\|_\H\d t \\&\nonumber \leq \int_{A_{\delta,n}}\sqrt{L_\B(t)}\|f(t)\|_\H\big(\|\vi{\Y}^{\p}_{n}(t)\|_\H+\|\vi{\Y}^{\p}(t)\|_\H\big)\d t+\delta\int_{A_{\delta,n}^c}\sqrt{L_\B(t)}\|f(t)\|_\H\d t \\& \leq 
	2K \bigg(\int_{A_{\delta,n}}L_\B(t)\d t\bigg)^{\frac{1}{2}}\bigg(\int_0^T\|f(t)\|_\H^2\d t\bigg)^{\frac{1}{2}}+\delta\bigg(\int_0^TL_\B(t)\d t+\int_0^T\|f(t)\|_\H^2\d t\bigg). 
\end{align}
Note that the function $L_\B\in\L^1(0,T;\R^+)$. Using the absolute continuity of the Lebesgue integral, we obtain 
\begin{align}\label{B4}
	\lim_{n\to\infty}\int_{A_{\delta,n}}L_\B(t)\d t=0.
\end{align}
To estimate  the final term in the right hand side of \eqref{B02}, we use Hypothesis \ref{hyp2} (H.7) and H\"older's inequality as follows:
\begin{align}\label{B03}\nonumber
	&	\int_0^T\big\|(\I-\mathbf{P}_{n})\B(t,\vi{\Y}^{\p}_{n}(t))f(t)\big\|_\H\d t \\&\nonumber \leq  \|\I-\mathbf{P}_{n}\|_{\mathcal{L}(\H)} \int_0^T\|\B(t,\vi{\Y}^{\p}_{n}(t))f(t)\big\|_\H\d t \\&\nonumber\leq \|\I-\mathbf{P}_{n}\|_{\mathcal{L}(\H)}  \bigg(\int_0^T \|\B(t,\vi{\Y}^{\p}_{n}(t))\|_{\L_2}^2\d t\bigg)^{\frac{1}{2}}\bigg(\int_0^T \|f(t)\|_\H^2\d t\bigg)^{\frac{1}{2}} \\&\nonumber\leq   \|\I-\mathbf{P}_{n}\|_{\mathcal{L}(\H)}  \bigg(\int_0^T L_\B(t)\big(1+\|\vi{\Y}^{\p}_{n}(t)\big)\|_\H^2)\d t\bigg)^{\frac{1}{2}}\bigg(\int_0^T \|f(t)\|_\H^2\d t\bigg)^{\frac{1}{2}} \\&\leq  (1+K) \|\I-\mathbf{P}_{n}\|_{\mathcal{L}(\H)}  \bigg(\int_0^T L_\B(t)\d t\bigg)^{\frac{1}{2}}\bigg(\int_0^T \|f(t)\|_\H^2\d t\bigg)^{\frac{1}{2}} .
\end{align}
Substituting \eqref{B3}-\eqref{B03} in \eqref{B02}, we obtain the required result \eqref{B1} by passing $n\to\infty$ and the arbitrariness of $\delta$. 

 Our final aim of this Claim is   to prove the following:
\begin{align}\label{G1}
&\lim_{n\to\infty}	\int_{\Z_T} \big\|\big(\gamma_{n}(t,\vi{\Y}^{\p}_{n}(t),z)-\gamma(t,\vi{\Y}^{\p}(t),z)\big)\big(g(t,z)-1\big)\big\|_\H\nu(\d z)\d t=0.
\end{align}We consider 
\begin{align}\label{G01}\nonumber
	&	\int_{\Z_T} \big\|\big(\mathbf{P}_{n}\gamma(t,\vi{\Y}^{\p}_{n}(t),z)-\gamma(t,\vi{\Y}^{\p}(t),z)\big)\big(g(t,z)-1\big)\big\|_\H\nu(\d z)\d t	
	\\& \nonumber\leq \int_{\Z_T} \big\|\mathbf{P}_{n}\big(\gamma(t,\vi{\Y}^{\p}_{n}(t),z)-\gamma(t,\vi{\Y}^{\p}(t),z)\big)\big(g(t,z)-1\big)\big\|_\H\nu(\d z)\d t\\&\quad + 	\int_{\Z_T} \big\|(\I-\mathbf{P}_{n})\gamma(t,\vi{\Y}^{\p}_{n}(t),z)\big(g(t,z)-1\big)\big\|_\H\nu(\d z)\d t.
\end{align}Using Hypothesis \ref{hyp2} (H.9) and the definition of $A_{\delta,n}$, we estimate the first term of the right hand side of the above inequality \eqref{G01} as
\begin{align}\label{G2}\nonumber&
	\int_{\Z_T} \big\|\mathbf{P}_{n}\big(\gamma(t,\vi{\Y}^{\p}_{n}(t),z)-\gamma(t,\vi{\Y}^{\p}(t),z)\big)\big(g(t,z)-1\big)\big\|_\H\nu(\d z)\d t
	 \\& \nonumber\leq 
	\int_{\Z_T}\|\gamma(t,\vi{\Y}^{\p}_{n}(t),z)-\gamma(t,\vi{\Y}^{\p}(t),z)\|_\H\big|g(t,z)-1\big|\nu(\d z)\d t \\&\nonumber\leq \int_{\Z_T} R_\gamma(t,z)\|\vi{\Y}^{\p}_{n}(t)-\vi{\Y}^{\p}(t)\|_\H\big|g(t,z)-1\big|\nu(\d z)\d t \\& \leq \delta\int_{A_{\delta,n}^c}\int_\Z R_\gamma(t,z)\big|g(t,z)-1\big|\nu(\d z)\d t +2K\int_{A_{\delta,n}}\int_\Z R_\gamma (t,z)\big|g(t,z)-1\big|\nu(\d z)\d t. 
\end{align}To estimate  the final term in the right hand side of \eqref{G01}, we use Hypothesis \ref{hyp2} (H.8) and H\"older's inequality as follows:
\begin{align}\label{G02}\nonumber
	&	\int_{\Z_T} \big\|(\I-\mathbf{P}_{n})\gamma(t,\vi{\Y}^{\p}_{n}(t),z)\big(g(t,z)-1\big)\big\|_\H\nu(\d z)\d t \\&\nonumber\leq \|\I-\mathbf{P}_{n}\|_{\mathcal{L}(\H)}\int_0^T \|\gamma(t,\vi{\Y}^{\p}_{n}(t),z)\|_\H \big|g(t,z)-1\big|\nu(\d z)\d t \\&\leq 
	 (1+K)\|\I-\mathbf{P}_{n}\|_{\mathcal{L}(\H)}\int_0^TL_\gamma(t,z)\big|g(t,z)-1\big|\nu(\d z)\d t .
\end{align}Substituting \eqref{G2} and \eqref{G02} in \eqref{G01}, passing $n\to\infty$  and applying Lemma \ref{lemUF3} (3), we obtain the required result  \eqref{G1}. Due to the uniqueness of limits, we find $\wi{\B}(\cdot)f(\cdot)=\B(\cdot,\vi{\Y}^{\p}(\cdot))f(\cdot)$ and $\wi{\gamma}(\cdot,z)\big(g(\cdot,z)-1\big)=\gamma(\cdot,\vi{\Y}^{\p}(\cdot),z)\big(g(\cdot,z)-1\big)$ in $\L^1(0,T;\H)$.

 Using Claims 1, 2 and 3, we can pass the limit in \eqref{CPDE16} to obtain the following   system for a.e. $t\in[0,T]$: 
\begin{equation}\label{CPDE27}
	\left\{ 	\begin{aligned}
		\langle \partial_t \vi{\Y}^{\p}(t), \bf{\phi}\rangle& =-\langle\cc(t,\vi{\Y}^{\p}(t)),\bf{\phi}\rangle,\\
		\big(\vi{\Y}^{\p}(0),\bf{\phi}\big)&=\big(\x,\bf{\phi}\big),
	\end{aligned}
	\right.
\end{equation}for all $\bf{\phi}\in\V$.  
Moreover, we have the following  energy equality:
\begin{align}\label{CPDE29}
	\|\vi{\Y}^{\p}(t)\|_\H^2=\|\x\|_\H^2-2\int_0^t\langle \cc(\vi{\Y}^{\p}(s)),\vi{\Y}^{\p}(s)\rangle\d s,
\end{align}for all $t\in[0,T]$. 
Therefore, the equation \eqref{CPDE09}  has a weak solution $\vi{\Y}^{\p}\in \L^\infty(0,T;\H)\cap\L^\beta(0,T;\V)$, for $\beta\in(1,\infty)$, such that the following energy equality satisfied:
\begin{align}\label{CPDE39}\nonumber
		\|\vi{\Y}^{\p}(t)\|_\H^2 &\nonumber=\|\x\|_\H^2+2\int_0^t \bigg\{\langle\A(t,\vi{\Y}^{\p}(s)),\vi{\Y}^{\p}(s)\rangle+\big(\B(s,\vi{\Y}^{\p}(s))f(s),\vi{\Y}^{\p}(s)\big)\\&\qquad +\bigg(\int_\Z\gamma(t,\vi{\Y}^{\p}(s),z)\big(g(s,z)-1\big)\nu(\d z),\vi{\Y}^{\p}(s)\bigg)\bigg\}\d s ,
	\end{align}for all $t\in[0,T]$.

\vspace{2mm}
\noindent
\textbf{Step 4:} \textsf{Uniqueness of weak solution.} Let $\vi{\Y}^{\p}_1(\cdot)$ and $\vi{\Y}^{\p}_2(\cdot)$ be any two weak solution of \eqref{CPDE09} with the same initial data $\x\in \H$. From \eqref{CPDE09}, for all $t\in[0,T],$ we have
\begin{align*}
	\big(\vi{\Y}^{\p}_1(t)-\vi{\Y}^{\p}_1(t), \bf{\phi}\big)=-\int_0^t\langle \cc(\vi{\Y}^{\p}_1(s))-\cc(\vi{\Y}^{\p}_2(s)),\bf{\phi}\rangle\d s, \ \text{ for all } \ \bf{\phi}\in\V.
\end{align*}We know that the energy equality \eqref{CPDE29} holds, and using the local monotoncity condition \eqref{CPDE11}, we obtain 
\begin{align}\label{CPDE40}\nonumber
	\|\vi{\Y}^{\p}_1(t)-\vi{\Y}^{\p}_2(t)\|_\H^2 &=-2\int_0^t\langle \cc(\vi{\Y}^{\p}_1(s))-\cc(\vi{\Y}^{\p}_2(s)),\vi{\Y}^{\p}_1(s)-\vi{\Y}^{\p}_2(s)\rangle\d s \\&\nonumber\leq \int_0^t\bigg\{\big(\fk{a}(s)+\rho(\vi{\Y}^{\p}_1(s))+\eta(\vi{\Y}^{\p}_2(s))\big)+2\|f(s)\|_\H^2\\&\qquad+\int_\Z R_\gamma(s,z)\big|g(s,z)-1\big|\nu(\d z)\bigg\}\|\vi{\Y}^{\p}_1-\vi{\Y}^{\p}_2(s)\|_\H^2\d s,
\end{align}for all $t\in[0,T]$. Using the fact that $\vi{\Y}^{\p}_1,\vi{\Y}^{\p}_2\in \L^\infty(0,T;\H)\cap\L^\beta(0,T;\V)$ for $\beta\in(1,\infty)$,  $\fk{a}\in\L^1(0,T;\R^+)$, $f\in\L^2(0,T;\H)$ and Lemma \ref{lemUF}, the uniqueness is straightforward by an application of Gronwall's inequality in \eqref{CPDE40}.
\end{proof} 
Let us define 
\begin{align*}
	\mathscr{G}^0\left(\int_0^{\cdot} f (s)\d s,\nu_T^g\right)=\vi{\Y}^{\p},\ \text{ for }\ \p=(f,g)\in\bar{S}^\Upsilon,
\end{align*}
where $\vi{\Y}^{\p}(\cdot)$  is the unique weak solution of \eqref{SE2} obtained in Theorem \ref{CPDE1}.  Now, we are in a position to verify Condition \ref{Cond1} (i).
\begin{proposition}\label{VCond1}
Fix $\Upsilon\in\mathbb{N}$.	Let $\p_n=(f_n,g_n),\ \p=(f,g)\in \bar{S}^\Upsilon$ be such that $\p_n\to\p$ as $n\to\infty$. Then we have 
	\begin{align}\label{Vcond01}
		\mathscr{G}^0\left(\int_0^{\cdot} f_n (s)\d s,\nu_T^{g_n}\right)\to\mathscr{G}^0\left(\int_0^{\cdot} f (s)\d s,\nu_T^g\right)\ \text{ as } \ n\to\infty,
	\end{align}in $\C([0,T];\H)$.
\end{proposition}
\begin{proof}
Let us fix $	\mathscr{G}^0\left(\int_0^{\cdot} f_n (s)\d s,\nu_T^{g_n}\right):=\vi{\Y}^{\p_n}$. Then, $\vi{\Y}^{\p_n}(\cdot)$ satisfies the following deterministic control problem:
\begin{equation}\label{VCond02}
	\left\{
	\begin{aligned}
		\d \vi{\Y}^{\p_n}(t)&= \bigg[\A(t,\vi{\Y}^{\p_n}(t))+\B(t,\vi{\Y}^{\p_n}(t))f_n(t)+\int_\Z\gamma(t,\vi{\Y}^{\p_n}(t),z)\big(g_n(t,z)-1\big)\nu(\d z)\bigg]\d t,\\
		\vi{\Y}^{\p_n}(0)&=\x\in \H.
	\end{aligned}
\right.
\end{equation}Since $\p_n\in\bar{S}^\Upsilon$, using the energy estimate \eqref{CPDE01} and the Banach-Alaoglu theorem, we obtain the following convergence along a subsequence  (which is denoted by the same)
\begin{equation}\label{VCond03}
	\left\{
	\begin{aligned}
		\vi{\Y}^{\p_n} &\xrightharpoonup{\ast} 	\vi{\Y}, &&\text{ in } \L^\infty(0,T;\H),\\
			\vi{\Y}^{\p_n} &\xrightharpoonup{} 	\vi{\Y}, &&\text{ in } \L^\beta(0,T;\V), \\
	\partial_t	\vi{\Y}^{\p_n} &\xrightharpoonup{} 	\partial_t\vi{\Y}, &&\text{ in } \L^1(0,T;\H)+\L^{\frac{\beta}{\beta-1}}(0,T;\V)	, \ \text{ for }\ \beta\in(1,\infty),
		\end{aligned}
\right.
\end{equation}as $n\to\infty$. Using the Aubin-Lions compactness result, we obtain the strong convergence along a subsequence (still denoted by the same)
\begin{align}\label{VCond04}
		\vi{\Y}^{\p_n} \to 	\vi{\Y}, \text{ in } \L^\beta(0,T;\H), \ \text{ for } \ \beta\in(1,\infty),
\end{align}
and along a further subsequence 
\begin{align}\label{448}
	\|\vi{\Y}^{\p_n}(t)\|_{\H}\to \|\vi{\Y}(t)\|_{\H},\ \text{ for a.e. } \ t\in[0,T], 
\end{align}
as $n\to\infty$. 
Since $\p=(f,g)\in \bar{S}^\Upsilon$, there exists a unique weak solution $\vi{\Y}^{\p}=	\mathscr{G}^0\left(\int_0^{\cdot} f (s)\d s,\nu_T^g\right)$ of the problem \eqref{SE2}. 

\vspace{2mm}
\noindent
\textbf{Step 1:} \textsf{To prove $\vi{\Y}=\vi{\Y}^{\p}$.} 
Let us choose a continuously differentiable function $\Psi$ on $[0,T]$ with $\Psi(T)=0$. Taking the inner product  with $\Psi(t)\bfe_j$ to the first equation in \eqref{VCond02}, and then integrating by parts, we find 
\begin{align}\label{VC1}\nonumber
	-&\int_0^T \big(\vi{\Y}^{\p_n}(t),\Psi'(t)\bfe_j\big)\d t\\&\nonumber=\big(\x,\Psi(0)\bfe_j\big)+\int_0^T\langle \A(t,\vi{\Y}^{\p_n}(t)),\Psi(t)\bfe_j\rangle \d t + \int_0^T \big(\B(t,\vi{\Y}^{\p_n}(t))f_n(t),\Psi(t)\bfe_j\big) \d t\\&\quad+ \int_0^T \bigg(\int_\Z \gamma(t,\vi{\Y}^{\p_n}(t),z)\big(g_n(t,z)-1\big)\nu(\d z),\Psi(t)\bfe_j\bigg)\d t.
\end{align}
From \eqref{CPDE24} and Step 3, Claim 2 in the proof of Theorem \ref{CPDE1}, we easily have 
\begin{align}\label{450}
	\lim_{n\to\infty}\int_0^T\langle \A(t,\vi{\Y}^{\p_n}(t)),\Psi(t)\bfe_j\rangle \d t=\int_0^T\langle \A(t,\vi{\Y}(t)),\Psi(t)\bfe_j\rangle \d t.
\end{align}
The convergence of the penultimate term in the right hand side of the equation \eqref{VC1} can be justified as follows:
\begin{align}\label{VC2}\nonumber
&	\bigg|\int_0^T\big( \B(t,\vi{\Y}^{\p_n}(t))f_n(t),\Psi(t)\bfe_j\big) \d t-\int_0^T\big( \B(t,\vi{\Y}(t))f(t),\Psi(t)\bfe_j\big) \d t\bigg|\\&\nonumber\nonumber \leq 
	\bigg|\int_0^T\big( (\B(t,\vi{\Y}^{\p_n}(t))-\B(t,\vi{\Y}(t)))f_n(t),\Psi(t)\bfe_j\big) \d t\bigg|\\&\quad+\bigg|\int_0^T\big( \B(t,\vi{\Y}(t))(f_n(t)-f(t)),\Psi(t)\bfe_j\big) \d t\bigg|. 
\end{align}Let us take $C_\Psi=\sup\limits_{t\in[0,T]}|\Psi(t)|$.  Now, we consider the first term in the right hand side of the  inequality \eqref{VC2} and estimate it using Hypothesis \ref{hyp2} (H.7) as 
\begin{align}\label{VC3}\nonumber
	&	\bigg|\int_0^T\big( (\B(t,\vi{\Y}^{\p_n}(t))-\B(t,\vi{\Y}(t)))f_n(t),\Psi(t)\bfe_j\big) \d t\bigg|\\&\nonumber\leq \sup_{t\in[0,T]}|\Psi(t)| \int_0^T\sqrt{ \L_B(t)}\| \vi{\Y}^{\p_n}(t)-\vi{\Y}(t)\|_\H \|f_n(t)\|_\H \d t \\&\nonumber\leq 
	C_{\Psi}\int_0^T \big(L_\B(t)+\|f_n(t)\|_\H^2\big)\|\vi{\Y}^{\p_n}(t)-\vi{\Y}(t)\|_\H\d t\nonumber\\&\to 0 \text{ as } n\to\infty,
\end{align}where we have used the   convergence \eqref{448}, the fact that $\|f_n(t)\|_{\H}\to\|f(t)\|_{\H}$, for a.e. $t\in[0,T]$  and Dominated Convergence Theorem. Note that 
\begin{align*}
&	\int_0^T \big(L_\B(t)+\|f_n(t)\|_\H^2\big)\|\vi{\Y}^{\p_n}(t)-\vi{\Y}(t)\|_\H\d t\nonumber\\&\leq \sup_{t\in[0,T]}\left(\|\vi{\Y}^{\p_n}(t)\|_{\H}+\|\vi{\Y}(t)\|_\H\right)\left(\int_0^T L_\B(t)\d t+\Upsilon\right)\leq C. 
\end{align*}
The final term of the right hand side of  \eqref{VC2} can be estimates as follows:
\begin{align}\label{VC4}\nonumber
\bigg|\int_0^T\big( \B(t,\vi{\Y}(t))(f_n(t)-f(t)),\Psi(t)\bfe_j\big) \d t\bigg| &\leq C_\Psi \int_0^T\|\B(t,\vi{\Y}(t))(f_n(t)-f(t))\|_\H\d t \\&\to 0\  \text{ as }\ n\to\infty,
\end{align}where we have used the fact that  the operator $\B(\cdot,\vi{\Y}(\cdot))$ is Hilbert-Schmidt in $\H$ and hence it is compact in $\H$ and we know that a compact operator maps weakly convergent sequences into  strongly  convergent sequences. One can show the above convergence in the following way also:
\begin{align*}
&	\int_0^T\|\B(t,\vi{\Y}(t))(f_n(t)-f(t))\|_\H\d t\nonumber\\& \leq \bigg(1+\sup_{t\in[0,T]}\|\vi{\Y}(t)\|_{\H}\bigg)\left( \int_0^T\mathfrak{b}(t)\d t\right)^{1/2}\left(\int_0^T\|f_n(t)-f(t)\|_{\H}^2\d t\right)^{1/2}\nonumber\\&\to 0 \ \text{ as }\ n\to\infty,
\end{align*}
where we have used Hypothesis \ref{hyp1} (H.5) and H\"older's inequality. 

The convergence of the final term in the right hand side of the equation \eqref{VC1} can be justified as follows:
\begin{align}\label{VC5}\nonumber
&\bigg|	\int_0^T \bigg(\int_\Z \gamma(t,\vi{\Y}^{\p_n}(t),z)\big(g_n(t,z)-1\big)\nu(\d z),\Psi(t)\bfe_j\bigg)\d t\\&\nonumber\quad -	\int_0^T \bigg(\int_\Z \gamma(t,\vi{\Y}(t),z)\big(g(t,z)-1\big)\nu(\d z),\Psi(t)\bfe_j\bigg)\d t\bigg|\\&\nonumber\leq  \bigg|
	\int_0^T \bigg(\int_\Z \big(\gamma(t,\vi{\Y}^{\p_n}(t),z)-\gamma(t,\vi{\Y}(t),z)\big)\big(g_n(t,z)-1\big)\nu(\d z),\Psi(t)\bfe_j\bigg)\d t\bigg|\\&\quad +
	\bigg|
	\int_0^T \bigg(\int_\Z \gamma(t,\vi{\Y}(t),z)\big((g_n(t,z)-1)-(g(t,z)-1)\big)\nu(\d z),\Psi(t)\bfe_j\bigg)\d t\bigg|  .
\end{align}	For any $\delta>0$, define $A_{\delta,n}=\{t\in[0,T]:\|\vi{\Y}^{\p_n}(t)-\vi{\Y}(t)\|_\H>\delta\}$. Then,  using a similar argument as in \eqref{B2}, we conclude
\begin{align}\label{VC6}
\lim_{n\to\infty}\lambda_T\big(A_{\delta,n}\big)=0.
\end{align}Let us fix $K=\sup\limits_{n\in\N}\sup\limits_{t\in[0,T]}\|\vi{\Y}^{\p_n}(t)\|_\H\vee \sup\limits_{t\in[0,T]}\|\vi{\Y}(t)\|_\H<\infty$. Now, we consider the penultimate term in the right hand side of the  inequality \eqref{VC5} and estimate it using the fact that $g_n\in S^\Upsilon$ and Hypothesis \ref{hyp2} (H.9) as 
\begin{align}\label{VC7}\nonumber
	& \bigg|
	\int_0^T \bigg(\int_\Z \big(\gamma(t,\vi{\Y}^{\p_n}(t),z)-\gamma(t,\vi{\Y}(t),z)\big)\big(g_n(t,z)-1\big)\nu(\d z),\Psi(t)\bfe_j\bigg)\d t\bigg| \\&\nonumber\leq \int_{\Z_T} R_\gamma(t,z)\|\vi{\Y}^{\p_n}(t)-\vi{\Y}(t)\|_\H \big|g_n(t,z)-1\big||\Psi(t)|\nu(\d z)\d t \\&\nonumber\leq C_\Psi \int_{\Z_T} R_\gamma(t,z)\|\vi{\Y}^{\p_n}(t)-\vi{\Y}(t)\|_\H \big|g_n(t,z)-1\big|\nu(\d z)\d t \\&\leq \delta C_\Psi \int_{A_{\delta,n}^c}\int_\Z R_\gamma(t,z)\big|g_n(t,z)-1\big|\nu(\d z)\d t +2K C_\Psi \int_{A_{\delta,n}}\int_\Z R_\gamma (t,z)\big|g_n(t,z)-1\big|\nu(\d z)\d t.
\end{align}
Using the fact that $g_n,g\in S^\Upsilon$, the definition of $A_{\delta,n}$ and Hypothesis \ref{hyp2} (H.8), we estimate the final term of \eqref{VC5} as
\begin{align}\label{VC8}\nonumber
	&\bigg|
	\int_0^T \bigg(\int_\Z \gamma(t,\vi{\Y}(t),z)\big((g_n(t,z)-1)-(g(t,z)-1)\big)\nu(\d z),\Psi(t)\bfe_j\bigg)\d t\bigg|  \\&\nonumber\leq \int_{\Z_T}L_\gamma(t,z)\|\vi{\Y}(t)\|_\H \big|(g_n(t,z)-1)-(g(t,z)-1)\big||\Psi(t)|\nu(\d z)\d t \\&\nonumber\leq C_\Psi \int_{\Z_T}L_\gamma(t,z)\|\vi{\Y}(t)\|_\H \big|(g_n(t,z)-1)-(g(t,z)-1)\big|\nu(\d z)\d t  \\&\nonumber\leq \delta  C_\Psi \int_{A_{\delta,n}^c}\int_\Z L_\gamma(t,z) \big|(g_n(t,z)-1)-(g(t,z)-1)\big|\nu(\d z)\d t\\&\nonumber\quad+ K C_\Psi \int_{A_{\delta,n}}\int_\Z L_\gamma(t,z)\big|(g_n(t,z)-1)-(g(t,z)-1)\big|\nu(\d z)\d t 
	\\&\nonumber \leq C_\Psi \bigg\{\delta \int_{A_{\delta,n}^c}\int_\Z L_\gamma(t,z) \big|g_n(t,z)-1\big|\nu(\d z)\d t+ K  \int_{A_{\delta,n}}\int_\Z L_\gamma(t,z)\big|g_n(t,z)-1\big|\nu(\d z)\d t \bigg\}\\&\quad + C_\Psi \bigg\{\delta \int_{A_{\delta,n}^c}\int_\Z L_\gamma(t,z) \big|g(t,z)-1\big|\nu(\d z)\d t+ K  \int_{A_{\delta,n}}\int_\Z L_\gamma(t,z)\big|g(t,z)-1\big|\nu(\d z)\d t \bigg\}.
\end{align}Substituting \eqref{VC7} and \eqref{VC8} in \eqref{VC5},  passing $n\to\infty$ and applying  Lemma \ref{lemUF3} (3),  we obtain
\begin{align}\label{VC08}\nonumber
		&\bigg|	\int_0^T \bigg(\int_\Z \gamma(t,\vi{\Y}^{\p_n}(t),z)\big(g_n(t,z)-1\big)\nu(\d z),\Psi(t)\bfe_j\bigg)\d t\\&\quad -	\int_0^T \bigg(\int_\Z \gamma(t,\vi{\Y}(t),z)\big(g(t,z)-1\big)\nu(\d z),\Psi(t)\bfe_j\bigg)\d t\bigg| \to 0 \ \text{ as } \ n\to\infty,
\end{align}since $\delta>0$ is arbitrary. Combining \eqref{450}-\eqref{VC4} and  \eqref{VC08} with \eqref{VC1}, we deduce 
\begin{align}\label{VC9}\nonumber
	-&\int_0^T \big(\vi{\Y}(t),\Psi'(t)\bfe_j\big)\d t\\&\nonumber=\big(\x,\Psi(0)\bfe_j\big)+\int_0^T\langle \A(t,\vi{\Y}(t)),\Psi(t)\bfe_j\rangle \d t + \int_0^T \big(\B(t,\vi{\Y}(t))f(t),\Psi(t)\bfe_j\big) \d t\\&\quad+ \int_0^T \bigg(\int_\Z \gamma(t,\vi{\Y}(t),z)\big(g(t,z)-1\big)\nu(\d z),\Psi(t)\bfe_j\bigg)\d t.
\end{align}
Now, one can conclude by using similar arguments as in the proof of Theorem 3.1, \cite{RT} (see Section 3, Chapter III). Since the equality \eqref{VC9} holds for any linear combination of $\{\bfe_1,\bfe_2,\ldots\}$, therefore it is holds for any $\v\in\V$ also by a continuity argument. By choosing $\Psi\in\mathcal{D}([0,T))$ (the space of test functions over $[0,T)$), we find that $\vi{\Y}(\cdot)$ satisfies the equation \eqref{VC9} in the sense of distributions.   Therefore, by the uniqueness of weak solutions of \eqref{SE1}, we have  $\vi{\Y}=\vi{\Y}^{\p}$. 

\vspace{2mm}
\noindent
\textbf{Step 2:} \textsf{To prove $	\vi{\Y}^{\p_n}\to \vi{\Y}^{\p}$ in $\C([0,T];\H)$ as $n\to\infty$.} 
Our aim is to establish the following convergence:
\begin{align}\label{VCond05}
	\sup_{t\in[0,T]}\|\vi{\Y}^{\p_n}(t)-\vi{\Y}^{\p}(t)\|_\H^2
	\to 0\  \text{ as } \ n\to\infty.
\end{align}Recall that for the system \eqref{SE1}, the energy equality \eqref{CPDE39} holds true. Define $ \vi{\bfX}^{\p_n}_{\p}(\cdot):=  \vi{\Y}^{\p_n} (\cdot)- \vi{\Y}^{\p} (\cdot)$, so that $ \vi{\bfX}^{\p_n}_{\p}(\cdot)$ satisfies the following SPDEs for a.e. $t\in[0,T]$: 
\begin{equation}\label{VCond06}
	\left\{
	\begin{aligned}
		\d \vi{\bfX}^{\p_n}_{\p}(t)&= \bigg[\A(t,\vi{\Y}^{\p_n}(t))-\A(t,\vi{\Y}^{\p}(t))+\B(t,\vi{\Y}^{\p_n}(t))f_n(t)-\B(t,\vi{\Y}^{\p}(t))f(t)\\&\quad+\int_\Z\gamma(t,\vi{\Y}^{\p_n}(t),z)\big(g_n(t,z)-1\big)\nu(\d z)-\int_\Z\gamma(t,\vi{\Y}^{\p}(t),z)\big(g(t,z)-1\big)\nu(\d z)\bigg]\d t,\\
		\vi{\bfX}^{\p_n}_{\p}(0)&=\bf{0}.
	\end{aligned}
	\right.
\end{equation}
It should be noted that both $ \vi{\Y}^{\p_n} (\cdot)$ and $ \vi{\Y}^{\p} (\cdot)$ satisfy the estimate \eqref{CPDE01}. Taking the inner product with $\vi{\bfX}^{\p_n}_{\p}(\cdot)$ to the first equation of the system \eqref{VCond06} and then using Hypothesis \ref{hyp1} (H.2), we get 
\begin{align}\label{VCond07}\nonumber
&	\|\vi{\bfX}^{\p_n}_{\p}(t)\|_\H^2 \\&\nonumber=2\int_0^t \langle \A(s,\vi{\Y}^{\p_n}(s))-\A(s,\vi{\Y}^{\p}(s)),\vi{\bfX}^{\p_n}_{\p}(s)\rangle \d s\\&\nonumber\quad +2\int_0^t \big(\B(s,\vi{\Y}^{\p_n}(s))f_n(s)-\B(s,\vi{\Y}^{\p}(s))f(s),\vi{\bfX}^{\p_n}_{\p}(s)\big)\d s\\&\nonumber\quad + 2\int_0^t\bigg(\int_\Z \big\{\gamma(s,\vi{\Y}^{\p_n}(s),z)\big(g_n(s,z)-1\big)-\gamma(s,\vi{\Y}^{\p}(s),z)\big(g(s,z)-1\big)\big\}\nu(\d z),\vi{\bfX}^{\p_n}_{\p}(s)\bigg) \\&\nonumber \leq 
\int_0^t \big\{\fk{a}(s)+\rho(\vi{\Y}^{\p_n}(s))+\eta(\vi{\Y}^{\p}(s))\big\}\|\vi{\bfX}^{\p_n}_{\p}(s)\|_\H^2\d s\\&\nonumber\quad+ 2\int_0^t\big((\B(s,\vi{\Y}^{\p_n}(s))-\B(s,\vi{\Y}^{\p}(s)))f_n(s), \vi{\bfX}^{\p_n}_{\p}(s)\big)\d s \\&\nonumber\quad+2\int_0^t\big(\B(s,\vi{\Y}^{\p}(s))\big(f_n(s)-f(s)\big), \vi{\bfX}^{\p_n}_{\p}(s)\big)\d s\\&\nonumber\quad
+ 2\int_0^t\int_\Z \left(\big(\gamma(s,\vi{\Y}^{\p_n}(s),z)-\gamma(s,\vi{\Y}^{\p}(s),z)\big)\big(g_n(s,z)-1\big),\vi{\bfX}^{\p_n}_{\p}(s) \right)\nu(\d z)\d s\\&\nonumber\quad +
2\int_0^t\int_\Z \left(\gamma(s,\vi{\Y}^{\p_n}(s),z)\big((g_n(s,z)-1)-(g(s,z)-1)\big),\vi{\bfX}^{\p_n}_{\p}(s) \right)\nu(\d z)\d s\\&\leq \int_0^t \big\{\fk{a}(s)+\rho(\vi{\Y}^{\p_n}(s))+\eta(\vi{\Y}^{\p}(s))\big\}\|\vi{\bfX}^{\p_n}_{\p}(s)\|_\H^2\d s+\sum_{i=1}^{4}|J_i^n(t)|,
\end{align}
for all $t\in[0,T]$, where $J_i^n$, for $i=1,\ldots,4$ represent the final four terms. Using the Cauchy-Schwarz inequality, Hypothesis \ref{hyp2} (H.7) and Young's inequality, we estimate  $|J_1^n(t)|$ as
\begin{align}\label{VCond08}\nonumber
	|J_1^n(t)|&\leq 2\int_0^t\|\B(s,\vi{\Y}^{\p_n}(s))-\B(s,\vi{\Y}^{\p}(s))\|_{\L_2}\|f_n(s)\|_\H\|\vi{\bfX}^{\p_n}_{\p}(s)\|_\H\d s	\\&\nonumber\leq 2\int_0^t\sqrt{L_\B(s)}\|f_n(s)\|_\H\|\vi{\bfX}^{\p_n}_{\p}(s)\|_\H^2\d s \\&\leq \int_0^t\big\{L_\B(s)+\|f_n(s)\|_\H^2\big\}\|\vi{\bfX}^{\p_n}_{\p}(s)\|_\H^2\d s.
\end{align}Now, we consider the term $|J_3^n(t)|$, and estimate it using the Cauchy-Schwarz inequality, Hypothesis \ref{hyp2} (H.9) and Lemma \ref{lemUF3} as 
\begin{align}\label{VCond09}\nonumber
	|J_3^n(t)| &\leq  2\int_0^t\int_\Z \|\gamma(s,\vi{\Y}^{\p_n}(s),z)-\gamma(s,\vi{\Y}^{\p}(s),z)\|_\H\big|g_n(s,z)-1\big|\|\vi{\bfX}^{\p_n}_{\p}(s)\|_\H\nu(\d z)\d s \\& \leq 
	 2\int_0^t\int_\Z R_\gamma(s,z)\big|g_n(s,z)-1\big|\|\vi{\bfX}^{\p_n}_{\p}(s)\|_\H^2\nu(\d z)\d s.
\end{align}
 Substituting \eqref{VCond08} and \eqref{VCond09} in \eqref{VCond07}, we obtain
\begin{align}\label{VCond10}\nonumber
	\|\vi{\bfX}^{\p_n}_{\p}(t)\|_\H^2 & \nonumber\leq \int_0^t \bigg\{\fk{a}(s)+\rho(\vi{\Y}^{\p_n}(s))+\eta(\vi{\Y}^{\p}(s))+L_\B(s)+\|f_n(s)\|_\H^2\\&\qquad+2\int_\Z R_\gamma(s,z)\big|g_n(s,z)-1\big|\nu(\d z)\bigg\}\|\vi{\bfX}^{\p_n}_{\p}(s)\|_\H^2\d s+|J_2^n(t)|+|J_4^n(t)|,
\end{align}
for all $t\in[0,T]$. An application of Gronwall's inequality in \eqref{VCond10} yields
\begin{align}\label{VCond11}\nonumber
	\sup_{t\in[0,T]}	\|\vi{\bfX}^{\p_n}_{\p}(t)\|_\H^2& \leq \bigg\{\sup_{t\in[0,T]}\big[|J_2^n(t)|+|J_4^n(t)|\big]\bigg\}\\&\nonumber\quad\times\exp\bigg\{\int_0^T \bigg[\fk{a}(t)+\rho(\vi{\Y}^{\p_n}(t))+\eta(\vi{\Y}^{\p}(t))+L_\B(t)+\|f_n(t)\|_\H^2\\&\qquad+2\int_\Z R_\gamma(t,z)\big|g_n(t,z)-1\big|\nu(\d z)\bigg]\d t\bigg\}. 
	\end{align}
Note that the integral appearing in the exponential is independent of $n,$  since $\p_n,\p\in \bar{S}^\Upsilon$ $\rho(\cdot)$ satisfies Hypothesis \ref{hyp1} (H.2) and  $\vi{\Y}^{\p_n}(\cdot)$ has the bound given in \eqref{CPDE01}. 

Now, for any $\delta>0$, let us define  $A_{\delta,n}=\big\{t\in[0,T]: \|\vi{\bfX}^{\p_n}_{\p}(t)\|_\H>\delta\}$. Then, by using the similar calculations as in  \eqref{B2}, we have $\lim\limits_{n\to\infty}\lambda_T(A_{\delta,n})=0$.

Splitting the interval $[0,T]$ into two parts, we estimate $\sup\limits_{t\in[0,T]}|J_2^n(t)|$ as
\begin{align}\label{VCond12}\nonumber
	&\sup_{t\in[0,T]}|J_2^n(t)| \\&\nonumber\leq 2 \int_0^T \sqrt{L_\B(t)}\big(1+\|\vi{\Y}^{\p}(t)\|_\H\big)\|f_n(t)-f(t)\|_\H\|\vi{\bfX}^{\p_n}_{\p}(t)\|_\H\d t\\&\nonumber \leq 
	2\bigg\{\int_0^T\sqrt{L_\B(t)}\big(1+\|\vi{\Y}^{\p}(t)\|_\H\big)\|f_n(t)\|_\H\|\vi{\bfX}^{\p_n}_{\p}(t)\|_\H\d t\\&\nonumber\qquad +\int_0^T\sqrt{L_\B(t)}\big(1+\|\vi{\Y}^{\p}(t)\|_\H\big)\|f(t)\|_\H\|\vi{\bfX}^{\p_n}_{\p}(t)\|_\H\d t\bigg\}\\&\nonumber \leq 2\bigg\{C\int_{A_{\delta,n}}\sqrt{L_\B(t)}\|f_n(t)\|_\H \d t+\delta C\int_{A_{\delta,n}^c}\sqrt{L_\B(t)}\|f_n(t)\|_\H \d t\bigg\}\\&\nonumber\quad +
	2\bigg\{C\int_{A_{\delta,n}}\sqrt{L_\B(t)}\|f(t)\|_\H \d t+\delta C\int_{A_{\delta,n}^c}\sqrt{L_\B(t)}\|f(t)\|_\H \d t\bigg\} \\&\nonumber\leq
	C\bigg\{\bigg(\int_{A_{\delta,n}}L_\B(t)\d t\bigg)^{\frac{1}{2}}\bigg(\int_0^T\|f_n(t)\|_\H^2\d t\bigg)^{\frac{1}{2}}+\delta \bigg(\int_0^TL_\B(t)\d t+\int_0^T\|f_n(t)\|_\H^2\d t\bigg)\\&\quad+ C\bigg\{\bigg(\int_{A_{\delta,n}}L_\B(t)\d t\bigg)^{\frac{1}{2}}\bigg(\int_0^T\|f(t)\|_\H^2\d t\bigg)^{\frac{1}{2}}+\delta \bigg(\int_0^TL_\B(t)\d t+\int_0^T\|f(t)\|_\H^2\d t\bigg),
\end{align}where we have used the Cauchy-Schwarz inequality, Hypothesis \ref{hyp2} (H.7) and Young's inequality. Using the absolute continuity of the Lebesgue integral in \eqref{VCond12}, we find
\begin{align}\label{VCond13}
	\lim_{n\to\infty}\sup_{t\in[0,T]}|J_2^n(t)|=0.
\end{align}Again, we split the interval $[0,T]$ into two parts, and using the Cauchy-Schwarz inequality, Hypothesis \ref{hyp2} (H.8) and Young's inequality. we estimate $\sup\limits_{t\in[0,T]}|J_4^n(t)|$ as
 \begin{align}\label{VCond14}\nonumber
 &	\sup_{t\in[0,T]}|J_4^n(t)| \\&\leq 2\int_{\Z_T} \|\gamma(t,\vi{\Y}^{\p_n}(t),z)\|_\H\big|g_n(t,z)-1\big|\|\vi{\bfX}^{\p_n}_{\p}(t)\|_\H\nu(\d z)\d t\nonumber\\&\nonumber\quad +2\int_{\Z_T} \|\gamma(t,\vi{\Y}^{\p}(t),z)\|_\H\big|g(t,z)-1\big|\|\vi{\bfX}^{\p_n}_{\p}(t)\|_\H\nu(\d z)\d t \\&\nonumber\leq 
 	2\int_{\Z_T} L_\gamma(t,z)\big(1+\|\vi{\Y}^{\p_n}(t)\|_\H\big)\big|g_n(t,z)-1\big|\|\vi{\bfX}^{\p_n}_{\p}(t)\|_\H\nu(\d z)\d t\\&\nonumber\quad+	2\int_{\Z_T} L_\gamma(t,z)\big(1+\|\vi{\Y}^{\p}(t)\|_\H\big)\big|g(t,z)-1\big|\|\vi{\bfX}^{\p_n}_{\p}(t)\|_\H\nu(\d z)\d t \\&\nonumber\leq 
 	C\int_{A_{\delta,n}}\int_\Z L_\gamma(t,z)\big|g_n(t,z)-1\big|\nu(\d z)\d t +\delta C \int_{A_{\delta,n}^c}\int_\Z L_\gamma(t,z)\big|g_n(t,z)-1\big|\nu(\d z)\d t
 	\\&\quad +
 	C\int_{A_{\delta,n}}\int_\Z L_\gamma(t,z)\big|g(t,z)-1\big|\nu(\d z)\d t +\delta C \int_{A_{\delta,n}^c}\int_\Z L_\gamma(t,z)\big|g(t,z)-1\big|\nu(\d z)\d t. 
 \end{align}Using Lemma \ref{lemUF3} in \eqref{VCond14}, we deduce 
\begin{align}\label{VCond15}
		\lim_{n\to\infty}\sup_{t\in[0,T]}|J_4^n(t)|=0.
\end{align}Substituting \eqref{VCond13} and \eqref{VCond15} in \eqref{VCond11}, we obtain the required result \eqref{VCond05}.
\end{proof}

\section{Verification of Condition \ref{Cond1} (2)}\label{Sec5}\setcounter{equation}{0} Let us now verify Condition (2) in \ref{Cond1}. A similar verification  result for a class of SPDEs with locally monotone coefficients has been obtained in  \cite{JXJZ}. 
Fix $\Upsilon\in\N$, for any $\boldsymbol{q}_\e=(\psi_\e,\varphi_\e)\in\vi{\mathcal{U}}^\Upsilon$, we consider the following controlled SPDEs for a.e. $t\in[0,T]$ in $\V^*$:
\begin{equation}\label{CSPDE1}
	\left\{
	\begin{aligned}
		\d \vi{\Y}^\e(t)&=  \A(t, \vi{\Y}^\e(t))\d t+\B(t, \vi{\Y}^\e(t))\psi_\e(t)\d t+\sqrt{\e}\B(t, \vi{\Y}^\e(t))\d\W(t)\\&\quad+\int_\Z\gamma(t, \vi{\Y}^\e(t),z)\big(\varphi_\e(t,z)-1\big)\nu(\d z)\d t+\e \int_\Z\gamma(t, \vi{\Y}^\e(t-),z)\vi{N}^{\e^{-1}\varphi_\e}(\d t,\d z),\\
		\vi{\Y}^\e(0)&=\x\in\H.
	\end{aligned}
	\right.
\end{equation}
Remember that $\vi{N}^{\e^{-1}\varphi_\e}=N^{\e^{-1}\varphi_\e}-\e^{-1}\varphi_\e\lambda_T\otimes\nu$, where $\lambda_T$ is the Lebesgue measure and $\nu$ is the intensity measure.

\begin{definition}\label{SSCPDE}
	Let $(\bar{\mathbf{V}},\mathcal{B}(\bar{\mathbf{V}}),\{\bar{\mathscr{F}}_t^{\bar{\mathbf{V}}}\}_{t\geq 0},\bar{\P}^{\bar{\mathbf{V}}})$  be the filtered probability space (see subsection \ref{subsec2.4}). Assume the initial data $\x\in\H$. Then \eqref{CSPDE1} has a \textsf{probabilistically strong solution} if and only if there exists a progressively measurable  process $\Y^\e:[0,T]\times\bar{\mathbf{V}}\to\H$, $\bar{\P}^{\bar{\mathbf{V}}}$-a.s., paths 
	\begin{align*}
		\Y^\e(t,\omega)\in\D([0,T];\H)\cap \L^\beta(0,T;\V),\; \bar{\P}^{\bar{\mathbf{V}}}\text{-a.s., for } \beta\in(1,\infty), 
	\end{align*}and  the following equality holds $\bar{\P}^{\bar{\mathbf{V}}}$-a.s.,  in $\V^*$:
	\begin{align*}
		\Y^\e(t)&=\x+\int_0^t\A(s,\Y^\e(s))\d s+\int_0^t\B(s, \vi{\Y}^\e(s))\psi_\e(s)\d s+\sqrt{\e}\int_0^t \B(s,\Y^\e(s))\d\W(s)\\&\quad+\int_0^t\int_\Z\gamma(s, \vi{\Y}^\e(s),z)\big(\varphi_\e(s,z)-1\big)\nu(\d z)\d s+\e\int_0^t\int_\Z\gamma(s,\Y^\e(s-),z)\vi{N}^{\e^{-1}\varphi_\e}(\d s,\d z), 
	\end{align*}for all $t\in[0,T]$.
	
\end{definition}
Let us take $\boldsymbol{q}_\e=(\psi_\e,\varphi_\e)\in\vi{\mathcal{U}}^\Upsilon$ and define $\nu_\e=\frac{1}{\varphi_\e}$. Then, we have the following result (for proof see Theorem III.3.24, \cite{JJANS}, Lemma 2.3, \cite{ABPFD}):
\begin{lemma}
	The processes 
	\begin{align*}
		\mathcal{E}_t^\e(\nu_\e)&:=\exp\bigg\{\int_{(0,t]\times\Z\times[0,\e^{-1}]}\ln\big(\nu_\e(s,z)\big)\bar{N}(\d s,\d z,\d r)\\&\qquad+\int_{(0,t]\times\Z\times[0,\e^{-1}]}\big(-\nu_\e(s,z)+1\big)\bar{\nu}_T(\d s,\d z,\d r)\bigg\},
	\end{align*}
	and 
	\begin{align*}
		\wi{\mathcal{E}}_t^\e(\psi_\e)&:=\exp\bigg\{-\frac{1}{\sqrt{\e}}\int_0^t\psi_\e(s)\d\W(s)-\frac{1}{2\e}\int_0^t\|\psi_\e(s)\|^2\d s\bigg\}, 
	\end{align*}are $\{\bar{\mathscr{F}}_t^{\bar{\mbf{V}}}\}$-martingales. Set 
	\begin{align*}
		\vi{\mathcal{E}}_t^\e(\psi_\e,\nu_\e):=\mathcal{E}_t^\e(\nu_\e)\wi{\mathcal{E}}_t^\e(\psi_\e).
	\end{align*}Then 
	\begin{align*}
		\mb{Q}_t^\e(\mathrm{G})= \int_{\mathrm{G}}\vi{\mathcal{E}}_t^\e(\psi_\e,\nu_\e)\d \bar{\P}^{\bar{\mbf{V}}}, \ \text{ for } \ \mathrm{G}\in\mathcal{B}(\bar{\mbf{V}}), 
	\end{align*}defines a probability measure in $\bar{\mbf{V}}$.
\end{lemma}
Define the process $\W_{\mathbb{Q}^{\e}_t}(\cdot)$ and the random measure $\vi{N}_{\mathbb{Q}^{\e}_t}^{\e^{-1}\varphi_\e}(\cdot,\cdot)$ by   $\d\W_{\mathbb{Q}^{\e}_t}(t):=\frac{1}{\sqrt{\e}}\psi_{\e}(t)\d t+\d\W(t)$ and  $\vi{N}_{\mathbb{Q}^{\e}_t}^{\e^{-1}\varphi_\e}(\d t,\d z):=\e^{-1}(\phi_{\e}(t,z)-1)\nu(\d z)\d t+\vi{N}^{\e^{-1}\varphi_\e}(\d s,\d z)$, respectively. Then an  application of Girsanov's theorem (see Theorem 1.29, \cite{BOAS}) yields   $\W_{\mathbb{Q}^{\e}_t}(\cdot)$  is a Brownian motion with respect to $\{\bar{\mathcal{F}}_t^{\bar{\mathbf{V}}}\}_{t\geq 0}$ and $\mathbb{Q}^{\e}_t$  and $\vi{N}_{\mathbb{Q}^{\e}_t}^{\e^{-1}\varphi_\e}(\cdot,\cdot)$  is the $(\{\bar{\mathcal{F}}_t^{\bar{\mathbf{V}}}\}_{t\geq 0},\mathbb{Q}^{\e}_t)$-compensated Poisson random measure of ${N}_{\mathbb{Q}^{\e}_t}^{\e^{-1}\varphi_\e}(\cdot,\cdot)$. Under the above change in measures, one can transform the system \eqref{CSPDE1} into the following form: 
\begin{align}\label{5.2}
	\d \vi{\Y}^\e(t&)=  \A(t, \vi{\Y}^\e(t))\d t+\sqrt{\e}\B(t, \vi{\Y}^\e(t))\d\W_{\mathbb{Q}^{\e}_t}(t)+\e \int_\Z\gamma(t, \vi{\Y}^\e(t-),z)\vi{N}_{\mathbb{Q}^{\e}_t}^{\e^{-1}\varphi_\e}(\d t,\d z),
\end{align}
for a.e. $t\in[0,T]$ in $\V^*$. The system \eqref{5.2} is similar to the one given in \eqref{1.1} whose existence of a unique pathwise strong solution is known from Theorem \ref{ER1}. With the above settings, we have the following result on the solution representation:  
\begin{lemma}[Lemma 7.1, \cite{ZBXPJZ}]\label{lem5.3}
	Assume  $\e>0$, for every process $\boldsymbol{q}_\e=(\psi_\e,\varphi_\e)\in \vi{\mathcal{U}}^\Upsilon$ defined on the the probability space $(\bar{\mathbf{V}},\mathcal{B}(\bar{\mathbf{V}}),\{\bar{\mathcal{F}}_t^{\bar{\mathbf{V}}}\},\mb{Q}_T^\e)$, the process $\vi{\Y}^\e$ defined by 
	\begin{align}\label{Vf}
		\vi{\Y}^\e=\mathscr{G}^\e\left(\sqrt{\e}\W(\cdot)+\int_0^{\cdot}\psi_\e(s)\d s,\e N^{\e^{-1}\varphi_\e}\right),
	\end{align}
	is the unique strong solution of the controlled SPDEs \eqref{CSPDE1}. 
\end{lemma}

\subsection{Tightness of the family $\{\mathscr{L}(\vi{\Y}^\e)\}$}
Let us now establish the tightness of the family $\{\mathscr{L}(\vi{\Y}^\e)\}_{\e>0}$.  The following estimates (Lemmas \ref{lemUF2} and \ref{lemUF02}) are helpful in the sequel. 
\begin{lemma}\label{lemUF2}
	For $p=2, 2+\alpha$ or $M$ in the Hypothesis \ref{hyp2} (H.8) and $\boldsymbol{q}_\e=(\psi_\e,\varphi_\e)\in {\mathcal{U}}^{\Upsilon}$  there exists $\e_p,C_p>0$, such that 
	\begin{align}\label{UE1}
		\sup_{\e\in(0,\e_p]}\bar{\E}\left(\sup_{t\in[0,T]}\|\vi{\Y}^\e(t)\|_\H^p\right)+\bar{\E}\left( \int_0^T \|\vi{\Y}^\e(t)\|_\H^{p-2}\|\vi{\Y}^\e(t)\|_\V^\beta\d t\right)\leq C_p(1+\|\x\|_{\H}^p).
	\end{align}
\end{lemma}
\begin{proof}
	Applying infinite dimensional It\^o's formula (see Theorem 1.2, \cite{IGDS}) to the process $\|\vi{\Y}^\e(\cdot)\|_\H^p$, we find for all $t\in[0,T]$ 
	\begin{align}\label{UE2}
		\|\vi{\Y}^\e(t)\|_\H^p=\|\x\|_\H^p+\sum_{j=1}^{7}J_j(t),
	\end{align}where 
	\begin{align*}
		J_1(t)&=\frac{p}{2}\int_0^t\|\vi{\Y}^\e(s)\|_\H^{p-2}\big[2\langle \A(s,\vi{\Y}^\e(s)),\vi{\Y}^\e(s)\rangle+\|\B(s,\vi{\Y}^\e(s))\|_{\L_2}^2\big]\d s,\\
		J_2(t)&= \frac{p(p-2)}{2}\int_0^t\|\vi{\Y}^\e(s)\|_\H^{p-4}\|\vi{\Y}^\e(s)\circ \B(s,\vi{\Y}^\e(s))\|_\H^2\d s,\\
		J_3(t)&=\sqrt{\e}p\int_0^t \|\vi{\Y}^\e(s)\|_\H^{p-2} \big(\B(s,\vi{\Y}^\e(s))\d\W(s),\vi{\Y}^\e(s)\big),\\
		J_4(t)&=p\int_0^t \|\vi{\Y}^\e(s)\|_\H^{p-2} \big(\B(s,\vi{\Y}^\e(s))\psi_\e(s),\vi{\Y}^\e(s)\big),\\
		J_5(t)&= p\int_0^t\int_\Z  \|\vi{\Y}^\e(s-)\|_\H^{p-2}\big(\e \gamma(s,\vi{\Y}^\e(s-),z),\vi{\Y}^\e(s-)\big)\vi{N}^{\e^{-1}\varphi_\e}(\d s,\d z),\\
		J_6(t)&= \int_0^t\int_\Z\big[\|\vi{\Y}^\e(s-)+\e \gamma(s,\vi{\Y}^\e(s-),z)\|_\H^p-\|\vi{\Y}^\e(s-)\|_\H^p\\&\qquad-p\|\vi{\Y}^\e(s-)\|_\H^{p-2}\big(\e\gamma(s,\vi{\Y}^\e(s-),z),\vi{\Y}^\e(s-)\big)\big]N^{\e^{-1}\varphi_\e}(\d s,\d z),
	\end{align*}and
	\begin{align*}
		J_7(t)&= p\int_0^t\int_\Z \|\vi{\Y}^\e(s)\|_\H^{p-2}\big( \gamma(s,\vi{\Y}^\e(s),z)(\varphi_\e(s,z)-1),\vi{\Y}^\e(s)\big)\nu(\d z)\d s.
	\end{align*}We consider the term $J_1(\cdot)$ and estimate it using Hypothesis \ref{hyp1} (H.3), (H.5), H\"older's and Young's inequalities as
	\begin{align}\label{UE3}\nonumber
		|J_1(t) |&\leq \frac{p}{2}\int_0^t \|\vi{\Y}^\e(s)\|_\H^{p-2}\bigg\{2\big[\fk{a}(s)\big(1+\|\vi{\Y}^\e(s)\|_\H^2\big)-L_\A\|\vi{\Y}^\e(s)\|_\V^\beta\big]+\fk{b}(s)\big(1+\|\vi{\Y}^\e(s)\|_\H^2\big)\bigg\}\d s 
		\\&\nonumber \leq 
		-pL_\A \int_0^t\|\vi{\Y}^\e(s)\|_\H^{p-2}\|\vi{\Y}^\e(s)\|_\V^\beta\d s+C\int_0^t\big[\fk{a}(s)+\fk{b}(s)\big]\d s\\&\quad+C\int_0^t \big[\fk{a}(s)+\fk{b}(s)\big]\|\vi{\Y}^\e(s)\|_\H^p\d s.
	\end{align}We estimate  the term $J_2(\cdot)$ using Hypothesis \ref{hyp1} (H.5), H\"older's and Young's inequalities as 
	\begin{align}\label{UE4}\nonumber
		J_2(t)&\leq \frac{p(p-2)}{2}\int_0^t \|\vi{\Y}^\e(s)\|_\H^{p-2}\big\{ \fk{b}(s)\big(1+\|\vi{\Y}^\e(s)\|_\H^2\big)\big\}\d s \\&\leq   C\int_0^t\fk{b}(s)\d s+C\int_0^t\fk{b}(s)\|\vi{\Y}^\e(s)\|_\H^p\d s. 
	\end{align}Now, we consider the term $J_4(\cdot)$ and estimate it using Hypothesis \ref{hyp1} (H.5) and Young's inequality as
	\begin{align}\label{UE5}\nonumber
		|J_4(t)|&\leq p\int_0^t \|\vi{\Y}^\e(s)\|_\H^{p-1}\|\B(s,\vi{\Y}^\e(s)\|_{\L_2}\|\psi_\e(s)\|_\H\d s \\&\nonumber \leq p\int_0^t \|\vi{\Y}^\e(s)\|_\H^{p-1}\|\psi_\e(s)\|_\H \big[\fk{b}(s)\big(1+\|\vi{\Y}^\e(s)\|_\H^2\big)\big]^\frac{1}{2}\d s \\&\nonumber \leq 
		p \int_0^t \sqrt{\fk{b}(s)}\big(1+2\|\vi{\Y}^\e(s)\|_\H^p\big)\|\psi_\e(s)\|_\H\d s \\& \leq C\bigg(\int_0^t \fk{b}(s)\d s+\int_0^t\|\psi_\e(s)\|_\H^2\d s+ \int_0^t\big[\fk{b}(s)+\|\psi_\e(s)\|_\H^2\big]\|\vi{\Y}^\e(s)\|_\H^p\d s\bigg).
	\end{align}For the term $J_6(\cdot)$, we apply the following inequality 
	\begin{align*}
		\big| \|\x+\h\|_\H^p-\|\x\|_\H^p-p\|\x\|_\H^{p-2}(\x,\h)\big| \leq C_p\big(\|\x\|_\H^{p-2}\|\h\|_\H^2+\|\h\|_\H^2\big),
	\end{align*}to find
	\begin{align}\label{UE6}
		|J_6(t)| \leq C_p\int_0^t\int_\Z \big(\| \vi{\Y}^\e(s-)\|_\H^{p-2}\|\e\gamma(s,\vi{\Y}^\e(s-),z)\|_\H^2+\|\e\gamma(s,\vi{\Y}^\e(s-),z)\|_\H^p\big)N^{\e^{-1}\varphi_\e}(\d s,\d z).
	\end{align}Using Hypothesis \ref{hyp2} (H.8), we find 
	\begin{align}\label{UE7}\nonumber
		|J_7(t)|&\nonumber\leq p\int_0^t\int_\Z \|\vi{\Y}^\e(s)\|_\H^{p-1}\|\gamma(s,\vi{\Y}^\e(s),z)\|_\H|\varphi_\e(s,z)-1|\nu(\d z)\d s \\&\nonumber \leq p\int_0^t\int_\Z \|\vi{\Y}^\e(s)\|_\H^{p-1}\big(1+ \|\vi{\Y}^\e(s)\|_\H\big)L_\gamma(s,z)|\varphi_\e(s,z)-1|\nu(\d z)\d s \\&\leq 
		p\int_0^t\int_\Z L_\gamma (s,z)|\varphi_\e(s,z)-1|\nu(\d z)\d s  \nonumber\\&\quad+2p\int_0^t\int_\Z \|\vi{\Y}^\e(s)\|_\H^pL_\gamma(s,z)|\varphi_\e(s,z)-1|\nu(\d z)\d s.
	\end{align}Substituting \eqref{UE3}-\eqref{UE7} in \eqref{UE2}, we obtain 
	\begin{align}\label{UE8}\nonumber
		&	\|\vi{\Y}^\e(t)\|_\H^p+pL_\A \int_0^t\|\vi{\Y}^\e(s)\|_\H^{p-2}\|\vi{\Y}^\e(s)\|_\V^\beta\d s\\ &\nonumber\leq 
		\|\x\|_\H^p+C\int_0^t\big[\fk{a}(s)+\fk{b}(s)+\|f(s)\|_\H^2\big]\d s+C\int_0^t \big[\fk{a}(s)+\fk{b}(s)+\|f(s)\|_\H^2\big]\|\vi{\Y}^\e(s)\|_\H^p\d s\\&\nonumber\quad+C_p\int_0^t\int_\Z \big(\| \vi{\Y}^\e(s-)\|_\H^{p-2}\|\e\gamma(s,\vi{\Y}^\e(s-),z)\|_\H^2+\|\e\gamma(s,\vi{\Y}^\e(s-),z)\|_\H^p\big)N^{\e^{-1}\varphi_\e}(\d s,\d z)\\&\nonumber\quad +p\int_0^t\int_\Z L_\gamma (s,z)|\varphi_\e(s,z)-1|\nu(\d z)\d s  +2p\int_0^t\int_\Z \|\vi{\Y}^\e(s)\|_\H^pL_\gamma(s,z)|\varphi_\e(s,z)-1|\nu(\d z)\d s\\&\quad +|J_3(t)|+|J_4(t)|,
	\end{align}
	for all $t\in[0,T]$. An application of Gronwall's inequality in \eqref{UE8} yields
	\begin{align}\label{UE9}\nonumber
		&	\|\vi{\Y}^\e(t)\|_\H^p+pL_\A \int_0^t\|\vi{\Y}^\e(s)\|_\H^{p-2}\|\vi{\Y}^\e(s)\|_\V^\beta\d s\\ &\nonumber\leq \bigg(\|\x\|_\H^p+C\int_0^T \big[\fk{a}(s)+\fk{b}(s)\big]\d s+C\Upsilon+\sup_{t\in[0,T]}|J_3(t)|+\sup_{t\in[0,T]}|J_4(t)|+p C_{L_\gamma,\Upsilon}\\&\nonumber\quad+C_p\int_{\Z_T} \big(\| \vi{\Y}^\e(s-)\|_\H^{p-2}\|\e\gamma(s,\vi{\Y}^\e(s-),z)\|_\H^2+\|\e\gamma(s,\vi{\Y}^\e(s-),z)\|_\H^p\big)N^{\e^{-1}\varphi_\e}(\d s,\d z) \bigg)\\&\qquad\times \exp\bigg(C\int_0^T\big[\fk{a}(s)+\fk{b}(s)\big]\d s+C\Upsilon+2p C_{L_\gamma,\Upsilon}\bigg),
	\end{align}where we have used Lemma \ref{lemUF} and the fact that $\boldsymbol{q}_\e=(\psi_\e,\varphi_\e)\in {\mathcal{U}}^{\Upsilon}$. Now, we consider the term $\sup\limits_{t\in[0,T]}|J_3(t)|$ and estimate it using Hypothesis \ref{hyp1} (H.5), the Burkholder-Davis-Gundy, H\"older's and Young's inequalities as 
	\begin{align}\label{UE10}\nonumber
		&	\bar{\E}\bigg[\sup_{s\in[0,T]}|J_3(s)|\bigg] \\&\nonumber \leq \sqrt{\e}C_p\bar{\E}\bigg[\int_0^T\|\vi{\Y}^\e(s)\|^{2p-2}\|\B(s,\vi{\Y}^\e(s))\|_{\L_2}^2\d s\bigg]^\frac{1}{2}\\&\nonumber \leq \sqrt{\e}C_p\bar{\E}\bigg[\sup_{s\in[0,T]}\|\vi{\Y}^\e(s)\|_\H^p \int_0^T\|\vi{\Y}^\e(s)\|_\H^{p-2}\|\B(s,\vi{\Y}^\e(s))\|_{\L_2}^2\d s\bigg]^\frac{1}{2}\\&\nonumber\leq \frac{1}{4}\bar{\E}\bigg[\sup_{s\in[0,T]}\|\vi{\Y}^\e(s)\|_\H^p\bigg] 
		+C_p\e\bar{\E}\bigg[\int_0^T\|\vi{\Y}^\e(s)\|_\H^{p-2}\fk{b}(s)\big(1+\|\vi{\Y}^\e(s)\|_\H^2\big)\d s\bigg]
		\\&\nonumber \leq \frac{1}{4}\bar{\E}\bigg[\sup_{s\in[0,T]}\|\vi{\Y}^\e(s)\|_\H^p\bigg] + \e\bar{\E}\bigg[\int_0^T\fk{b}(s)\|\vi{\Y}^\e(s)\|_\H^p\d s\bigg]+C_{{\e},p}\int_0^T\fk{b}(s)\d s \\&\leq \bigg(\frac{1}{4}+\e\int_0^T\fk{b}(s)\d s\bigg)\bar{\E}\bigg[\sup_{s\in[0,T]}\|\vi{\Y}^\e(s)\|_\H^p\bigg] + C_{{\e},p}\int_0^T\fk{b}(s)\d s.
	\end{align}
	Once again using the Hypothesis \ref{hyp2} (H.8), Lemma \ref{lemUF},  the Burkholder-Davis-Gundy, H\"older's and Young's inequalities, we find 
	\begin{align}\label{UE11}\nonumber
		&\bar{\E}\bigg[\sup_{s\in[0,T]}|J_5(t)|\bigg] \\&\nonumber\leq \e C_p\bar{\E}\bigg[\int_{\Z_T} \|\vi{\Y}^\e(s-)\|_\H^{2p-4}|\big(\gamma(s,\vi{\Y}^\e(s-),z),\vi{\Y}^\e(s-)\big)|^2N^{\e^{-1}\varphi_\e}(\d s,\d z)\bigg]^{\frac{1}{2}} \\&\nonumber\leq \e C_p
		\bar{\E}\bigg[\int_{\Z_T} \|\vi{\Y}^\e(s-)\|_\H^{2p-2}L_\gamma^2(s,z)\big(1+\|\vi{\Y}^\e(s-)\|_\H\big)^2 N^{\e^{-1}\varphi_\e}(\d s,\d z)\bigg]^{\frac{1}{2}} \\&\nonumber \leq \e C_p \bar{\E}\bigg[\bigg(\sup_{s\in[0,T]}\|\vi{\Y}^\e(s)\|_\H^p\bigg)\int_{\Z_T}\|\vi{\Y}^\e(s-)\|_\H^{p-2}L_\gamma^2(s,z)\big(1+\|\vi{\Y}^\e(s-)\|_\H\big)^2N^{\e^{-1}\varphi_\e}(\d s,\d z \bigg]^{\frac{1}{2}} \\&\nonumber\leq 
		\frac{1}{4}\bar{\E}\bigg[\sup_{s\in[0,T]}\|\vi{\Y}^\e(s)\|_\H^p\bigg]+ C_p\e \bar{\E}\bigg[\bigg(1+\sup_{s\in[0,T]}\|\vi{\Y}^\e(s)\|_\H^p\bigg)\int_{\Z_T} L_\gamma^2(s,z)\varphi_\e(s,z)\nu(\d z)\d s\bigg] \\& \leq 
		\bigg(\frac{1}{4}+\e  C_p C_{L_\gamma,2,2,\Upsilon}\bigg)\bar{\E}\bigg[\sup_{s\in[0,T]}\|\vi{\Y}^\e(s)\|_\H^p\bigg]+\e C_p C_{L_\gamma,2,2,\Upsilon}.
	\end{align}Again, using Hypothesis \ref{hyp2} (H.8), H\"older's inequality and Lemma \ref{lemUF}, we find 
	\begin{align}\label{UE12}\nonumber
		&C_p\bar{\E}\bigg[\int_{\Z_T} \| \vi{\Y}^\e(s-)\|_\H^{p-2}\|\e\gamma(s,\vi{\Y}^\e(s-),z)\|_\H^2N^{\e^{-1}\varphi_\e}(\d s,\d z)\bigg] \\&\nonumber \leq 
		\e C_p \bar{\E}\bigg[\int_{\Z_T}\| \vi{\Y}^\e(s)\|_\H^{p-2} L_\gamma^2(s,z)\big(1+\|\vi{\Y}^\e(s)\|_\H\big)^2\varphi_\e(s,z)\nu(\d z)\d s\bigg]\\&\nonumber \leq \e C_p \bar{\E}\bigg[\bigg(1+\sup_{s\in[0,T]}\|\vi{\Y}^\e(s)\|_\H^p\bigg)\int_{\Z_T}  L_\gamma^2(s,z)\varphi_\e(s,z)\nu(\d z)\d s\bigg] \\& \leq \e C_p C_{L_\gamma,2,2,\Upsilon}\bar{\E}\bigg[\sup_{s\in[0,T]}\|\vi{\Y}^\e(s)\|_\H^p\bigg]+\e C_p C_{L_\gamma,2,2,\Upsilon},
	\end{align}and 
	\begin{align}\label{UE13}\nonumber
		&	C_p\bar{\E}\bigg[\int_{\Z_T}\|\e\gamma(s,\vi{\Y}^\e(s-),z)\|_\H^pN^{\e^{-1}\varphi_\e}(\d s,\d z) \bigg]\\&\nonumber =\e^{p-1}C_p\bar{\E}\bigg[\int_{\Z_T}\|\gamma(s,\vi{\Y}^\e(s),z)\|_\H^p\varphi_\e(s,z)\nu(\d z)\d s\bigg] \\&\nonumber\leq 
		\e^{p-1}C_p\bar{\E}\bigg[\bigg(\sup_{s\in[0,T]}\|\vi{\Y}^\e(s)\|_\H^p+1\bigg)\int_{\Z_T}  L_\gamma^p(s,z)\varphi_\e(s,z)\nu(\d z)\d s\bigg] \\& \leq 
		\e^{p-1}C_pC_{L_\gamma,p,p,\Upsilon}\bar{\E}\bigg[\sup_{s\in[0,T]}\|\vi{\Y}^\e(s)\|_\H^p\bigg]+\e^{p-1}C_p C_{L_\gamma,p,p,\Upsilon}.
	\end{align}Finally, combining \eqref{UE9}-\eqref{UE13}, we obtain that there exists a positive $\e_p$, such that 
	\begin{align}\label{UE14}\nonumber
		&	\sup_{\e\in(0,\e_p]}\left\{\bar{\E}\bigg[\sup_{s\in[0,T]}\|\vi{\Y}^\e(s)\|_\H^p\bigg]+pL_\A\bar{\E}\bigg[ \int_0^T\|\vi{\Y}^\e(s)\|_\H^{p-2}\|\vi{\Y}^\e(s)\|_\V^\beta\d s\bigg]\right\}
		\\&\leq C_{\Upsilon,p,T,\int_0^T\left[\fk{a}(s)+\fk{b}(s)\right]\d s,L_\gamma}(1+\|\x\|_{\H}^2), 
	\end{align} so that the proof can be completed.
\end{proof} 
\begin{lemma}\label{lemUF02}
	For $p=\frac{M}{2}$, there exists a constant $C_p$ such that 
	\begin{align}\label{UE014}
		\sup_{\e\in(0,\e_{2p})} \bar{\E}\bigg[\bigg(\int_0^T\|\vi{\Y}^\e(s)\|_\V^\beta\d s\bigg)^p\bigg]\leq C_p,
	\end{align} where $\e_{2p}$ is the constant coming from Lemma \ref{lemUF2}. 
\end{lemma}
\begin{proof}
	Applying It\^o's formula to the process $\|\vi{\Y}^\e(\cdot)\|_\H^2$, we get
	\begin{align}\label{UE15}\nonumber
		\|\vi{\Y}^\e(t)\|_\H^2 &\nonumber= 
		\|\x\|_\H^2+\int_0^t\big\{2\langle \A(s,\vi{\Y}^\e(s)),\vi{\Y}^\e(s)\rangle +\e\|\B(s,\vi{\Y}^\e(s))\|_{\L_2}^2\big\}\d s\\&\nonumber\quad+2\sqrt{\e}\int_0^t\big(\B(s,\vi{\Y}^\e(s))\d\W(s),\vi{\Y}^\e(s)\big)+\int_0^t\big(\B(s,\vi{\Y}^\e(s))\psi_\e(s),\vi{\Y}^\e(s)\big)\\&\nonumber\quad+2\int_0^t\int_\Z\big(\e\gamma(s,\vi{\Y}^\e(s-),z),\vi{\Y}^\e(s-)\big)\vi{N}^{\e^{-1}\varphi_\e}(\d s,\d z) \\&\nonumber\quad+\int_0^t\int_\Z\|\e\gamma(s,\vi{\Y}^\e(s-),z)\|_\H^2N^{\e^{-1}\varphi_\e}(\d s,\d z)\\&\nonumber\quad+2\int_0^t
		\int_\Z\big(\gamma(s,\vi{\Y}^\e(s),z)(\varphi_\e(s,z)-1),\vi{\Y}^\e(s)\big)\nu(\d z)\d s
		\\&=:\|\x\|_\H^2+\sum_{i=1}^{6} I_i(t).
	\end{align}We consider the term  $I_1(\cdot)$ and estimate it using the Hypothesis \ref{hyp1} (H.3) and (H.5) as 
	\begin{align}\label{UE16}
		I_1(t) \leq \int_0^t \big\{2\fk{a}(s)\big(1+\|\vi{\Y}^\e(s)\|_\H^2\big)-2L_\A\|\vi{\Y}^\e(s)\|_\V^\beta+ \e\fk{b}(s)\big(1+\|\vi{\Y}^\e(s)\|_\H^2\big) \big\}\d s.
	\end{align}For rest of the calculations, we fix $p=M/2$. Now, we consider $\bar{\E}\big[\sup\limits_{t\in[0,T]}|I_2(t)|^p\big]$ and estimate it using  the Burkholder-Davis-Gundy inequality (see Theorem 1.1, \cite{DLB}), Hypothesis \ref{hyp1} (H.5), H\"older's and Young's inequalities  as
	\begin{align}\label{UE17}\nonumber
		&\bar{\E}\bigg[\sup_{t\in[0,T]}|I_2(t)|^p\bigg] \\&\nonumber\leq \e^\frac{p}{2}C_p \bar{\E}\bigg[\int_0^T\|\vi{\Y}^\e(s)\|_\H^2\|\B(s,\vi{\Y}^\e(s))\|_{\L_2}^2\d s\bigg]^\frac{p}{2} \\&\nonumber\leq \e^\frac{p}{2}C_p \bar{\E}\bigg[\sup_{s\in[0,T]}\|\vi{\Y}^\e(s)\|_\H^{2p}\bigg]+C_p\bar{\E}\bigg[\int_0^T\fk{b}(s)\big(1+\|\vi{\Y}^\e(s)\|_\H^2\big)\d s\bigg]^p\\&\leq  \e^\frac{p}{2}C_p \bar{\E}\bigg[\sup_{s\in[0,T]}\|\vi{\Y}^\e(s)\|_\H^{2p}\bigg]+C_p\bigg(\int_0^T\fk{b}(s)\d s\bigg)^p\bigg(1+\bar{\E}\bigg[\sup_{s\in[0,T]}\|\vi{\Y}^\e(s)\|_\H^{2p}\bigg]\bigg)
	\end{align}
	Let us consider the term $\bar{\E}\big[\sup\limits_{t\in[0,T]}|I_3(t)|^p\big]$ and estimate it using Hypothesis \ref{hyp1} (H.5), the Cauchy-Schwarz, H\"older's and Young's inequalities as 
	\begin{align}\label{UE18}\nonumber
		\bar{\E}\bigg[\sup_{t\in[0,T]}|I_3(t)|^p\bigg] &\nonumber\leq C_p\bar{\E}\bigg[\int_0^T\|\B(s,\vi{\Y}^\e(s))\|_{\L_2}\|\psi_\e(s)\|_\H\|\vi{\Y}^\e(s)\|_\H \d s\bigg]^p \\&\nonumber\leq C_p \bar{\E}\bigg[\int_0^T \sqrt{\fk{b}(s)} \big(1+2\|\vi{\Y}^\e(s)\|_\H^2\big)\|\psi_\e(s)\|_\H \d s\bigg]^p \\& \leq 
		C_p\left(1+\bar{\E}\bigg[\sup_{s\in[0,T]}\|\vi{\Y}^\e(s)\|_\H^{2p}\bigg]\right)\bigg[\int_0^T\fk{b}(s)\d s+\Upsilon\bigg]^p. 
	\end{align}With the help of Corollary 2.4, \cite{ZBEH}, Hypothesis (H.8), Lemmas  \ref{lemUF} and \ref{lemUF2},  we estimate the term $\bar{\E}\big[\sup\limits_{t\in[0,T]}|I_5(t)|^p\big]$ as
	\begin{align}\label{UE19}\nonumber
		&\bar{\E}\bigg[\sup\limits_{t\in[0,T]}|I_5(t)|^p\bigg]\nonumber\\&\nonumber \leq 
		\e^{2p}C_p\bar{\E}\bigg[\bigg(\int_{\Z_T}\|\gamma(s,\vi{\Y}^\e(s-),z)\|_\H^2N^{\e^{-1}\varphi_\e}(\d s,\d z)\bigg)^p\bigg]\\&\nonumber \leq 	\e^{2p-1}C_p\bar{\E}\bigg[\int_{\Z_T} \|\gamma(s,\vi{\Y}^\e(s),z)\|_\H^{2p}\varphi_\e(s,z)\nu(\d z)\d s\bigg]\\&\nonumber\quad + \e^{p}C_p\bar{\E}\bigg[\int_{\Z_T} \|\gamma(s,\vi{\Y}^\e(s),z)\|_\H^2\varphi_\e(s,z)\nu(\d z)\d s\bigg]^p \\&\nonumber\leq \e^{2p-1}C_p\bar{\E}\bigg[\int_{\Z_T}L_\gamma^{2p}(s,z)\big(1+\|\vi{\Y}^\e(s)\|_\H^{2p}\big)\varphi_\e(s,z)\nu(\d z)\d s\bigg]\\&\nonumber\quad + \e^{p}C_p\bar{\E}\bigg[\int_{\Z_T} L_\gamma^2(s,z)\big(1+\|\vi{\Y}^\e(s)\|_\H^2\big)\varphi_\e(s,z)\nu(\d z)\d s\bigg]^p\\&\nonumber\leq \e^{2p-1}C_p\bigg(1+\bar{\E}\bigg[\sup_{s\in[0,T]}\|\vi{\Y}^\e(s)\|_\H^{2p}\bigg]\bigg)\bigg(\sup_{\varphi\in  S^\Upsilon}\int_{\Z_T} L_\gamma^{2p}(s,z)\varphi(s,z)\nu(\d z)\d s\bigg)\\&\nonumber\quad+ \e^{p}C_p\bigg(1+\bar{\E}\bigg[\sup_{s\in[0,T]}\|\vi{\Y}^\e(s)\|_\H^{2p}\bigg]\bigg)\bigg(\sup_{\varphi\in  S^\Upsilon}\int_{\Z_T} L_\gamma^2(s,z)\varphi(s,z)\nu(\d z)\d s\bigg)^p\\&\leq C_p\bigg( \e^{2p-1} C_{L_\gamma,2p,2p,\Upsilon}+\e^{p}C_{L_\gamma,2,2,\Upsilon}^p\bigg)\bigg(1+\bar{\E}\bigg[\sup_{s\in[0,T]}\|\vi{\Y}^\e(s)\|_\H^{2p}\bigg]\bigg).
	\end{align}Now, we consider the term $\bar{\E}\big[\sup\limits_{t\in[0,T]}|I_4(t)|^p\big]$ and estimate it using the Burkholder-Davis-Gundy inequality (see Theorem 1.1, \cite{DLB}) and \eqref{UE19}   as
	\begin{align}\label{UE20}\nonumber
		&\bar{\E}\bigg[\sup_{t\in[0,T]}|I_4(t)|^p\bigg]\\&\nonumber \leq\e^{p} C_p \bar{\E}\bigg[\int_{\Z_T} \|\gamma(s,\vi{\Y}^\e(s),z)\|_\H^2\|\vi{\Y}^\e(s)\|_\H^2N^{\e^{-1}\varphi_\e}(\d s,\d z)\bigg]^{\frac{p}{2}}\\&\nonumber\leq \e^{p} C_p\bigg\{ \bar{\E}\bigg[\sup_{s\in[0,T]}\|\vi{\Y}^\e(s)\|_\H^p\bigg]+\bar{\E}\bigg[ \int_{\Z_T} \|\gamma(s,\vi{\Y}^\e(s),z)\|_\H^2N^{\e^{-1}\varphi_\e}(\d s,\d z)\bigg]^p\bigg\}
		\\&\leq \e^{p} C_p\bar{\E}\bigg[\sup_{s\in[0,T]}\|\vi{\Y}^\e(s)\|_\H^p\bigg]+C_p\bigg( \e^{p-1} C_{L_\gamma,2p,2p,\Upsilon}+C_{L_\gamma,2,2,\Upsilon}^p\bigg)\bigg(1+\bar{\E}\bigg[\sup_{s\in[0,T]}\|\vi{\Y}^\e(s)\|_\H^{2p}\bigg]\bigg).
	\end{align}Using  Hypothesis \ref{hyp2} (H.8), Lemmas  \ref{lemUF} and \ref{lemUF2}, we estimate the term $\bar{\E}\big[\sup\limits_{t\in[0,T]}|I_6(t)|^p\big]$ as
	\begin{align}\label{UE21}\nonumber
		&	\bar{\E}\bigg[\sup_{t\in[0,T]}|I_6(t)|^p\bigg]\\& \nonumber\leq 	C_p\bar{\E}\bigg[\int_{\Z_T} \|\gamma (s,\vi{\Y}^\e(s),z)\|_\H\|\vi{\Y}^\e(s)\|_\H|\varphi_\e(s,z)-1|\nu(\d z)\d s \bigg]^p \\&\nonumber\leq C_p \bar{\E}\bigg[\int_{\Z_T} L_\gamma (s,z)\|\vi{\Y}^\e(s)\|_\H\big(1+\|\vi{\Y}^\e(s)\|_\H\big)|\varphi_\e(s,z)-1|\nu(\d z)\d s \bigg]^p \\&\nonumber\leq C_p \bigg(1+\bar{\E}\bigg[\sup_{s\in[0,T]}\|\vi{\Y}^\e(s)\|_\H^{2p}\bigg]\bigg)\bigg(\sup_{\varphi\in S^\Upsilon}\int_{\Z_T} L_\gamma (s,z)|\varphi_\e(s,z)-1|\nu(\d z)\d s \bigg)^p \\&\leq C_p C_{L_\gamma,\Upsilon}^p\bigg(1+\bar{\E}\bigg[\sup_{s\in[0,T]}\|\vi{\Y}^\e(s)\|_\H^{2p}\bigg]\bigg) .
	\end{align}
	Substituting \eqref{UE16}-\eqref{UE21} in \eqref{UE15} and using Lemmas \ref{lemUF} and \ref{lemUF2}, we obtain the required estimate \eqref{UE014}.
\end{proof}
Let us move to our main result of this subsection, that is, the tightness of the laws of $\{\vi{\Y}^\e\}_{\e>0}$. A similar result is established in \cite{JXJZ} (see Proposition 4.1), where the authors consider SPDEs  with locally monotone coefficients perturbed by pure  jump noise.
\begin{proposition}\label{propUF}
	For some $\e>0$, the family $\{\mathscr{L}(\Y^\e)\}_{\e\in(0,\e_0]}$ is tight in the space $\D([0,T];\V^*)$ with the Skorokhod topology. Furthermore, set
	\begin{align*}
		X_1^\e(t)&= \int_0^t \A(s,\vi{\Y}^\e(s))\d s,\\
		X_2^\e(t)&= \int_0^t\B(s,	\vi{\Y}^\e(s))\psi_\e(s)\d s,\\
		X_3^\e(t)&= \int_0^t\int_\Z \gamma(s,\vi{\Y}^\e(s),z)\big(\varphi_\e(s,z)-1\big)\nu(\d z)\d s,\\
		M_1^\e(t)&=\sqrt{\e}\int_0^t  \B(s,\vi{\Y}^\e(s))\d\W(s),\\
		M_2^\e(t)&= \e\int_0^t\int_\Z \gamma(s,\vi{\Y}^\e(s-),z)\vi{N}^{\e^{-1}\varphi_\e}(\d s,\d z).
	\end{align*}Then 
	\begin{enumerate}
		\item [(1)] $\lim\limits_{\e\to0}\bar{\E}\bigg[\sup\limits_{t\in[0,T]}\|M_1^\e(t)\|_\H^2\bigg]=\lim\limits_{\e\to0}\bar{\E}\bigg[\sup\limits_{t\in[0,T]}\|M_2^\e(t)\|_\H^2\bigg]=0$,
		\item [(2)] $\{X_1^\e(t)\}_{t\in[0,T]}$ is tight in $\C([0,T];\V^*)$,
		\item [(3)] $\{X_2^\e(t)\}_{t\in[0,T]}$ is tight in $\C([0,T];\V^*)$,
		\item [(4)] $\{X_3^\e(t)\}_{t\in[0,T]}$ is tight in $\C([0,T];\V^*)$.
	\end{enumerate}
\end{proposition}
\begin{proof} 
	The proof is divided into the following steps: 
	
	\vspace{2mm}
	\noindent
	\textbf{Step 1:} \textsf{Proof of (1).} 
	First, we prove the statement (1). Using the Burkholder-Davis-Gundy inequality (see Theorem 1.1, \cite{CMMR}), Hypothesis \ref{hyp1} (H.5), Lemmas \ref{lemUF} and \ref{lemUF2}, we deduce
	\begin{align}\label{UE22}\nonumber
		\bar{\E}\bigg[\sup_{t\in[0,T]}\|M_1^\e(t)\|_\H^2\bigg] &\nonumber\leq {\e}C \bar{\E}\bigg[\int_0^T\|\B(s,\vi{\Y}^\e(s))\|_{\L_2}^2\d s\bigg]\\& \nonumber\leq {\e}C\bar{\E}\bigg[\int_0^T \fk{b}(s)\big(1+\|\vi{\Y}^\e(s)\|_\H^2\big)\d s\bigg]\\&\nonumber \leq 
		{\e}C\bigg\{\bigg(\int_0^T\fk{b}(s)\d s\bigg)\bar{\E}\bigg[\sup_{s\in[0,T]}\big(1+\|\vi{\Y}^\e(s)\|_\H^2\big)\bigg]\bigg\}\\ &
		\to 0\ \text{ as } \ \e\to0.
	\end{align}Now, we consider the term $\bar{\E}\big[\sup\limits_{t\in[0,T]}\|M_2^\e(t)\|_\H^2\big] $ and estimate it using the Burkholder-Davis-Gundy inequality (see Theorem 1.1, \cite{CMMR}), Hypothesis \ref{hyp1} (H.8) and Lemma \ref{lemUF2} as
	\begin{align}\label{UE23} \nonumber
		&	\bar{\E}\bigg[\sup_{t\in[0,T]}\|M_2^\e(t)\|_\H^2\bigg] \\&\nonumber\leq  \e C \bar{\E}\bigg[\int_{\Z_T} \|\gamma(s,\vi{\Y}^\e(s),z)\|_\H^2\varphi_\e(s,z)\nu(\d z)\d s\bigg] \\&\nonumber \leq  \e C \bar{\E}\bigg[\int_{\Z_T} L_\gamma^2(s,z)\big(1+\|\vi{\Y}^\e(s)\|_\H\big)^2\varphi_\e(s,z)\nu(\d z)\d s\bigg] \\& \nonumber\leq \e C \bar{\E}\bigg(1+\bar{\E}\bigg[\sup_{s\in[0,T]}\|\vi{\Y}^\e(s)\|_\H^2\bigg]\bigg)\bigg(\sup_{\varphi\in S^\Upsilon}\int_{\Z_T} L_\gamma^2(s,z)\varphi(s,z)\nu(\d z)\d s\bigg) \\&\nonumber \leq 
		\e C \bar{\E}\bigg(1+\bar{\E}\bigg[\sup_{s\in[0,T]}\|\vi{\Y}^\e(s)\|_\H^2\bigg]\bigg)2 C_{L_\gamma,2,2,\Upsilon} \\&\to0\ \text{ as } \ \e\to0.
	\end{align}In order to prove the tightness property, it is enough to show that for any $\delta>0$, there exists a compact subset $K_\delta\subset \C([0,T];\V^*)$, such that 
	\begin{align*}
		\bar{\P}(X_i^\e\in K_\delta)>1-\delta, \ \text{ for } \ i=1,2,3.
	\end{align*}
	
	\vspace{2mm}
	\noindent
	\textbf{Step 2:} \textsf{Proof of (3) and (4).} 
	Let us define 
	\begin{align*}
		\mathcal{D}_{N,\Upsilon}&:= \bigg\{(\boldsymbol{r}(t),\q(t)): \boldsymbol{r}\in \D([0,T];\H)\cap \L^\beta(0,T;\V), \; \sup_{t\in[0,T]}\|\boldsymbol{r}(t)\|_\H\leq N, \q\in \mathcal{U}^\Upsilon\bigg\},\\
		\mathcal{R}_1(\mathcal{D}_{N,\Upsilon})&:= \bigg\{h_1(\cdot)=\int_0^{\cdot}\B(s,\boldsymbol{r}(s))\psi(s)\d s,\; (\boldsymbol{r},\q)\in \mathcal{D}_{N,\Upsilon}\bigg\}.
	\end{align*} For any $h_1\in	\mathcal{R}_1(\mathcal{D}_{N,\Upsilon})$, by using  Hypothesis \ref{hyp1} (H.5), Cauchy-Schwarz and H\"older's inequalities, we have 
	\begin{align}\label{UE24}\nonumber
		\|h_1(t)-h_1(s)\|_\H& \leq \int_s^t\|\B(\tau,\boldsymbol{r}(\tau))\|_{\L_2}\|\psi(\tau)\|_\H \d \tau\\& \nonumber\leq \sup_{\tau\in[s,t]}\big(1+\|\boldsymbol{r}(\tau)\|_\H\big)\int_s^t\sqrt{\fk{b}(\tau)}\|\psi(\tau)\|_\H\d \tau\\&\nonumber \leq 
		(1+N) \bigg(\int_s^t\fk{b}(\tau)\d \tau\bigg)^{\frac{1}{2}}\bigg(\int_s^t\|\psi(\tau)\|_\H^2\d \tau\bigg)^{\frac{1}{2}}\\&\leq (1+N)\Upsilon^{\frac{1}{2}}\bigg(\int_s^t\fk{b}(\tau)\d \tau\bigg)^{\frac{1}{2}}. 
	\end{align}
	Now, using the fact that $\fk{b}\in\L^1(0,T;\R^+)$, we conclude:
	\begin{itemize}
		\item for any $\xi>0$, there exists a $\varpi>0$ independent of $h$ such that for any $s,t\in[0,T]$ and $|t-s|\leq \varpi$, we have
		\begin{align*}
			\|h_1(t)-h_1(s)\|_\H \leq \xi, \ \text{ for all } \ h\in \mathcal{R}_1(\mathcal{D}_{N,\Upsilon}), 
		\end{align*}
		\item and  \begin{align*}
			\sup_{h_1\in \mathcal{R}(\mathcal{D}_{N,\Upsilon})}\sup_{t\in[0,T]} \|h_1(t)\|_\H = \sup_{h_1\in \mathcal{R}_1(\mathcal{D}_{N,\Upsilon})}\sup_{t\in[0,T]}\|h_1(t)-h_1(0)\|_\H \leq (1+N)\Upsilon^{\frac{1}{2}}\left(\int_0^T\fk{b}(t)\d t\right)^{\frac{1}{2}}. 
	\end{align*}\end{itemize}
	Since, the embedding $\V\hookrightarrow \H$ is compact and hence $\H\hookrightarrow \V^*$  also. By the Arzel\'a-Ascoli theorem (see Appendix A5, \cite{WR}), $\mathcal{R}_1^c(\mathcal{D}_{N,\Upsilon})$ (complement of $\mathcal{R}_1(\mathcal{D}_{N,\Upsilon})$) in $\C([0,T];\V^*)$, is a compact subset of $\C([0,T];\V^*)$.
	
	Therefore, an application of Markov's inequality yields
	\begin{align}\label{529}
		&\bar{\P}\big(X_2^\e\in \mathcal{R}_1^c(\mathcal{D}_{N,\Upsilon})\big) \nonumber\\&\geq\bar{\P}\bigg(\sup_{t\in[0,T]}\|\vi{\Y}^\e(t)\|_\H\leq N\bigg) = 1- \bar{\P}\bigg(\sup_{t\in[0,T]}\|\vi{\Y}^\e(t)\|_\H>N\bigg) \nonumber\\&\geq 1-\frac{1}{N^2}\bar{\E}\bigg[\sup_{t\in[0,T]}\|\vi{\Y}^\e(t)\|_\H^2\bigg]\geq 1-\frac{C}{N^2},
	\end{align}where we have used Lemma \ref{lemUF2}. Hence, we obtain the tightness of $\{X_2^\e\}$ in $\C([0,T];\V^*)$.
	
	Let us now define
	\begin{align*}
		\mathcal{R}_2(\mathcal{D}_{N,\Upsilon})&:= \bigg\{h_2(\cdot)=\int_0^{\cdot}\int_\Z\gamma(s,\boldsymbol{r}(s),z)\big(\varphi(s,z)-1\big)\nu(\d z)\d s,\  (\boldsymbol{r},\q)\in \mathcal{D}_{N,\Upsilon}\bigg\}.
	\end{align*} For any $h_2\in\mathcal{R}_2(\mathcal{D}_{N,\Upsilon})$, we have 
	\begin{align}\label{UE25}\nonumber
		\|h_2(t)-h_2(s)\|_\H &\leq \int_s^t\int_\Z \|\gamma(\tau,\boldsymbol{r}(\tau),z)\|_\H|\varphi(\tau,z)-1|\nu(\d z)\d \tau \\& \nonumber\leq \sup_{\tau\in[s,t]}\big(1+\|\boldsymbol{r}(\tau)\|_\H\big)\int_s^t\int_\Z L_\gamma(\tau,z)|\varphi(\tau,z)-1|\nu(\d z)\d \tau\\&\leq 
		(1+N)\sup_{\varphi\in S^\Upsilon} \int_s^t\int_\Z L_\gamma(\tau,z) |\varphi(\tau,z)-1|\nu(\d z)\d \tau.
	\end{align}Using Lemma \ref{lemUF} in the above inequality, we deduce:
	\begin{itemize}
		\item for any positive $\xi$,  there exists a positive $\varpi$ independent of $h_2$ such that for any $s,t\in[0,T]$ and $|t-s|\leq  \varpi$, we have
		\begin{align*}
			\|h_2(t)-h_2(s)\|_\H \leq \xi, \ \text{ for all } \ h_2\in \mathcal{R}_2(\mathcal{D}_{N,\Upsilon}), 
		\end{align*} 
		\item and
		\begin{align*}
			\sup_{h_2\in \mathcal{R}_2(\mathcal{D}_{N,\Upsilon})}\sup_{t\in[0,T]} \|h_2(t)\|_\H = \sup_{h\in \mathcal{R}(\mathcal{D}_{N,\Upsilon})}\sup_{t\in[0,T]}\|h_2(t)-h_2(0)\|_\H \leq (1+N)C_{L_\gamma,\Upsilon}.
		\end{align*}
	\end{itemize}
	Once again an application of  Arzel\'a-Ascoli theorem yields $\mathcal{R}_2^c(\mathcal{D}_{N,\Upsilon})$ (complement of $\mathcal{R}(\mathcal{D}_{N,\Upsilon})$) in $\C([0,T];\V^*)$, is a compact subset of $\C([0,T];\V^*)$.
	A calculation similar to \eqref{529} provides the tightness of $\{X_3^\e\}$ in $\C([0,T];\V^*)$.

	\vspace{2mm}
	\noindent
	\textbf{Step 3:} \textsf{Proof of (2).} 
	Let us now move to the proof of part (2). Using H\"older's, Young's inequalities, Hypothesis \ref{hyp1} (H.4), Lemmas \ref{lemUF} and \ref{lemUF2}, recall $\eta_0$ in Hypothesis \ref{hyp2} (H.8), for $p=\beta+\eta_0$, we find 
	\begin{align*}
		&	\bar{\E}\big[\|X_1^\e(t)-X_1^\e(s)\|_{\V^*}^p\big] \\&\leq \bar{\E}\bigg[\int_s^t\|\A(\tau,\vi{\Y}^\e(\tau))\|_{\V^*}\d \tau\bigg]^p\\& \leq 
		|t-s|^{\frac{p}{\beta}}\bar{\E}\bigg[\int_s^t \|\A(\tau,\vi{\Y}^\e(\tau))\|_{\V^*}^{\frac{\beta}{\beta-1}}\d \tau\bigg]^{\frac{(\beta-1)p}{\beta}}\\& \leq 
		|t-s|^{\frac{p}{\beta}}\bar{\E}\bigg[\int_s^t   \big(\fk{a}(\tau)+C\|\vi{\Y}^\e(\tau)\|_\V^\beta\big)\big(1+\|\vi{\Y}^\e(\tau)\|_\H^\alpha\big)	\d \tau\bigg]^{\frac{(\beta-1)p}{\beta}} \\& \leq C|t-s|^{\frac{p}{\beta}}\bigg\{\bar{\E}\bigg[\int_s^t\big(1+\|\vi{\Y}^\e(\tau)\|_\H^\alpha\big)	\d \tau\bigg]^{\frac{2p(\beta-1)}{\beta}}+\bar{\E}\bigg[\int_s^t\big(\fk{a}(\tau)+C\|\vi{\Y}^\e(\tau)\|_\V^\beta\big)\d \tau\bigg]^{\frac{2p(\beta-1)}{\beta}}\bigg\}
		\\& \leq C_{p,\beta,\fk{a}}|t-s|^{\frac{p}{\beta}}.
	\end{align*}Using Kolmogorov's criterion (cf. Theorem 3.3, \cite{DaZ}), for every $\varpi\in \big(0,\frac{1}{\beta}-\frac{1}{p}\big)$, there exists a constant $C_{\varpi}$ (independent of $\e$) such that 
	\begin{align}\label{UE26}
		\bar{\E}\bigg[\sup_{t\neq s\in[0,T]}\frac{\|X_1^\e(t)-X_1^\e(s)\|_{\V^*}^p}{|t-s|^{p\varpi}}\bigg]\leq C_{\varpi}.
	\end{align}
	From \eqref{CSPDE1}, we have for a.e. $t\in[0,T]$ in $\V^*$
	\begin{align*}
		\vi{\Y}^\e(t)=\x +X_1^\e(t)+X_2^\e(t)+X_3^\e(t)+M_1^\e(t)+M_2^\e(t).
	\end{align*}Then 
	\begin{align}\label{UE27}\nonumber
		\bar{\E}\bigg[\sup_{t\in[0,T]}\|X_1^\e(t) \|_\H^2 \bigg]&\nonumber\leq C\bigg\{\|\x\|_\H^2+	\bar{\E}\bigg[\sup_{t\in[0,T]}\|		\vi{\Y}^\e(t)\|_\H^2 \bigg]
		+\bar{\E}\bigg[\sup_{t\in[0,T]}\|	X_2^\e(t)\|_\H^2 \bigg]	\\&\quad+\bar{\E}\bigg[\sup_{t\in[0,T]}\|	X_3^\e(t)\|_\H^2 \bigg]+\bar{\E}\bigg[\sup_{t\in[0,T]}\|	M_1^\e(t)\|_\H^2 \bigg]+\bar{\E}\bigg[\sup_{t\in[0,T]}\|	M_2^\e(t)\|_\H^2 \bigg]\bigg\}.
	\end{align}One should note that 
	\begin{align}\label{UE2p8}\nonumber
		\bar{\E}\bigg[\sup_{t\in[0,T]}\|	X_2^\e(t)\|_\H^2\bigg]&\leq  \bar{\E}\bigg[\int_0^T\|\B(s,\vi{\Y}^\e(s))\psi_\e(s)\|_\H\d s\bigg]^2\\&\nonumber\leq \bar{\E}\bigg[ \int_0^T\sqrt{\fk{b}(s)}\big(1+\|\vi{\Y}^\e(s)\|_\H\big)\|\psi_\e(s)\|_\H\d s\bigg]^2\\&\leq 2 \bigg(1+\bar{\E}\bigg[\sup_{s\in[0,T]}\|\vi{\Y}^\e(s)\|_\H^2\bigg]\bigg)\bigg(
		\int_0^T \fk{b}(s)\d s+\Upsilon\bigg)^2,
	\end{align}
	and 
	\begin{align}\label{UE28}\nonumber
		&	\bar{\E}\bigg[\sup_{t\in[0,T]}\|	X_3^\e(t)\|_\H^2\bigg] \\ \nonumber &\leq \bar{\E}\bigg[\int_{\Z_T} \|\gamma(s,\vi{\Y}^\e(s),z)\|_\H\big|\varphi_\e(s,z)-1\big|\nu(\d z)\d s\bigg]^2\\& \nonumber\leq 2\bigg(1+\bar{\E}\bigg[\sup_{s\in[0,T]}\|\vi{\Y}^\e(s)\|_\H^2\bigg]\bigg)\bigg(\sup_{\varphi\in S^\Upsilon}\int_{\Z_T} L_\gamma(s,z)\big|\varphi(s,z)-1\big|\nu(\d z)\d s\bigg)^2\\&\leq  C_{L_\gamma,\Upsilon}\bigg(1+\bar{\E}\bigg[\sup_{s\in[0,T]}\|\vi{\Y}^\e(s)\|_\H^2\bigg]\bigg).
	\end{align}
	Combining \eqref{UE22}, \eqref{UE23}, \eqref{UE2p8}, \eqref{UE28} and Lemma \ref{lemUF2}, we find  
	\begin{align}\label{UE29}
		\bar{\E}\bigg[\sup_{s\in[0,T]}\|X_1^\e(t) \|_\H^2 \bigg]\leq C<\infty,
	\end{align}where a constant $C$ which  is independent of $\e$.
	
	For $\varpi\in(0,1)$ and $R>0$, define 
	\begin{align*}
		K_{\varpi,R} :=\bigg\{\boldsymbol{y}\in\C([0,T];\V^*):\sup_{t\in[0,T]}\|\boldsymbol{y}(t)\|_\H+\sup_{t\neq s\in[0,T]}\frac{\|\boldsymbol{y}(t)-\boldsymbol{y}(s)\|_{\V^*}}{|t-s|^{\varpi}}\leq R\bigg\}.
	\end{align*}Once again using the Arzel\'a-Ascoli theorem, one can see that $	K_{\varpi,R}$ is a compact subset of  $\C([0,T];\V^*)$. Using \eqref{UE26}, \eqref{UE29} and Markov's inequality, for some $\varpi\in(0,1)$ and any $R>0$, we find
	\begin{align*}
		\bar{\P}(X_1^\e(t)\notin	K_{\varpi,R}) \geq \frac{C_{\varpi,T}}{R}.
	\end{align*}Thus, we obtain the tightness of $\{X_1^\e\}$ in the space $\C([0,T];\V^*)$.
	
	Finally, the tightness of laws of $\{\vi{\Y}^\e\}_{\e\in(0,\e_0]}$ in the space $\D([0,T];\V^*)$ follows from the above conclusions.
\end{proof}

\subsection{Convergence of the processes}
The aim of this section is to identify the limit function of the sequence  $\{\vi{\Y}^\e\}_{\e\in(0,\e_0]}$ as a solution of the equation \eqref{Prop.EX1} given below. We will use the tightness property which has been established in the previous section. We have  already proved that the limits of $M_1^\e$ and $M_2^\e$ goes to 0 as $\e\to0 $  in Proposition \ref{propUF}, so we only need to take care of the remaining three terms. Let us denote the limit function of $\{\vi{\Y}^\e\}_{\e\in(0,\e_0]}$  by $\Y$.  In this section, we assume that for a.e. $\omega$, $\q_\e=(\psi_\e,\varphi_\e)(\omega)\in \bar{S}^\Upsilon$ converges to $\q=(\psi,\varphi)(\omega)$ in $\bar{S}^\Upsilon$ weakly and $\vi{\Y}^\e(\omega)$ converges to $\Y({\omega})$ in $\D([0,T];\V^*)$ strongly with respect to the supremum norm. A similar result to the next Lemma  has been obtained in \cite{JXJZ} (Lemma 5.1).
\begin{lemma}\label{Cgslem1}
	There exists a subsequence $\{\e_j\}$, $\bar{\Y}\in\L^\beta(\Omega\times[0,T];\V)\cap \L^{\alpha+2}(\Omega;\L^\infty(0,T;\H))$ and $\wi{\A}(\cdot) \in\L^{\frac{\beta}{\beta-1}}(\Omega\times[0,T];\V^*)$, for $\beta\in(1,\infty)$ such that
	\begin{enumerate}
		\item [(1)] $\vi{\Y}^{\e_j}\xrightharpoonup{}\bar{\Y}$ in $\L^\beta(\Omega\times[0,T];\V)$, for $\beta\in(1,\infty)$; 
		\item   [(2)] $\vi{\Y}^{\e_j}\xrightharpoonup{\ast} \bar{\Y}$  in $\L^p(\Omega;\L^\infty(0,T;\H))$;
		\item [(3)]   $ \A(\cdot,\vi{\Y}^{\e_j})\xrightharpoonup{} \wi{\A}(\cdot)$ in $\L^{\frac{\beta}{\beta-1}}(\Omega\times[0,T];\V^*)$;
		\item [(4)] and  \begin{align*}
			\lim_{\e\to0} \bar{\E}\bigg[\sup_{t\in[0,T]}\|\vi{\Y}^\e(t)-\Y(t)\|_{\V^*}\bigg]=0,
		\end{align*}and for $m=\frac{\beta}{\beta+1}$, 
		\begin{align}\label{Cgs7}
			\lim_{\e\to0} \bar{\E}\bigg[\int_0^T\|\vi{\Y}^\e(t)-\Y(t)\|_\H^{2m}\d t\bigg]=0.
		\end{align}
	\end{enumerate}
\end{lemma}
\begin{proof}
	Using the uniform estimate \eqref{UE1}, an application of the Banach-Alaoglu theorem  implies (1) and (2).

	Using the Hypothesis \ref{hyp1} (H.4), H\"older's inequality, Lemmas \ref{lemUF2} and \ref{lemUF02}, we have
	\begin{align}\label{Cgs1}\nonumber
		&\bar{\E}\bigg[\int_0^T	\|\A(t,\vi{\Y}^\e(t))\|_{\V^*}^{\frac{\beta}{\beta-1}}\d t\bigg] 
		\\&\nonumber \leq \bar{\E}\bigg[\int_0^T\big(\fk{a}(t)+C\|\vi{\Y}^\e(t)\|_\V^\beta\big)\big(1+\|\vi{\Y}^\e(t)\|_\H^\alpha\big)\d t\bigg] \\&\nonumber\leq \bigg\{\bar{\E}\bigg[1+\sup_{t\in[0,T]}\|\vi{\Y}^\e(t)\|_\H^\alpha\bigg]^{\frac{\alpha+2}{\alpha}}\bigg\}^{\frac{\alpha}{\alpha+2}}\bigg\{\bar{\E}\bigg[\int_0^T \big(\fk{a}(t)+C\|\vi{\Y}^\e(t)\|_\V^\beta \big)\d t\bigg]^{\frac{\alpha+2}{2}}\bigg\}^{\frac{2}{\alpha+2}}\\&\leq C<\infty.
	\end{align}By Lemma \ref{lemUF2}, we have 
	\begin{align}\label{Cgs2}
		\bar{\E}\bigg[\sup_{t\in[0,T]}\|\vi{\Y}^\e(t)\|_\H^2\bigg] \leq C, \ \text{ and }\  \bar{\E}\bigg[\int_0^T\|\vi{\Y}^\e(t)\|_\V^\beta\d t\bigg]\leq C.
	\end{align}Hence, by the strong convergence of $\Y^\e(\omega)$ to $\Y(\omega)$ in the space $\D([0,T];\V^*)$ with respect to the supremum norm, Fatou's lemma and \eqref{Cgs2}, we have
	\begin{align}\label{Cgs3}
		\bar{\E}\bigg[\sup_{t\in[0,T]}\|\Y(t)\|_\H^2\bigg]&\leq \liminf_{\e\to0}\bar{\E}\bigg[\sup_{t\in[0,T]}\|\vi{\Y}^\e(t)\|_\H^2\bigg] \leq C,\\\label{Cgs4}
		\bar{\E}\bigg[\int_0^T\|\Y(t)\|_\V^\beta\d t\bigg]&\leq \liminf_{\e\to0}\bar{\E}\bigg[\int_0^T\|\vi{\Y}^\e(t)\|_\V^\beta\d t\bigg]\leq C, 
	\end{align}and 
	\begin{align}\label{Cgs5}
		\lim_{\e\to0}\bar{\E}\bigg[\sup_{t\in[0,T]}\|\vi{\Y}^\e(t)-\Y(t)\|_{\V^*}\bigg]=0.
	\end{align}The above convergence \eqref{Cgs5} can be justified as follows:
	
	Define \begin{align*}
		\Omega_\delta^\e:=\bigg\{\omega:\sup_{t\in[0,T]}\|\vi{\Y}^\e(t)-\Y(t)\|_{\V^*}\geq \delta\bigg\}.
	\end{align*}The strong convergence of $\vi{\Y}^\e(\omega)$ to $\Y(\omega)$ in the space $\D([0,T];\V^*)$ with respect to supremum norm implies 
	\begin{align}\label{Cgs6}
		\lim_{\e\to 0}\bar{\P}(\Omega_\delta^\e)=0, \text{ for all } \delta>0.
	\end{align}Using \eqref{Cgs2}-\eqref{Cgs6}, we deduce 
	\begin{align*}
		&	\lim_{\e\to0} \bar{\E}\bigg[\sup_{t\in[0,T]}\|\vi{\Y}^\e(t)-\Y(t)\|_{\V^*}\bigg]\\& 
		= \lim_{\e\to0} \bigg\{\bar{\E}\bigg[\sup_{t\in[0,T]}\|\vi{\Y}^\e(t)-\Y(t)\|_{\V^*}\cdot \chi_{\Omega_\delta^\e}  \bigg] +\bar{\E}\bigg[\sup_{t\in[0,T]}\|\vi{\Y}^\e(t)-\Y(t)\|_{\V^*}\cdot \chi_{({\Omega_\delta^\e})^c}\bigg]\bigg\} \\& \leq \delta+\lim_{\e\to0} \bigg\{\bar{\E}\bigg[\|\vi{\Y}^\e(t)-\Y(t)\|_{\V^*}^2\bigg]\bigg\}^{\frac{1}{2}}\cdot \big(\bar{\P}(\Omega_\delta^\e)\big)^{\frac{1}{2}} \\& \leq \delta.
	\end{align*}The arbitrariness of $\delta$ gives \eqref{Cgs5}, which is first part of (4).
	
	Choosing  $m=\frac{\beta}{\beta+1}$, we obtain 
	\begin{align}\label{543}
		&	\bar{\E}\bigg[\int_0^T\|\vi{\Y}^\e(t)-\Y(t)\|_\H^{2m}\d t\bigg]\nonumber\\&=\bar{\E}\bigg[\int_0^T|\langle\vi{ \Y}^\e(t)-\Y(t),\Y^\e(t)-\Y(t)\rangle|^m\d t\nonumber\\&\leq  \bar{\E}\bigg[\int_0^T\|\vi{\Y}^\e(t)-\Y(t)\|_{\V^*}^m\|\vi{\Y}^\e(t)-\Y(t)\|_\V^m\d t\bigg] \nonumber\\& \leq \bigg\{\bar{\E}\bigg[\int_0^T\|\vi{\Y}^\e(t)-\Y(t)\|_{\V^*}\d t \bigg]\bigg\}^{\frac{\beta-m}{\beta}}\bigg\{\bar{\E}\bigg[\int_0^T\|\vi{\Y}^\e(t)-\Y(t)\|_\V^\beta\d t \bigg]\bigg\}^{\frac{m}{\beta}}.
	\end{align}Combining \eqref{Cgs2}, \eqref{Cgs4} and \eqref{Cgs5}, one can derive \eqref{Cgs7}. 
\end{proof}
Let us first prove the convergence of the control terms. 
\begin{lemma}\label{Cgslem2}
	For any $\boldsymbol{h}\in\H$, we have 
	\begin{align}\label{Cgslem01}\nonumber
		&	\lim_{\e_j\to0}\bigg( \int_0^t\bigg\{\B(s,\vi{\Y}^{\e_j}(s))\psi_{\e_j}(s)+\int_\Z \gamma(s,\vi{\Y}^{\e_j}(s),z)\big(\varphi_{\e_j}(s,z)-1\big)\nu(\d z)\bigg\}\d s,\boldsymbol{h}\bigg)\\&=\bigg(  \int_0^t\bigg\{\B(s,\Y(s))\psi(s)+\int_\Z \gamma(s,\Y(s),z)\big(\varphi(s,z)-1\big)\nu(\d z)\bigg\}\d s,\boldsymbol{h}\bigg).
	\end{align}
\end{lemma}
\begin{proof}
	Since $\vi{\Y}^\e$ converges to $\Y$ in $\D([0,T];\V^*)$ strongly $ \bar{\P}$-a.s., we obtain  $\sup\limits_{s\in[0,T]}\|\Y(s)\|_\H<\infty,\ \bar{\P}$-a.s.  Also since $L_\gamma\in\mathcal{H}_2$ and $L_\B\in \L^1(0,T;\R^+)$, it follows from \eqref{LDP001} and Lemma \ref{lemUF2}  that 
	\begin{align}\label{Cgslem02}\nonumber
		&	\lim_{\e_j\to0}\bigg( \int_0^t\bigg\{\B(s,\Y(s))\psi_{\e_j}(s)+\int_\Z \gamma(s,\Y(s),z)\big(\varphi_{\e_j}(s,z)-1\big)\nu(\d z)\bigg\}\d s,\boldsymbol{h}\bigg)\\&=\bigg(  \int_0^t\bigg\{\B(s,\Y(s))\psi(s)+\int_\Z \gamma(s,\Y(s),z)\big(\varphi(s,z)-1\big)\nu(\d z)\bigg\}\d s,\boldsymbol{h}\bigg).
	\end{align}For any $\delta>0$. Define 
	$
	A_{\e,\delta}:=\big\{s\in[0,T]:\|\vi{\Y}^{\e}(s)-\Y(s)\|_\H>\delta\}$.  From \eqref{543}, we infer 
	\begin{align*}
		\lim_{\e\to 0}\bar{\E}\left[\lambda_{T}\left(A_{\e,\delta}\right)\right]\leq\frac{1}{\delta^{2m}}\lim_{\e\to 0}\bar{\E}\bigg[\int_0^T\|\vi{\Y}^\e(t)-\Y(t)\|_\H^{2m}\d t\bigg]=0. 
	\end{align*}
	Therefore there exists a subsequence $\{\e_j\}$ (still denoted by the same index for simplicity) such that 
	\begin{align}\label{Adelta}
		\lim_{\e_j\to0} \lambda_T(A_{\e_j,\delta})=0, \ \; \bar{\P}\text{-a.s.}
	\end{align}
	Using Hypothesis \ref{hyp2} (H.7), (H.9) and Lemma \ref{lemUF2}, we find
	\begin{align}\label{Cgslem03}\nonumber
		&\int_0^T\bigg\{\|\B(t,\vi{\Y}^{\e_j}(t))-\B(t,\Y(t))\|_{\L_2}\|\psi_{\e_j}(t)\|_\H\\&\nonumber\quad+\int_\Z \|\gamma(t,\vi{\Y}^{\e_j}(t),z)-\gamma(t,\Y(t),z)\|_\H\big|\varphi_{\e_j}(t,z)-1\big|\nu(\d z)\bigg\}\d t\\&\nonumber\leq \int_0^T \bigg\{\sqrt{L_\B(t)}\|\psi_{\e_j}(t)\|_\H+\int_\Z R_\gamma (t,z)\big|\varphi_{\e_j}(t,z)-1\big|\nu(\d z)\bigg\}\|\vi{\Y}^{\e_j}(t)-\Y(t)\|_\H\d t\\&\nonumber
		\leq \delta\int_{A_{\e_j,\delta}^c} \bigg\{\frac{1}{2}\big(L_\B(t)+\|\psi_{\e_j}(t)\|_\H^2\big)+\int_\Z R_\gamma (t,z)\big|\varphi_{\e_j}(t,z)-1\big|\nu(\d z) \bigg\}	\d t\\&\nonumber\quad + \sup_{t\in[0,T]} \|\vi{\Y}^{\e_j}(t)-\Y(t)\|_\H \int_{A_{\e_j,\delta}}\bigg\{\frac{1}{2}\big(L_\B(t)+\|\psi_{\e_j}(t)\|_\H^2\big)+\int_\Z R_\gamma (t,z)\big|\varphi_{\e_j}(t,z)-1\big|\nu(\d z) \bigg\}\d t
		\\& \nonumber\leq \delta\sup_{\q\in\mathcal{U}^\Upsilon}\bigg(\int_{A_{\e_j,\delta}^c} \bigg\{\frac{1}{2}\big(L_\B(t)+\|\psi(t)\|_\H^2\big)+\int_\Z R_\gamma (t,z)\big|\varphi(t,z)-1\big|\nu(\d z) \bigg\}	\d t\bigg)\\&\nonumber\quad +\sup_{t\in[0,T]} \|\vi{\Y}^{\e_j}(t)-\Y(t)\|_\H\sup_{\q\in\mathcal{U}^\Upsilon}\bigg(  \int_{A_{\e_j,\delta}}\bigg\{\frac{1}{2}\big(L_\B(t)+\|\psi(t)\|_\H^2\big)\\&\nonumber\qquad+\int_\Z R_\gamma (t,z)\big|\varphi(t,z)-1\big|\nu(\d z) \bigg\}\d t\bigg)\\&\nonumber \leq 
		\delta C_{L_\B,L_\gamma,\Upsilon}+\sup_{t\in[0,T]} \|\vi{\Y}^{\e_j}(t)-\Y(t)\|_\H\sup_{\q\in\mathcal{U}^\Upsilon}\bigg(  \int_{A_{\e_j,\delta}}\bigg\{\frac{1}{2}\big(L_\B(t)+\|\psi(t)\|_\H^2\big)\\&\qquad+\int_\Z R_\gamma (t,z)\big|\varphi(t,z)-1\big|\nu(\d z) \bigg\}\d t\bigg). 
	\end{align}One should note that 
	\begin{align}\label{Cgslem04}\nonumber
		&\bar{\E}\left[\sup_{t\in[0,T]} \|\vi{\Y}^{\e_j}(t)-\Y(t)\|_\H\right. \nonumber\\&\quad\left.\times\sup_{\q\in\mathcal{U}^\Upsilon}\bigg(  \int_{A_{\e_j,\delta}}\bigg\{\frac{1}{2}\big(L_\B(t)+\|\psi(t)\|_\H^2\big)+\int_\Z R_\gamma (t,z)\big|\varphi(t,z)-1\big|\nu(\d z) \bigg\}\d t\bigg)\right]\nonumber\\&\nonumber \leq \left\{\bar{\E}\bigg[\sup_{t\in[0,T]} \|\vi{\Y}^{\e_j}(t)-\Y(t)\|_\H^2\bigg]\right\}^{\frac{1}{2}}\\&\quad\times\left\{\bar{\E}\bigg[\sup_{\q\in\mathcal{U}^\Upsilon} \int_{A_{\e_j,\delta}}\bigg\{\frac{1}{2}\big(L_\B(t)+\|\psi(t)\|_\H^2\big)+\int_\Z R_\gamma (t,z)\big|\varphi(t,z)-1\big|\nu(\d z) \bigg\}\d t\bigg]^2\right\}^\frac{1}{2}. 
	\end{align}
	Using the Dominated Convergence Theorem, \eqref{Adelta}, Lemmas \ref{lemUF3} (3) and \ref{lemUF2}, we obtain 
	\begin{align}\label{Cgslem05}
		\lim_{\e_j\to0}\bar{\E}\bigg[\sup_{\q\in\mathcal{U}^\Upsilon} \int_{A_{\e_j,\delta}}\bigg\{\frac{1}{2}\big(L_\B(t)+\|\psi(t)\|_\H^2\big)+\int_\Z R_\gamma (t,z)\big|\varphi(t,z)-1\big|\nu(\d z) \bigg\}\d t\bigg]^2=0.
	\end{align}By \eqref{Cgs2}, \eqref{Cgslem03} and \eqref{Cgslem05}, we find
	\begin{align}\label{Cgslem06}\nonumber
		&\lim_{\e_j\to0}\bar{\E}\bigg[	\int_0^T\bigg\{\|\B(t,\vi{\Y}^{\e_j}(t))-\B(t,\Y(t))\|_{\L_2}\|\psi_{\e_j}(t)\|_\H\\&\quad+\int_\Z \|\gamma(t,\vi{\Y}^{\e_j}(t),z)-\gamma(t,\Y(t),z)\|_\H\big|\varphi_{\e_j}(s,z)-1\big|\nu(\d z)\bigg\}\d t\bigg]=0.
	\end{align}Hence, there exists a subsequence $\{\e_j\}$ (still denoted  by the same index) such that 
	\begin{align}\label{Cgslem07}\nonumber
		&\lim_{\e_j\to0}\int_0^T \bigg\{\|\B(t,\vi{\Y}^{\e_j}(t))-\B(t,\Y(t))\|_{\L_2}\|\psi_{\e_j}(t)\|_\H\\&\quad+\int_\Z \|\gamma(t,\vi{\Y}^{\e_j}(t),z)-\gamma(t,\Y(t),z)\|_\H\big|\varphi_{\e_j}(s,z)-1\big|\nu(\d z)\bigg\}\d t=0,\;\bar{\P}\text{-a.s.}
	\end{align}Combining \eqref{Cgslem07} with \eqref{Cgslem02}, we obtain the required result \eqref{Cgslem01}.
\end{proof}

Let us fix 
\begin{align}\label{Cgs8}
	\vi{\Y}(t):=\x+\int_0^t \wi{\A}(s)\d s+\int_0^t\B(s,\Y(s))\psi(s)\d s+\int_0^t\gamma(s,\Y(s),z)\big(\varphi(s,z)-1\big)\nu(\d z)\d s, 
\end{align}
where $\wi{\A}(\cdot)$ is the same limit  appearing in Lemma \ref{Cgslem1}. By the weak limit of \eqref{CSPDE1}, one can show that  (cf. \cite{AKMTM4})
\begin{align*}
	\vi{\Y}(t,\omega)=\bar{\Y}(t,\omega)=\Y(t,\omega),\text{ for } \d t\otimes\bar{\P}\text{-a.e. }(t,\omega).
\end{align*}
For any  $\Psi\in\L^\infty(0,T;\R^+)$, define 
\begin{align*}
	\mathcal{G}	(\Y,\q,\bfX):= &\bar{\E}\bigg[\int_0^T\Psi(t)\int_0^t \bigg(\big(\B(s,\Y(s))\psi(s),\bfX(s)\big)\\&\qquad + \bigg(\int_\Z\gamma(s,\Y(s),z)(\varphi(s,z)-1)\nu(\d z),\bfX(s)\big)\bigg)\d s\d t\bigg].
\end{align*}
\begin{lemma}\label{Cgslem3}
	The following limit holds true: 
	\begin{align}\label{Cgslem3.1}
		\lim_{\e_j\to0} \mathcal{G}(\vi{\Y}^{\e_j},\q_{\e_j},\vi{\Y}^{\e_j})=\mathcal{G}(\Y,\q,\Y).
	\end{align}
\end{lemma}
\begin{proof}
	Using Lemma \ref{lemUF} and the fact that $\sup\limits_{s\in[0,T]}\|\Y(s)\|_\H<\infty,\;\bar{\P}$-a.s., we have for all $(t,\omega)\in[0,T]\times \Omega$, 
	\begin{align*}
		&	\lim_{\e_j\to0} \Psi(t) \int_0^t\bigg(\big(\B(s,\Y(s))\psi_{\e_j}(s),\Y(s)\big)+\int_\Z \gamma(s,\Y(s),z)\big(\varphi_{\e_j}(s,z)-1\big)\nu(\d z),\Y(s)\bigg)\d s\\&=\Psi(t) \int_0^t\bigg(\big(\B(s,\Y(s))\psi(s),\Y(s)\big)\int_\Z \gamma(s,\Y(s),z)\big(\varphi(s,z)-1\big)\nu(\d z),\Y(s)\bigg)\d s.
	\end{align*}Using Lemma \ref{lemUF2}, we find 
	\begin{align*}
		&\sup_{\q\in\mathcal{U}^\Upsilon}\bigg|\Psi(t)\int_0^t\bigg\{  \big(\B(s,\Y(s))\psi(s),\Y(s)\big)+\bigg(\int_\Z\gamma(s,\Y(s),z)\big(\varphi(s,z)-1\big)\nu(\d z)\d s,\Y(s)\bigg)   \bigg|\\& \leq C_\Psi \sup_{\q\in\mathcal{U}^\Upsilon}\int_0^T \bigg(\|\B(s,\Y(s))\|_{\L_2}\|\psi(s)\|_\H\|\Y(s)\|_\H\\&\qquad+\int_\Z\|\gamma(s,\Y(s),z)\|_\H\|\Y(s)\|_\H\big|\varphi(s,z)-1\big|\nu(\d z)\d s\bigg) \\&\leq
		C_\Psi\bigg(1+2\sup_{s\in[0,T]}\|\Y(s)\|_\H^2\bigg)\sup_{\q\in\mathcal{U}^\Upsilon}\bigg\{\int_0^T\fk{b}(s)\d s+\int_0^T\|\psi(s)\|_\H^2\d s\\&\qquad +\bigg(\int_{\Z_T} L_\gamma(s,z)\big|\varphi(s,z)-1\big|\nu(\d z)\d s\bigg)\bigg\}
		\\& \leq C_{\Psi,L_\gamma,\Upsilon,\fk{b}}\bigg(1+2\sup_{s\in[0,T]}\|\Y(s)\|_\H^2\bigg),
	\end{align*}where we have used a calculation similar to  \eqref{UE18} for the first term and Hypothesis \ref{hyp2} (H.8) for the second term. 
	Applying the Dominated Convergence Theorem, we get
	\begin{align}\label{Cgs9}
		\lim_{\e_j\to0}\mathcal{G}(\Y,\q_{\e_j},\Y)=\mathcal{G}(\Y,\q,\Y).
	\end{align}For any $\delta>0$, recall 
	$	A_{\e_j,\delta}:=\big\{s\in[0,T]:\|\vi{\Y}^{\e_j}(s)-\Y(s)\|_\H>\delta\}$, and by \eqref{Adelta} there exists a subsequence $\{\e_j\}$ (still denoted  by the same index) such that 
	\begin{align*}
		\lim_{\e_j\to0} \lambda_T(A_{\e_j,\delta})=0,\; \bar{\P}\text{-a.s.}
	\end{align*} Then, we have 
	\begin{align}\label{Cgs10}\nonumber
		&	\big|\mathcal{G}(\vi{\Y}^{\e_j},\q_{\e_j},\vi{\Y}^{\e_j})-\mathcal{G}(\vi{\Y}^{\e_j},\q_{\e_j},\Y)\big| 
		\\&\nonumber\leq C_{\Psi}\bar{\E}\bigg[\int_0^T\bigg(\|\B(s,\vi{\Y}^{\e_j}(s))\|_{\L_2}\|\psi_{\e_j}(s)\|_\H\|\vi{\Y}^{\e_j}(s)-\Y(s)\|_\H\\&\qquad+\int_\Z\|\gamma(s,\vi{\Y}^{\e_j}(s),z)\|_\H\big|\varphi_{\e_j}(s,z)-1\big|\|\vi{\Y}^{\e_j}(s)-\Y(s)\|_\H\nu(\d z)\bigg)\d s\bigg]. 
	\end{align}We consider the first term in the right hand side of the inequality \eqref{Cgs10} and estimate it using Hypothesis \ref{hyp1} (H.5), H\"older's and Young's inequalities as
	\begin{align}\label{Cgs11}\nonumber
		&	\bar{\E}\bigg[\int_0^T\|\B(s,\vi{\Y}^{\e_j}(s))\|_{\L_2}\|\psi_{\e_j}(s)\|_\H\|\vi{\Y}^{\e_j}(s)-\Y(s)\|_\H\d s\bigg]\\& \nonumber
		\leq C\bar{\E}\bigg[\int_0^T \sqrt{\fk{b}(s)}\big(1+\|\vi{\Y}^{\e_j}(s)\|_\H\big)\|\psi_{\e_j}(s)\|_\H\|\vi{\Y}^{\e_j}(s)-\Y(s)\|_\H\d s\bigg]\\&\nonumber
		\leq C\delta \bar{\E}\bigg[\int_{A_{\e_j,\delta}^c}\sqrt{\fk{b}(s)}\big(1+\|\vi{\Y}^{\e_j}(s)\|_\H\big)\|\psi_{\e_j}(s)\|_\H\d s\bigg]\\&\nonumber\quad+ 
		C \bar{\E}\bigg[\int_{A_{\e_j,\delta}}\sqrt{\fk{b}(s)}\big(1+\|\vi{\Y}^{\e_j}(s)\|_\H\big)\|\psi_{\e_j}(s)\|_\H\|\vi{\Y}^{\e_j}(s)-\Y(s)\|_\H\d s\bigg]\\&\nonumber
		\leq C\delta \bar{\E}\bigg[\sup_{s\in[0,T]}\big(1+\|\vi{\Y}^{\e_j}(s)\|_\H\big)\sup_{\psi\in \mathcal{U}^\Upsilon}\bigg(\int_0^T \big\{\fk{b}(s)+\|\psi(s)\|_\H^2\big\}\d s\bigg)\bigg]\\&\nonumber\quad+ 
		C \bar{\E}\bigg[\sup_{s\in[0,T]}\left\{\big(1+\|\vi{\Y}^{\e_j}(s)\|_\H\big)\|\vi{\Y}^{\e_j}(s)-\Y(s)\|_\H\right\}\sup_{\psi\in \vi{S}^\Upsilon}\bigg(\int_{A_{\e_j,\delta}}\big\{\fk{b}(s)+\|\psi(s)\|_\H^2\big\}\d s\bigg)\bigg]\\&\nonumber\leq 
		\delta C_{\fk{b},\Upsilon}+ C\bigg\{1+\bar{\E}\bigg[\sup_{s\in[0,T]}\|\vi{\Y}^{\e_j}(s)\|_\H^4\bigg]\bigg\}^{\frac{1}{4}}\bigg\{\bar{\E}\bigg[\sup_{s\in[0,T]}\|\vi{\Y}^{\e_j}(s)-\Y(s)\|_\H^4\bigg]\bigg\}^{\frac{1}{4}}\\&\qquad \times \bigg\{\bar{\E}\bigg[\sup_{\psi\in \vi{S}^\Upsilon}\int_{A_{\e_j,\delta}}\big\{\fk{b}(s)+\|\psi(s)\|_\H^2\big\}\d s\bigg]^2\bigg\}^{\frac{1}{2}}.
	\end{align}We estimate the final term in the right hand side of \eqref{Cgs10}, using Hypothesis \ref{hyp2} (H.8), and H\"older's inequality as
	\begin{align}\label{Cgs12}\nonumber
		&\bar{\E}\bigg[\int_0^T	\int_\Z\|\gamma(s,\vi{\Y}^{\e_j}(s),z)\|_\H\big|\varphi_{\e_j}(s,z)-1\big|\|\vi{\Y}^{\e_j}(s)-\Y(s)\|_\H\nu(\d z)\d s\bigg] \\&\nonumber\leq 
		C\bar{\E}\bigg[\int_0^T	\int_\Z L_\gamma(s,z)\big(1+\|\vi{\Y}^{\e_j}(s)\|_\H\big)\big|\varphi_{\e_j}(s,z)-1\big|\|\vi{\Y}^{\e_j}(s)-\Y(s)\|_\H\nu(\d z)\d s\bigg]\\&\nonumber \leq  C\delta \bar{\E}\bigg[\int_{A_{\e_j,\delta}^c}	\int_\Z L_\gamma(s,z)\big(1+\|\vi{\Y}^{\e_j}(s)\|_\H\big)\big|\varphi_{\e_j}(s,z)-1\big|\nu(\d z)\d s\bigg]\\&\nonumber\quad+C\bar{\E}\bigg[\int_{A_{\e_j,\delta}}	\int_\Z L_\gamma(s,z)\big(1+\|\vi{\Y}^{\e_j}(s)\|_\H\big)\big|\varphi_{\e_j}(s,z)-1\big|\|\vi{\Y}^{\e_j}(s)-\Y(s)\|_\H\nu(\d z)\d s\bigg] \\&\nonumber\leq C\delta\bigg(1+\bar{\E} \bigg[\sup_{s\in[0,T]}\|\vi{\Y}^{\e_j}(s)\|_\H\bigg]\bigg)\bigg(\sup_{\varphi\in S^\Upsilon}\int_{\Z_T} L_\gamma(s,z)\big|\varphi(s,z)-1\big|\nu(\d z)\d s\bigg)\\&\nonumber\quad+C \bar{\E}\bigg[\sup_{s\in[0,T]}\big(1+\|\vi{\Y}^{\e_j}(s)\|_\H\big)\big(\|\vi{\Y}^{\e_j}(s)-\Y(s)\|_\H\big)\\&\nonumber\qquad\times\bigg(\sup_{\varphi\in S^\Upsilon}\int_{A_{\e_j,\delta}}\int_\Z L_\gamma(s,z)\big|\varphi(s,z)-1\big|\nu(\d z)\d s\bigg)\bigg]\\&\nonumber\leq 
		\delta C_{L_\gamma,\Upsilon}+C \bigg\{1+\bar{\E}\bigg[\sup_{s\in[0,T]}\|\vi{\Y}^{\e_j}(s)\|_\H^4\bigg]\bigg\}^{\frac{1}{4}}\bigg\{\bar{\E}\bigg[\sup_{s\in[0,T]}\|\vi{\Y}^{\e_j}(s)-\Y(s)\|_\H^4\bigg]\bigg\}^{\frac{1}{4}}\\&\qquad \times \bigg\{\bar{\E}\bigg[\sup_{\varphi\in S^\Upsilon}\int_{A_{\e_j,\delta}}\int_\Z L_\gamma(s,z)\big|\varphi(s,z)-1\big|\nu(\d z)\d s\bigg]^2\bigg\}^{\frac{1}{2}}.
	\end{align}Substituting \eqref{Cgs11} and \eqref{Cgs12} in \eqref{Cgs10}, we obtain 
	\begin{align}\label{Cgs13}
		\lim_{\e_j\to0} \big|\mathcal{G}(\vi{\Y}^{\e_j},\q_{\e_j},\vi{\Y}^{\e_j})-\mathcal{G}(\vi{\Y}^{\e_j},\q_{\e_j},\Y)\big| =0.
	\end{align}On the other hand, we have 
	\begin{align}\label{Cgs14}\nonumber
		&	\big|\mathcal{G}(\vi{\Y}^{\e_j},\q_{\e_j},\Y)-\mathcal{G}(\Y,\q_{\e_j},\Y)\big| \\& \nonumber\leq  C\bar{\E}\bigg[\int_0^T\bigg\{\|\B(s,\vi{\Y}^{\e_j}(s))-\B(s,\Y(s))\|_{\L_2}\|\psi_{\e_j}(s)\|_\H \|\Y(s)\|_\H\\&\nonumber\qquad+\int_\Z\|\gamma(s,\vi{\Y}^{\e_j}(s),z)-\gamma(s,\Y(s),z)\|_\H\big|\varphi_{\e_j}(s,z)-1\big|\|\Y(s)\|_\H\nu(\d z)\bigg\}\d s\bigg] 
		\\&\nonumber\leq C\bar{\E}\bigg[\int_0^T\bigg\{L_\B(s)\|\vi{\Y}^{\e_j}(s)-\Y(s)\|_\H\|\psi_{\e_j}(s)\|_\H \|\Y(s)\|_\H\\&\qquad+\int_\Z R_\gamma(s,z)\|\vi{\Y}^{\e_j}(s)-\Y(s)\|_\H \big|\varphi_{\e_j}(s,z)-1\big|\|\Y(s)\|_\H\nu(\d z)\bigg\}\d s\bigg] .
	\end{align}A similar calculation to \eqref{Cgs13} implies 
	\begin{align}\label{Cgs15}
		\lim_{\e_j\to0}\big|\mathcal{G}(\vi{\Y}^{\e_j},\q_{\e_j},\Y)-\mathcal{G}(\Y,\q_{\e_j},\Y)\big| =0.
	\end{align}Combining \eqref{Cgs9}, \eqref{Cgs13} and \eqref{Cgs15}, we obtain \eqref{Cgslem3.1}, which completes the proof.
\end{proof}

\begin{remark}
	Since $\Psi\in\mathrm{L}^{\infty}(0,T;\R^+)$ is arbitrary, from Lemma \ref{Cgslem3}, along a further subsequence (still denoted by the same index for simplicity) one can obtain  
	\begin{align}\label{561}
		&	\lim_{\e_j\to 0}	\bar{\E}\bigg[\int_0^T\bigg\{\big(\B(s,\vi{\Y}^{\e_j}(s))\psi_{\e_j}(s),\vi{\Y}^{\e_j}(s)\big)\nonumber\\&\qquad+\int_\Z\bigg( \gamma(s,\vi{\Y}^{\e_j}(s),z)\big(\varphi_{\e_j}(s,z)-1\big),\vi{\Y}^{\e_j}(s)\bigg)\nu(\d z)\bigg\}\d s\bigg]\nonumber\\&= \bar{\E}\bigg[\int_0^T\bigg\{\big(\B(s,\Y(s))\psi(s),\Y(s)\big)+\int_\Z\bigg( \gamma(s,\Y(s),z)\big(\varphi(s,z)-1\big),\Y(s)\bigg)\nu(\d z)\bigg\}\d s\bigg]. 
	\end{align}
\end{remark}

Our next aim is to identify the limit function $\wi{\A}(\cdot)$. In view of Lemma \ref{Cgslem00}, it is enough to show the condition \eqref{Cgslem001}. Now, our goal is to verify the condition \eqref{Cgslem001}.
\begin{lemma}\label{Cgslem4}
	We have 
	\begin{align*}
		\wi{\A}(\cdot)=\A(\cdot,\Y(\cdot)),\;\d t\otimes\bar{\P}\text{-a.e. }
	\end{align*}
\end{lemma}
\begin{proof}
	
	Applying infinite dimensional It\^o's formula to the process $\|\vi{\Y}^{\e_j}(\cdot)\|_\H^2$ (see Theorem 1.2, \cite{IGDS}), we get
	\begin{align}\label{ALA1}\nonumber
		&\|\vi{\Y}^{\e_j}(t)\|_\H^2-\|\x\|_\H^2 \\&\nonumber
		=2\int_0^t \langle \A(s,\vi{\Y}^{\e_j}(s)),\vi{\Y}^{\e_j}(s)\rangle 
		+\big(\B(s,\vi{\Y}^{\e_j}(s))\psi_{\e_j}(s),\vi{\Y}^{\e_j}(s)\big)	\\&\qquad\nonumber+\int_\Z\bigg( \gamma(s,\vi{\Y}^{\e_j}(s),z)\big(\varphi_{\e_j}(s,z)-1\big),\vi{\Y}^{\e_j}(s)\bigg)\bigg\}\d s\\&\nonumber\quad+\e_j\int_0^t \|\B(s,\vi{\Y}^{\e_j}(s))\|_{\L_2}^2\d s+\sqrt{\e_j}\int_0^t \big(\B(s,\vi{\Y}^{\e_j}(s))\d\W(s),\vi{\Y}^{\e_j}(s)\big)\\&\nonumber\quad 
		+\e_j^2\int_0^t \int_\Z \|\gamma(s,\vi{\Y}^{\e_j}(s-),z)\|_\H^2N^{\e_j^{-1}\varphi_{\e_j}}(\d s,\d z)\\&\quad + 
		2\e_j\int_0^t \int_\Z \big(\gamma(s,\vi{\Y}^{\e_j}(s-),z),\vi{\Y}^{\e_j}(s-)\big)\vi{N}^{\e_j^{-1}\varphi_{\e_j}}(\d s,\d z),\ \bar{\P}\text{-a.s.},
	\end{align}
	for all $t\in[0,T]$. Notice that 
	\begin{align*}
		M_1^{\e_j}(t)&:=\sqrt{\e_j}\int_0^t \big(\B(s,\vi{\Y}^{\e_j}(s))\d\W(s),\vi{\Y}^{\e_j}(s)\big), \ \text{ and  }\\
		M_2^{\e_j}(t)&:=2\e_j\int_0^t \int_\Z \big(\gamma(s,\vi{\Y}^{\e_j}(s-),z),\vi{\Y}^{\e_j}(s-)\big)\vi{N}^{\e_j^{-1}\varphi_{\e_j}}(\d s,\d z),
	\end{align*}
	are square integrable martingales. Taking expectations on both sides of the expression \eqref{ALA1}, we find
	\begin{align}\label{ALA2}\nonumber
		&\bar{\E}\big[\|\vi{\Y}^{\e_j}(T)\|_\H^2\big]-\|\x\|_\H^2 
		\\&\nonumber=
		2\bar{\E}\bigg[\int_0^T \bigg\{\langle \A(s,\vi{\Y}^{\e_j}(s)),\vi{\Y}^{\e_j}(s)\rangle  \bigg\}\d s\bigg]
		\\&\nonumber\quad+2\bar{\E}\bigg[\int_0^T\bigg\{\big(\B(s,\vi{\Y}^{\e_j}(s))\psi_{\e_j}(s),\vi{\Y}^{\e_j}(s)\big)\nonumber\\&\qquad+\int_\Z\bigg( \gamma(s,\vi{\Y}^{\e_j}(s),z)\big(\varphi_{\e_j}(s,z)-1\big),\vi{\Y}^{\e_j}(s)\bigg)\nu(\d z)\bigg\}\d s\bigg]\nonumber\\&\nonumber\quad+
		\e_j\bar{\E}\bigg[\int_0^T    \|\B(s,\vi{\Y}^{\e_j}(s))\|_{\L_2}^2\d s\bigg]
		+\e_j\bar{\E}\bigg[\int_{\Z_T} \|\gamma(s,\vi{\Y}^{\e_j}(s),z)\|_\H^2\varphi_{\e_j}(s,z)\nu(\d z)\d s\bigg] \\&=: \sum_{i=1}^{4}J_i(t).
	\end{align}Using Hypothesis \ref{hyp1} (H.5), H\"older's inequality and Lemma \ref{lemUF2}, we estimate the term $J_3(\cdot)$ as
	\begin{align}\label{ALA3}\nonumber
		J_3(t) &\leq \e_j \bar{\E}\bigg[\int_0^T\fk{b}(s)\big(1+\|\vi{\Y}^{\e_j}(s)\|_\H^2\big)\d s\bigg] \\&\nonumber\leq \e_j \bigg(1+\bar{\E}\bigg[\sup_{s\in[0,T]}\|\vi{\Y}^{\e_j}(s)\|_\H^2\bigg]\bigg)\bigg(\int_0^T\fk{b}(s)\d s\bigg) \\&\leq \e_j C.
	\end{align}Now, we consider $J_4(\cdot)$, and estimate it using  Hypothesis \ref{hyp2} (H.8), H\"older's inequality, Lemmas \ref{lemUF} and \ref{lemUF3} as
	\begin{align}\label{ALA4}\nonumber
		J_4(t) &\leq \e_j \bar{\E}\bigg[\int_{\Z_T}L^2_\gamma(s,z)\big(1+\|\vi{\Y}^{\e_j}(s)\|_\H\big)^2\varphi_{\e_j}(s,z)\nu(\d z)\d s\bigg] \\&\nonumber\leq \e_j \bigg(1+\bar{\E}\bigg[\sup_{s\in[0,T]}\|\vi{\Y}^{\e_j}(s)\|_\H^2\bigg]\bigg)\bigg(\sup_{\varphi\in S^\Upsilon}\int_0^TL^2_\gamma(s,z)\varphi(s,z)\nu(\d z)\d s\bigg) \\&\leq \e_j C_{L_\gamma,2,2,\Upsilon}.
	\end{align}
	Substituting \eqref{ALA3} and \eqref{ALA4} in \eqref{ALA2}, and using Fubini's theorem and  \eqref{561}, we obtain
	\begin{align}\label{ALA5}\nonumber
		&	\liminf_{\e_j\to0}\bar{\E}\big[\|\vi{\Y}^{\e_j}(T)\|_\H^2-\|\x\|_\H^2\big] \\&\nonumber=2
		\liminf_{\e_j\to0}\bar{\E}\bigg[\int_0^T\langle \A(s,\vi{\Y}^{\e_j}(s)),\vi{\Y}^{\e_j}(s)\rangle \d s\bigg] \\&\nonumber\quad
		+2\bar{\E}\bigg[\int_0^T\bigg\{\big(\B(s,\Y(s))\psi(s),\Y(s)\big)\\&\qquad+\int_\Z\bigg( \gamma(s,\Y(s),z)\big(\varphi(s,z)-1\big),\Y(s)\bigg)\nu(\d z)\bigg\}\d s\bigg].
	\end{align}
	Taking the inner product with $\Y(\cdot)$ in \eqref{Cgs8}, we have 
	\begin{align}\label{ALA6}\nonumber
		\bar{\E}\big[	\|\Y(T)\|_\H^2 \big]&\nonumber=\|\x\|_\H^2+2\bar{\E}\bigg[\int_0^T \bigg\{\langle\wi{\A}(t),\Y(s)\rangle+\big(\B(s,\Y(s))\psi(s),\Y(s)\big)\\&\quad +\bigg(\int_\Z\gamma(t,\Y(s),z)\big(\varphi(s,z)-1\big)\nu(\d z),\Y(s)\bigg)\bigg\}\d s \bigg].
	\end{align}
	Using the fact that $\|\cdot\|_{\H}$ and $\|\cdot\|_{\V}$  are lower semicontinuous in $\V^*$, the convergence $\vi{\Y}^\e(\omega)\to\Y({\omega})$ in $\D([0,T];\V^*)$  and Fatou's lemma, yield 
	\begin{align}\label{ALA7}
		\bar{\E}\big[\|\Y(T)\|_\H^2\big] -\|\x\|_\H^2&\leq \liminf_{\e_j\to0}\bar{\E}\big[\|\vi{\Y}^{\e_j}(T)\|_\H^2\big]-\|\x\|_\H^2.
	\end{align} Combining \eqref{ALA5}-\eqref{ALA7},  we arrive at
	\begin{align}\label{ALA78}
		\bar{\E}\bigg[\int_0^T \langle \wi{\A}(s),\Y(s)\rangle \d s\bigg] \leq \liminf_{\e_j\to0} \bar{\E}\bigg[\int_0^T \langle \A(s,\vi{\Y}^{\e_j}(s)),\vi{\Y}^{\e_j}(s)\rangle \d s\bigg],
	\end{align}which verify the required condition \eqref{Cgslem001} of Lemma \ref{Cgslem00}.
	Therefore $\wi{\A}(\cdot)=\A(\cdot,\Y(\cdot)),\; \d t\otimes \bar{\P}$-a.e., as required. 
\end{proof}

\begin{proposition}\label{Prop.EX}
	The process $\Y(\cdot,\omega)$ solves the following equation:
	\begin{align}\label{Prop.EX1}\nonumber
		\Y(t,\omega)&=\x+\int_0^t \A(s,\Y(s,\omega))\d s +\int_0^t\B(s,\Y(s,\omega))\psi(s)\d s\\&\quad +\int_0^t\int_\Z \gamma(s,\Y(s,\omega),z)\big(\varphi(s,z)-1\big)\nu(\d z)\d s,
	\end{align}which has a unique solution in $\C([0,T];\H)\cap \L^\beta(0,T;\V)$, for $\beta\in(1,\infty)$.
\end{proposition}
\begin{proof}
	The proof follows from Lemmas \ref{Cgslem1}-\ref{Cgslem4} (see the proof of Theorem \ref{CPDE1} also). 
\end{proof}For simplicity of notations, in this step, we suppress the dependency of the second variable  in each of the processes.
\begin{lemma}\label{Cgslem6}
	There exists a subsequence indexed by $\varpi_j$, such that the following holds:
	\begin{align}\label{Cgslem61}
		\lim_{\varpi_j\to0}\sup_{t\in[0,T]}\|\vi{\Y}^{\varpi_j}(t)-\Y(t)\|_\H^2=0,\;\bar{\P}\text{-a.s.}
	\end{align}
\end{lemma}
\begin{proof}
	Setting $\vi{\bfZ}^{\e_j}(\cdot)=\vi{\Y}^{\e_j}(\cdot)-\Y(\cdot)$. Applying It\^o's formula to the process $\|\vi{\bfZ}^{\e_j}(\cdot)\|_\H^2$, we obtain 
	\begin{align}\label{Cgslem62}\nonumber
		&	\|\vi{\bfZ}^{\e_j}(t)\|_\H^2\\&\nonumber= 2\int_0^t\langle \A(s,\vi{\Y}^{\e_j}(s))-\A(s,\Y(s)) ,\vi{\bfZ}^{\e_j}(s)\rangle \d s\\&\nonumber\quad +2\int_0^t\big(\B(s,\vi{\Y}^{\e_j}(s))\psi_{\e_j}(s)-\B(s,\Y(s))\psi(s),\vi{\bfZ}^{\e_j}(s)\big)+\e_j\int_0^T\|\B(s,\vi{\Y}^{\e_j}(s))\|_{\L_2}^2\d s \\&\nonumber\quad +2\sqrt{\e_j}\int_0^t\big(\B(s,\vi{\Y}^{\e_j}(s))\d \W(s),\vi{\bfZ}^{\e_j}(s)\big) \\&\nonumber\quad+2\int_0^t\int_\Z\bigg( \gamma(s,\vi{\Y}^{\e_j}(s),z)\big(\varphi_{\e_j}(s,z)-1\big)-\gamma(s,\Y(s),z)\big(\varphi(s,z)-1\big),\vi{\bfZ}^{\e_j}(s)\bigg)\nu(\d z)\d s\\&\nonumber\quad 
		+ 2\e_j\int_0^t\int_\Z\big(\gamma(s,\vi{\Y}^{\e_j}(s),z),\vi{\bfZ}^{\e_j}\big)\vi{N}^{\e_j^{-1}\varphi_{\e_j}}(\d s,\d z)\nonumber\\&\quad+ \e^2_j \int_0^t\int_\Z\|\gamma(s,\vi{\Y}^{\e_j}(s),z)\|_\H^2N^{\e_j^{-1}\varphi_{\e_j}}(\d s,\d z)\nonumber\\&:=\sum_{i=1}^{7}J_i(t).
	\end{align}Since we know that $\lim\limits_{\e_j\to0}\A(\cdot,\vi{\Y}^{\e_j}(\cdot))=\wi{\A}(\cdot)=\A(\cdot,\Y(\cdot)),\ \d t\otimes\bar{\P}\text{-a.e.}$, then 
	\begin{align}\label{Cgslem63}
		J_1(t)\to0,\text{ as } \e_j\to0.
	\end{align}Using \eqref{Cgs11}-\eqref{Cgs13}, we obtain 
	\begin{align}\label{Cgslem64}\nonumber
		&	\bar{\E}\bigg[\sup_{t\in[0,T]}\bigg|\int_0^t\bigg(\B(s,\vi{\Y}^{\e_j}(s))\psi_{\e_j}(s)+\int_\Z\gamma(s,\vi{\Y}^{\e_j}(s),z)\big(\varphi(s,z)-1\big)\nu(\d z), \vi{\bfZ}^{\e_j}(s) \bigg)\d s\bigg| \bigg] \\&\nonumber \leq \bar{\E}\bigg[ \int_0^T\bigg(\|\B(s,\vi{\Y}^{\e_j}(s))\|_{\L_2}\|\psi_{\e_j}(s)\|_\H+\int_\Z \|\gamma(s,\vi{\Y}^{\e_j}(s),z)\|_\H\big|\varphi(s,z)-1\big|\nu(\d z)\bigg)\|\vi{\bfZ}^{\e_j}(s)\|_\H \d s \bigg]
		\\&\nonumber \leq \bar{\E}\bigg[ \int_0^T\bigg(\sqrt{\fk{b}(s)}\|\psi_{\e_j}(s)\|_\H+\int_\Z L_\gamma\big|\varphi(s,z)-1\big|\nu(\d z)\bigg)\big(1+\|\vi{\Y}^{\e_j}(s)\|_\H\big)\|\vi{\bfZ}^{\e_j}(s)\|_\H \d s \bigg]\\&\to 0\ \text{ as } \ \e_j\to0.
	\end{align}Thus, we obtain 
	\begin{align}\label{Cgslem65}
		\lim_{\e_j\to0}\bar{\E}\bigg[\sup_{t\in[0,T]}|J_2(t)|\bigg]=0 \ \text{ and }\  \lim_{\e_j\to0}\bar{\E}\bigg[\sup_{t\in[0,T]}|J_5(t)|\bigg]=0.
	\end{align}We can pass   $\e_j\to 0$ in $J_3(\cdot)$ and $J_7(\cdot)$ as follows:
	\begin{align}\label{Cgslem66}\nonumber
		\bar{\E}\bigg[\sup_{t\in[0,T]}|J_3(t)|\bigg] &\leq \e_j\bar{\E}\bigg[\int_0^T \fk{b}(s)\big(1+\|\vi{\Y}^{\e_j}(s)\|^2\big)\d s\bigg]\\&\nonumber\leq \e_j\bigg(1+\bar{\E}\bigg[\sup_{s\in[0,T]}\|\vi{\Y}^{\e_j}(s)\|^2\bigg]\bigg) \bigg(\int_0^T \fk{b}(s)\d s\bigg)\\&\to 0\  \text{ as } \ \e_j\to0,\\\nonumber
		\bar{\E}\bigg[\sup_{t\in[0,T]}|J_7(t)|\bigg]  &\leq \e_j \bar{\E}\bigg[\int_{\Z_T} \|\gamma(s,\vi{\Y}^{\e_j}(s),z)\|_\H^2\varphi_{\e_j}(s,z)\nu(\d z)\d s\bigg]\\&\nonumber\leq \e_j \bigg(1+\bar{\E}\bigg[\sup_{s\in[0,T]}\|\vi{\Y}^{\e_j}(s)\|^2\bigg]\bigg)  \bigg(\sup_{\varphi\in S^\Upsilon} \int_{\Z_T} L^2_\gamma(s,z)\varphi(s,z)\nu(\d z)\d s\bigg)\\&\label{Cgslem67}
		\to 0\ \text{ as }\ \e_j\to0.
	\end{align}For $J_4(\cdot)$ and $J_6(\cdot)$, we use the Burkholder-Davis-Gundy and H\"older's inequalities as follows:
	\begin{align}\label{Cgslem68}\nonumber
		&	\bar{\E}\bigg[\sup_{t\in[0,T]}|J_4(t)|\bigg]\nonumber\\& \leq \sqrt{\e_j}C\bar{\E}\bigg[\int_0^T\|\B(s,\vi{\Y}^{\e_j}(s))\|_{\L_2}^2\|\vi{\bfZ}^{\e_j}(s)\|_\H^2\d s\bigg]^{\frac{1}{2}}\\&\nonumber \leq 
		\sqrt{\e_j}C\bar{\E}\bigg[\sup_{s\in[0,T]}\|\vi{\bfZ}^{\e_j}(s)\|_\H\bigg(\int_0^T\fk{b}(s)\big(1+\|\vi{\Y}^{\e_j}(s)\|_\H^2\big)\d s\bigg)^{\frac{1}{2}}\bigg]
		\\&\nonumber \leq 
		\sqrt{\e_j}C\bigg(\int_0^T\fk{b}(s)\d s\bigg)^{\frac{1}{2}}\bigg\{\bar{\E}\bigg[\sup_{s\in[0,T]}\|\vi{\bfZ}^{\e_j}(s)\|_\H^2\bigg]\bigg\}^{\frac{1}{2}}\bigg\{1+\bar{\E}\bigg[\sup_{s\in[0,T]}\|\vi{\Y}^{\e_j}(s)\|_\H^2\bigg]\bigg\}^{\frac{1}{2}}\\&\to0\ \text{ as } \ \e_j\to0,
	\end{align}and 
	\begin{align}\label{Cgslem69}\nonumber
		\bar{\E}\bigg[\sup_{t\in[0,T]}|J_6(t)|\bigg]&\leq \e_jC \bar{\E}\bigg[\int_{\Z_T} \|\gamma(s,\vi{\Y}^{\e_j}(s),z)\|_\H^2\|\vi{\bfZ}^{\e_j}(s)\|_\H^2 N^{\e_j^{-1}\varphi_{\e_j}}(\d s,\d z) \bigg]^{\frac{1}{2}}\\&\nonumber \leq 
		\e_jC \bigg\{\bar{\E}\bigg[\sup_{t\in[0,T]}\|\vi{\bfZ}^{\e_j}(s)\|_\H^2\bigg]\bigg\}^{\frac{1}{2}}\bigg\{\bigg(1+\bar{\E}\bigg[\sup_{s\in[0,T]}\|\vi{\Y}^{\e_j}(s)\|_\H^2\bigg]\bigg)\\&\nonumber\qquad\times\bigg(\sup_{\varphi\in S^\Upsilon}\int_{\Z_T}L^2_\gamma(s,z)\varphi(s,z)\nu(\d z)\d s \bigg)\bigg\}^{\frac{1}{2}}	\\&\to0\  \text{ as } \ \e_j\to0.		
	\end{align}Substituting \eqref{Cgslem63}-\eqref{Cgslem69} in \eqref{Cgslem62}, we obtain 
	\begin{align}\label{Cgslem70}
		\lim_{\e_j\to0} \bar{\E}\bigg[\sup_{t\in[0,T]}\|\vi{\bfZ}^{\e_j}(t)\|_\H^2\bigg]=0.
	\end{align}Therefore, there exists a subsequence indexed by $\varpi_j$ such that $\vi{\Y}^{\varpi_j}$ converges to $\Y,\; \bar{\P}$-a.s., for all $t\in[0,T]$. 
\end{proof}
\subsection{Verification of Condition \ref{Cond1} (2)}
Let us now verify the condition \ref{Cond1} (2). First we   recall \eqref{Vf}, the solution representation. 
\begin{theorem}\label{VCond2}
	Fixed $\Upsilon\in\N$, and let $\q_\e=(\psi_\e,\varphi_\e)$ and $\q=(\psi,\varphi)\in \mathcal{U}^\Upsilon$ be such that $\q_\e$ converges  in distribution  to $\q$ as $\e\to0$. Then, 
	\begin{align*}
		\mathscr{G}^\e\left(\sqrt{\e}\W(\cdot)+\int_0^{\cdot}\psi_\e\d s,\e N^{\e^{-1}\varphi_\e}\right) \Rightarrow \mathscr{G}^0\left(\int_0^{\cdot}\psi(s)\d s,\nu_T^\varphi\right).
	\end{align*}
\end{theorem}
\begin{proof} Recall $\bar{\mb{M}}$ from Section \ref{Sec2}, and the notations from Proposition \ref{propUF}. Denote 
	\begin{align*}
		\Pi= \big(\vi{\mathcal{U}}^\Upsilon,\D([0,T];\V^*),\C([0,T];\V^*),\C([0,T];\V^*),\C([0,T];\V^*),\C([0,T];\H_1),\bar{\mb{M}}\big),
	\end{align*}where $\H_1$ is a Hilbert space such that the embedding $\H\subset \H_1$ is Hilbert-Schmidt.
	From Proposition \ref{propUF}, we know that the laws of $\big\{\big(\q_\e,M_1^\e,M_2^\e,X_1^\e,X_2^\e,X_3^\e, \W,\bar{N}\big)\}$ is tight in $\Pi$. Let $\big(\q,0,0,X_1,X_2,X_3, \W,\bar{N}\big)$ be any limit point of the family defined above. Using Skorokhod's representation theorem (see Theorem A.1, \cite{PNKTRT}, or Theorem C.1, \cite{ZBWLJZ}), the construction of a new probability space is possible and we have the following convergences:
	\begin{align}\label{582}
		\big(\q_\e,M_1^\e,M_2^\e,X_1^\e,X_2^\e,X_3^\e, \W,\bar{N}\big)\to \big(\q,0,0,X_1,X_2,X_3, \W,\bar{N}\big),\text{ in } \Pi,\;\bar{\P}\text{-a.s.}
	\end{align}Set $	\vi{\Y}^\e(t)=\x +X_1^\e(t)+X_2^\e(t)+X_3^\e(t)+M_1^\e(t)+M_2^\e(t)$ and $	\vi{\Y}(t)=\x +X_1(t)+X_2(t)+X_3(t)$.  From the equation satisfied by the original processes, we may still assume that $\vi{\Y}$ also satisfies \eqref{CSPDE1}.

	Using that fact that if $f_n\in\D([0,T];\R)$ and $\lim\limits_{n\to\infty}f_n=0$ with the Skorokhod topology of $\D([0,T];\R)$, then $\lim\limits_{n\to\infty}\sup\limits_{t\in[0,T]}|f_n(t)|=0$. From  this fact and \eqref{582},  we have 
	\begin{align*}
		\lim_{\e\to0}\sup_{t\in[0,T]}\|M_2^\e(t)\|_{\V^*}=0,\ \;\bar{\P}\text{-a.s.},
	\end{align*}and 
	\begin{align*}
		\lim_{\e\to0}\sup_{t\in[0,T]}\|M_1^\e(t)\|_{\V^*}=0, \ \;\bar{\P}\text{-a.s.}
	\end{align*}
	Moreover, we deduce  
	\begin{align*}
		\lim_{\e\to0}\sup_{t\in[0,T]}\|X_1^\e(t)-X_1(t)\|_{\V^*}=0, \ \;\bar{\P}\text{-a.s.},\\
		\lim_{\e\to0}\sup_{t\in[0,T]}\|X_2^\e(t)-X_2(t)\|_{\V^*}=0,\  \;\bar{\P}\text{-a.s.},\\
		\lim_{\e\to0}\sup_{t\in[0,T]}\|X_3^\e(t)-X_3(t)\|_{\V^*}=0,\  \;\bar{\P}\text{-a.s.},
	\end{align*}and 
	\begin{align*}
		\lim_{\e\to0}\sup_{t\in[0,T]}\|\vi{\Y}^\e(t)-\vi{\Y}(t)\|_{\V^*}=0, \ \;\bar{\P}\text{-a.s.}
	\end{align*}Finally, using Proposition \ref{Prop.EX} and Lemma \ref{Cgslem6}, we can find the unique solution $\vi{\Y}(\cdot)$ of \eqref{Prop.EX1}, and therefore there exists a subsequence indexed by $\varpi$ such that 
	\begin{align*}
		\lim_{\varpi\to0}\sup_{t\in[0,T]}\|\vi{\Y}^{\varpi}(t)- \vi{\Y}(t)\|_{\H}=0, \ \;\bar{\P}\text{-a.s.},
	\end{align*}which completes the proof of this theorem.
\end{proof}
We established the proof of verification of Condition \ref{Cond1} (1) in Proposition \ref{VCond1} and (2) in Theorem \ref{VCond2}. Therefore, the proof of Theorem \ref{thrm5} is completed.

\subsection{Applications}\label{app} The LDP results obtained in this work are applicable to the stochastic versions of hydrodynamic models like (cf. \cite{HBEH,ICAM1,WL4,WLMR1,WLMR2,EM2,CPMR,MRSSTZ}, etc.) Burgers equations, 2D Navier-Stokes equations, 2D magneto-hydrodynamic equations, 2D Boussinesq model for the B\'enard convection, 2D Boussinesq system, 2D magnetic B\'enard equations, 3D Leray-$\alpha$-model, the Ladyzhenskaya model, some shell models of turbulence like GOY, Sabra, dyadic, etc.,  porous media equations,  $p$-Laplacian equations, fast-diffusion equations, power law fluids,  Allen-Cahn equations,    Kuramoto-Sivashinsky equations and 3D tamed Navier-Stokes equations. The paper \cite{MRSSTZ} provided a detailed framework of the models, which comes under the mathematical setting of this work,  like quasilinear SPDEs, convection-diffusion equations, Cahn-Hilliard equations, 2D Liquid crystal model, 2D Allen-Cahn-Navier-Stokes model, etc., and references therein.

\medskip\noindent
\textbf{Acknowledgments:} The first author would like to thank Ministry of Education, Government of India - MHRD for financial assistance. M. T. Mohan would  like to thank the Department of Science and Technology (DST), India for Innovation in Science Pursuit for Inspired Research (INSPIRE) Faculty Award (IFA17-MA110).


\begin{thebibliography}{99}
\bibitem{ADA1}	A. de Acosta, A general non-convex large deviation result with applications to stochastic equations, \emph{Prob. Theory Relat. Fields}, \textbf{118} (2000), 483--521.

\bibitem{ADA2}	A. de Acosta, Large deviations for vector valued L\'evy processes, \emph{Stochastic Process. Appl.}, \textbf{51} (1994), 75--115.
	
	\bibitem{SAZBJLW} S. Albeverio, Z. Brze\'zniak and J.-L. Wu, Existence of global solutions and invariant measures for stochastic differential equations driven by Poisson type noise with non-Lipschitz coefficients, \emph{J. Math. Anal. Appl.}, \textbf{371} (2010), 309--322.
	
	\bibitem{DA}  D. Aldous, Stopping times and tightness, \emph{Ann. Probability}, \textbf{6} (1978), 335--340.
	\bibitem{DAP} D. Applebaum, \emph{L\'evy Processes and Stochastic Calculus}, Cambridge Studies in Advanced Mathematics, Vol. {93}, Cambridge University press, 2004.
	
	
	\bibitem{VB1} V. Barbu, \emph{Nonlinear semigroups and differential equations in Banach	spaces}, Noordhoff, Leyden, 1976.
	
	
	
	\bibitem{VB2} 
	\newblock  V. Barbu, 
	\newblock  \emph{Nonlinear Differential Equations of Monotone Types	in Banach Spaces}, 
	\newblock  Springer, New York, 2010. 
	
	\bibitem{VBMR} V. Barbu and M. R\"ockner, An operatorial approach to stochastic partial differential equations driven by linear multiplicative noise, \emph{J. Eur. Math. Soc.}, \textbf{17} (2015), 1789--1815.
	
	\bibitem{HBAM}H. Bessaih and A. Millet, Large deviation principle and inviscid shell models, \emph{Electron. J. Probab.}, \textbf{14} (2009), 2551--2579.  
	
	\bibitem{HBEH}	H. Bessaih, E. Hausenblas and P. A. Razafimandimby, Strong solutions to stochastic hydrodynamical systems with multiplicative noise of jump type, 	\emph{NoDEA Nonlinear Differential Equations Appl.}, \textbf{22} (2015), 1661--1697.
%
	\bibitem{PB}	P. Billingsley, \emph{Convergence of probability measures}, Wiley Series in Probability and Statistics, Second edition, 1999.
%
\bibitem{MBPD}M. Boue and P. Dupuis, A variational representation for certain functionals of Brownian motion, \emph{Ann. Probab.}, \textbf{26} (1998), 1641--1659.

%
%
%
%
%
%
	\bibitem{ZBDG}Z. Brze\'zniak and D.  Gatarek, Martingale solutions and invariant measures for stochastic evolution equations in Banach spaces, 
	\emph{Stochastic Process. Appl.}, \textbf{84} (1999), 187--225. 
%
	\bibitem{ZBEH}Z. Brze\'zniak and E. Hausenblas, Maximal regularity for stochastic convolution driven by L\'evy processes, \emph{Probab. Theory Relat. Fields}, \textbf{145} (2009), 615--637.
%
%
%
%
%
%
	\bibitem{ZBWLJZ}	Z. Brze\'zniak, W. Liu and J. Zhu, Strong solutions for SPDE with locally monotone coefficients driven by L\'evy noise, \emph{Nonlinear Anal. Real World Appl.}, \textbf{17} (2014), 283--310.
%


	\bibitem{ZBXPJZ} Z. Brze\'zniak, X. Peng and  J. Zhai, Well-posedness and large deviations for 2-D stochastic Navier-Stokes equations with jumps, \emph{J. Eur. Math. Soc.}, (2022), \url{https://doi.org/10.4171/jems/1214}.
%
%
%

\bibitem{ABPD1}A. Budhiraja and P. Dupuis, A variational representation for positive functionals of infinite dimen- sional Brownian motion, \emph{Probab. and Math. Stat.}, \textbf{20} (2000), 39--61.

\bibitem{ABJCPD}A. Budhiraja, J. Chen and P. Dupuis, Large deviations for stochastic partial differential equations driven by a Poisson random measure, \emph{Stochastic Process. Appl.}, \textbf{123} (2013), 523--560.


\bibitem{ABPFD} A. Budhiraja and P. Dupuis, \emph{Analysis and Approximation of Rare Events: Representations and Weak Convergence Methods}, Springer, 2019.

\bibitem{ABPDVM}A. Budhiraja, P. Dupuis and V. Maroulas, Variational representations for continuous time processes, \emph{Ann. Inst. Henri Poincaré Probab. Stat.}, \textbf{47} (2011), 725--747.


	\bibitem{DLB} D. L. Burkholder, The best constant in the Davis inequality for the expectation of the martingale square function, \emph{Trans. Amer. Math. Soc.}, \textbf{354} (2002), 91--105. 
	
	\bibitem{SCMR}S. Cerrai and M.  R\"ockner,  Large deviations for stochastic reaction-diffusion systems with multiplicative noise and non-Lipschitz reaction term, \emph{Ann. Probab.}, \textbf{32} (2004), 1100--1139.

%
	\bibitem{PLC}P. L. Chow, Large deviation problem for some parabolic It\^o equations, \emph{Comm. Pure Appl. Math.}, \textbf{45} (1992), 97--120.
%
\bibitem{VVCMIV} V. V. Chepyzhov and M. I. Vishik, \emph{Attractors for Equations of Mathematical Physics}, American Mathematical Society, Providence, Rhode Island, 2002.
%
\bibitem{PCAM}P. Cherrier and A. Milani, \emph{Linear and Quasi-linear Evolution Equations in Hilbert-Spaces}, American Mathematical Society, 2012.
	\bibitem{ICAM1}I. Chueshov and A. Millet, Stochastic 2D hydrodynamical type systems: well posedness and large deviations, \emph{Appl. Math. Optim.}, \textbf{61} (2010), 379--420.
%
%
%
%
%
	\bibitem{DaZ}
	\newblock G. Da Prato and J. Zabczyk,
	\newblock \emph{Stochastic Equations in Infinite Dimensions 2nd ed.},
	\newblock Cambridge University Press, 2014.
	
		\bibitem{DZ} A. Dembo,  and  O. Zeitouni, 
	{\em Large Deviations Techniques and Applications}, Springer-Verlag,
	New York, 2000.
	
		\bibitem{DE} P. Dupuis and R. S.  Ellis, \emph{A Weak Convergence Approach to the Theory of Large Deviations}, Wiley-Interscience , New York, 1997.
	
	\bibitem{NH}N. Hirano, Nonlinear evolution equations with nonmonotonic perturbations, \emph{Nonlinear Anal.}, \textbf{13} (1989), 599--609.
	
%
%
\bibitem{ZDRZ} Z. Dong and R. Zhang, 3D tamed Navier-Stokes equations driven by multiplicative Lévy noise: existence, uniqueness and large deviations, \emph{J. Math. Anal. Appl.}, \textbf{492} (2020), 124404.

%
%
%
%
%
%
%
%

	 
\bibitem{IGDS} I. Gy\"ongy and D. $\check{\mathrm{S}}$i$\check{\mathrm{s}}$ka, It\^o formula for processes taking values in intersection of finitely many Banach spaces, \emph{Stoch PDE: Anal. Comp.},  \textbf{5} (2017), 428--455. 

%
%
%
%
%
%
%
%
%
%
%

\bibitem{WHSS}  W.  Hong, S.-S. Hu and W.  Liu, McKean-Vlasov SDEs and SPDEs with locally monotone coefficients,  \url{https://arxiv.org/pdf/2205.04043.pdf}. 

\bibitem{WHSL}  W.  Hong, S. Li and W.  Liu, Freidlin-Wentzell type large deviation principle for multiscale locally monotone SPDEs, \emph{SIAM J. Math. Anal.}, \textbf{53} (2021), 6517--6571. 


%
	\bibitem{NISW} N. Ikeda and S. Watanabe, \emph{Stochastic Differential Equations and Diffusion Processes}, North-Holland	Publishing Company, Amsterdam, 1981.
	
	\bibitem{JJANS}J. Jacod and A. N. Shiryaev, \emph{Limit Theorems for Stochastic Processes}, Springer-Verlag, 1987.
%
%
	\bibitem{KX} G. Kallianpur and  J. Xiong, \emph{Stochastic Differential Equations in Infinite Dimensional Spaces}, {Institute Math. Stat.}, 1996.  
	
\bibitem{KKMTM} K. Kinra and M. T. Mohan, 	Weak pullback mean random attractors for the stochastic convective Brinkman–Forchheimer equations and locally monotone stochastic partial differential equations,  \emph{Infin. Dimens. Anal. Quantum Probab. Relat. Top.},  \textbf{25}  (2022), 2250005.

%
	\bibitem{TKMR}T. Kosmala and M. Riedle, Variational solutions of stochastic partial differential equations with cylindrical L\'evy noise, \emph{Discrete Contin. Dyn. Syst. Ser. B}, \textbf{26} (2021), 2879--2898.


\bibitem{NVK}N. V. Krylov, \emph{On Kolmogorov’s equations for finite dimensional diffusion, in: Stochastic PDE’s and Kolmogorov’s Equations in Infinite Dimensions},
Cetraro, 1998, in: Lecture notes in Math., vol. 1715, Springer, Berlin, 1999, pp. 1–63.

 \bibitem{AKMTM4} A. Kumar and M. T. Mohan, Well-posedness of a class of stochastic partial differential equations with fully monotone coefficients perturbed by L\'evy noise, \emph{Submitted}, \url{https://arxiv.org/pdf/2209.06657.pdf}. 


%
%
%
\bibitem{SLWLYX}S. Li, W. Liu and Y. Xie, Small time asymptotics for SPDE with locally monotone coefficients, \emph{Discrete Contin. Dyn. Syst. Ser. B}, \textbf{25} (2020), 4801--4824.

\bibitem{WL} W. Liu, Well-posedness of stochastic partial differential equations with Lyapunov condition,  \emph{J.  Differential Equations}, \textbf{255} (2013),  572--592.
%

%
%
\bibitem{WL2}	W. Liu, Large deviations for stochastic evolution equations with small multiplicative noise, \emph{Appl. Math. Optim.}, \textbf{61} (2010), 27--56.
%
%
%
	\bibitem{WL4} W. Liu, Existence and uniqueness of solutions to nonlinear evolution equations with locally monotone operators, \emph{Nonlinear Anal.}, \textbf{75} (2011), 7543--7561.
%
%
%
%
	\bibitem{WLMR1}W. Liu and M. R\"ockner, SPDE in Hilbert space with locally monotone coefficients, \emph{J. Funct. Anal.}, \textbf{259} (2010), 2902--2922.
%
%
	\bibitem{WLMR2} W. Liu and M. R\"ockner, \emph{Stochastic Partial Differential Equations: An Introduction}, Springer, 2015.
%
%
\bibitem{WLMR3}W. Liu and M. R\"ockner, Local and global well-posedness of SPDE with generalized coercivity conditions, \emph{J. Differential Equations}, \textbf{254} (2013), 725--755. 
%
	\bibitem{WLRMJLDS1}W. Liu, M. R\"ockner and J. L. da Silva, Strong dissipativity of generalized time-fractional derivatives and quasi-linear (stochastic) partial differential equations, \emph{J. Funct. Anal.}, \textbf{281} (2021), 109135.
%
	\bibitem{WLRMJLDS2}W. Liu, M. R\"ockner and J. L. da Silva, Quasi-linear (stochastic) partial differential equations with time-fractional derivatives, \emph{SIAM J. Math. Anal.}, \textbf{50} (2018), 2588--2607.
	
	\bibitem{WLMRXSYX}W. Liu, M. R\"ockner, X. Sun, and Y. Xie, \emph{Strong averaging principle for slow-fast stochastic partial differential equations with locally monotone coefficients}, \url{https://arxiv.org/pdf/1907.03260.pdf}.
%
%
%
%
%
%

\bibitem{TMRZ} T. Ma and R. Zhu, Wong-Zakai approximation and support theorem for SPDEs with locally monotone coefficients, \emph{J. Math. Anal. Appl.}, \textbf{469} (2019), 623--660. 



\bibitem{UMMTM}U. Manna and M. T. Mohan, Large deviations for the shell model of turbulence perturbed by L\'evy noise, \emph{Commun. Stoch. Anal.}, \textbf{7} (2013), 39--63.

\bibitem{CMMR}C. Marinelli and M. R\"ockner, On the maximal inequalities of Burkholder, Davis and Gundy, 
\emph{Expo. Math.}, \textbf{34} (2016),  1--26.

%
%
%
	\bibitem{MM}M. M\'etivier, \emph{Stochastic Partial Differential Equations in Infinite-Dimensional Spaces},  Scuola Normale Superiore, 1988. 
%
	\bibitem{GJM1} G. J. Minty, Monotone (nonlinear) operators in Hilbert space, \emph{Duke Math. J.}, \textbf{29} (1962), 341--346.
%
%
%

	\bibitem{MTM2}M. T. Mohan, Martingale solutions of two and three dimensional stochastci convective Brinkman-Forchheimer equations forced by L\'evy noise, \url{https://arxiv.org/pdf/2109.05510.pdf}. 
%
	\bibitem{MTM3} M. T. Mohan, 	Large deviation principle for stochastic convective Brinkman-Forchheimer equations perturbed by pure jump noise,	\emph{J. Evol. Equ.}, \textbf{21} (2021), 4931--4971.
%
%
	\bibitem{MTM4} M. T. Mohan, 
	Well-posedness and asymptotic behavior of stochastic convective Brinkman-Forchheimer equations perturbed by pure jump noise,  \emph{Stoch. Partial Differ. Equ. Anal. Comput.}, \textbf{10} (2022), 614--690. 
%
%
%
	\bibitem{MTMSSS2} M. T. Mohan, and S. S. Sritharan, 
	Stochastic Navier-Stokes equations perturbed by L\'evy noise with hereditary viscosity,
	\emph{Infin. Dimens. Anal. Quantum Probab. Relat. Top.}, \textbf{22} (2019), 1950006, 32 pp.
%
%
%
%
	\bibitem{EM2}  	E. Motyl, 	Stochastic hydrodynamic-type evolution equations driven by L\'evy noise in 3D unbounded domains-abstract framework and applications, \emph{Stochastic Process. Appl.}, \textbf{124} (2014),  2052--2097.
%
%
%
%
	\bibitem{PNKTRT}P. Nguyen, K. Tawri and R. Temam, Nonlinear stochastic parabolic partial differential equations with a monotone operator of the Ladyzenskaya-Smagorinsky type, driven by a L\'evy noise, \emph{J. Funct. Anal.}, \textbf{281} (2021), 109157.
	
	\bibitem{BOAS} B. {\O}ksendal, and A. Sulem, \emph{Stochastic Control of jump diffusions}, Springer Berlin Heidelberg, 2005.
%
%
%
%
%
%
	\bibitem{CPMR}C. Pr\'ev\^ot and M. R\"ockner, \emph{A Concise Course on Stochastic Partial Differential Equations}, Lecture Notes in Mathematics,  Springer, 2007.
%
%
%
	\bibitem{MRSSTZ} M. R\"ockner, S. Shang and T. Zhang, Well-posedness of stochastic partial differential equations with fully local monotone coefficients, \url{https://arxiv.org/pdf/2206.01107.pdf}.
	
\bibitem{MRTZ}	M. R\"ockner and T. Zhang, Stochastic evolution equations of jump type: existence, uniqueness and
	large deviation principles, \emph{Potential Anal.}, \textbf{26} (2007), 255--279.
	
	\bibitem{WR}W. Rudin, \emph{Functional analysis},  2nd ed., International Series in Pure and Applied Mathematics, McGraw-Hill, Inc., New York 1991.
	
	

		\bibitem{SW} R. Sowers, Large Deviations for a reaction diffusion equation with non-Gaussian perturbations, \emph{Ann. Probab.} \textbf{20} (1992), 504--537.

\bibitem{NS}N. Shioji, Existence of periodic solutions for nonlinear evolution equations with pseudo monotone operators, \emph{Proc. Amer. Math. Soc.}, \textbf{125} (1997),
2921--2929.



	\bibitem{ST} D. Stroock, \emph{An Introduction to the theory of Large Deviations}, Springer-Verlog, Universitext, New York, 1984.
	
		\bibitem{ASJZ}A. Swiech, and J. Zabczyk, Large deviations for Stochastic PDE with L\'evy noise, \emph{J. Funct. Anal.}, \textbf{260} (2011), 674--723.
	

%
%
%
%
%
%
%
%
%
%
%
%
%
%
	\bibitem{ZTHWYW}Z. Tan, H. Wang and Y. Wang, Time-splitting methods to solve Hall-MHD systems with L\'evy noises, {\it Kinet. Relat. Models}, \textbf{12} (2019), 243--267.
%

\bibitem{RT}R. Temam, \emph{Navier-Stokes Equations: Theory and Numerical Analysis}, North-Holland, Amsterdam,
1984.

%
%
%
%
%
%
%
%
%

	\bibitem{VA}
S. R. S. Varadhan, 
\emph{Large deviations and Applications}, \textbf{46}, CBMS-NSF Series
in Applied Mathematics, SIAM, Philadelphia, (1984).


\bibitem{TXTZ} T. Xu and  T. Zhang, Large deviation principles for 2-D Stochastic Navier-Stokes equations driven by
Lévy processes, \emph{J. Funct. Anal.}, \textbf{257} (2009), 1519--1545.

	\bibitem{JXJZ}J. Xiong and J. Zhai,
Large deviations for locally monotone
stochastic partial differential equations driven
by L\'evy noise, \emph{Bernoulli}, \textbf{24} (2018), 2842--2874.

	\bibitem{XYJZTZ} X. Yang, J. Zhai, and T. Zhang, Large deviations for SPDEs of jump type, \emph{Stoch. Dyn.},
\textbf{15} (2015), 1550026 (30 pages).


	\bibitem{EZ} E. Zeidler, \emph{Nonlinear Functional Analysis and its Applications: II/B: Nonlinear monotone operators}, Springer-Verlag, 1990.

\bibitem{JZTZ} J. Zhai and T. Zhang, Large deviations for 2-D stochastic Navier-Stokes equations driven by multiplicative L\'evy noises, \emph{Bernoulli}, \textbf{21} (2015), 2351--2392.
%
%
%
%

	
	
\end{thebibliography}
\end{document}